\documentclass[journal,twoside,web]{ieeecolor}

\usepackage[a4paper, total={7.25in, 9.6in}]{geometry}

\usepackage{generic}
\usepackage{cite}
\usepackage{amsmath,amssymb,amsfonts}
\usepackage{graphicx}
\usepackage{textcomp}
\usepackage{stmaryrd}
\usepackage{tabularx}
\usepackage{multirow}
\usepackage{float}
\usepackage{mathtools}
\usepackage{comment}

\usepackage{color}

\newcommand{\tnf}[1]{\textnormal{#1}}
\newcommand{\tbf}[1]{\textbf{#1}}
\newcommand{\mbs}[1]{\boldsymbol{#1}}

\newcommand{\mcl}[1]{\mathcal{#1}}

\newcommand{\R}{\mathbb{R}}
\newcommand{\N}{\mathbb{N}}

\newcommand{\norm}[1]{\left\lVert{#1}\right\rVert}

\newcommand{\sign}[0]{\tnf{sgn}}

\newcommand{\bmat}[1]{\begin{bmatrix}#1\end{bmatrix}}

\newcommand{\smallbmat}[1]{\left[{\scriptsize\begin{smallmatrix}
		#1\end{smallmatrix} }\right]}

\floatstyle{ruled}
\newfloat{block}{thbp}{lop}
\floatname{block}{Block}

\newtheorem{thm}{Theorem}
\newtheorem{defn}[thm]{Definition}
\newtheorem{lem}[thm]{Lemma}
\newtheorem{prop}[thm]{Proposition}
\newtheorem{cor}[thm]{Corollary}

\floatstyle{ruled}
\newfloat{block}{thbp}{lop}
\floatname{block}{Block}

\let\bl\bigl
\let\bbl\Bigl

\let\br\bigr
\let\bbr\Bigr

\begin{document}

\title{\LARGE\textbf{Lyapunov Functions can Exactly Quantify Rate Performance of Nonlinear Differential Equations}}
\author{Declan S. Jagt, Matthew M. Peet, \IEEEmembership{Senior Member, IEEE}
	\vspace*{-0.305cm} 
	\thanks{Submitted to IEEE TAC on \today. This work was supported by the National Science Foundation grant CMMI-1931270}
	\thanks{D. S. Jagt and M. M. Peet are with the School for Engineering of Matter, Transport and Energy, Arizona State University, Tempe, AZ 85287 (e-mail: djagt@asu.edu and mpeet@asu.edu)
		%
	}
	\thanks{%
	}
}

\maketitle
\thispagestyle{plain}
\pagestyle{plain}

\begin{abstract}
	Pointwise-in-time stability notions for Ordinary Differential Equations (ODEs) establish bounds on the rate of decay of the system state -- allowing performance to be quantified by the maximum provable decay rate. 
	While Lyapunov tests have been proposed for numerous pointwise-in-time stability notions, including exponential, rational, and finite-time stability, it is unclear whether any of these tests can accurately characterize system performance.\\
	\indent In this paper, we start by proposing a generalized notion of rate performance --- with exponential, rational, and finite-time decay rates being special cases. Then, for any such notion and rate, we associate a Lyapunov condition which is shown to be necessary and sufficient for a system to achieve that rate. Finally, we show how the proposed conditions can be enforced using Sum-of-Squares (SOS) programming in the case of exponential, rational, and finite-time rate performance. Numerical examples in each case demonstrate that the corresponding SOS test can achieve tight bounds on the rate performance with accurate inner bounds on the associated regions of performance.
\end{abstract}





\section{INTRODUCTION}

Nonlinear Ordinary Differential Equations (ODEs) are a modeling tool for representing physical processes. The premise of this approach is that by studying the properties of solutions of the ODE, we may infer properties of the physical process.

One of the most basic properties that a physical system may possess is stability. Various notions of stability have been proposed to characterize solutions of the system, including Lyapunov stability, asymptotic stability, input-to-state stability, etc; each of which establishes some asymptotic behaviour of the solutions, and each of which can be verified by establishing the existence of a Lyapunov function with certain properties.

However, while Lyapunov theory is well-established for basic properties such as asymptotic stability~\cite{clarke1998smoothaConverseAsymptotic_LF,teel2014converseLF,wei2022converseLF}, class-$\mcl{KL}$ stability notions~\cite{kellet2015converseLFs_survey,teel2000smoothConverseKL_LF,kellett2023TwoMeasureISS} and input-to-state stability properties~\cite{dashkovskiy2011ISS,efimov2024ISS,sontag1999IOSnotions,sontag2000IOS,krichman2001IOSS}, the Lyapunov framework for evaluating \textit{performance} of nonlinear systems is less well-developed. To see this, consider the case of linear ODEs, for which we have a variety of computational tests which can precisely determine performance metrics such as quadratic costs, $H_2$ norm, and $H_\infty$ norm (see e.g.~\cite{boyd1994linear,scherer2000LMIs}). Such quantitative metrics for performance are significant in that they provide bounds on sensitivity to factors such as disturbances, modeling errors, and changes in state. For nonlinear systems, by contrast, there are no computational algorithms which are known to be capable of precisely determining any metric of performance for any subclass of substantially nonlinear systems.

The difficulty in establishing performance of nonlinear systems is threefold.
First, nonlinear systems admit several non-equivalent notions of stability, making it difficult to establish a universal metric for performance which can be quantified for any stable nonlinear system.
For example, consider the notions of exponential and rational stability. A system is exponentially stable if there exists some $M,k>0$ such that all solutions $x(t)$ satisfy $\norm{x(t)}\le M e^{-k t}\norm{x(0)}$, and rationally stable if there exists $M,k>0$ and $p\in \N$ such that $\norm{x(t)}^p\le M \frac{\norm{x(0)}^p}{1+\norm{x(0)}^p kt}$~\cite{bacciotti2005LF_book}. Then, rational stability does not imply exponential stability (consider e.g. $\dot{x}=-x^{3}$) nor does global exponential stability imply global rational stability (even for linear systems e.g. $\dot{x}=-x$). As such, although both exponential and rational stability performance can be quantified by the rate parameter $k$, the resulting metric of exponential rate performance may not be suitable for rationally stable systems, or vice versa.



The second difficulty with establishing a performance metric for nonlinear systems is that existing converse Lyapunov results do not consider system performance. For example, consider again the notions of rate performance for exponential and rational stability. Lyapunov conditions have been established for both exponential and rational stability~\cite{bacciotti2005LF_book}. However, although these Lyapunov conditions are guaranteed to establish \textit{some} lower bound on rate performance (e.g. a particular value of $k$), these bounds are not tight in that the Lyapunov conditions may be unable to verify the \textit{actual} rate performance.

Finally, the third difficulty with establishing performance of nonlinear systems is that, even if we have a necessary and sufficient Lyapunov condition for rate performance, numerically testing this Lyapunov condition may be computationally intractable.
While it is often possible to tighten Lyapunov conditions to Sum-Of-Squares (SOS) constraints -- allowing a polynomial Lyapunov function\footnote{exponential and rational stability of polynomial vector fields can be certified by a polynomial Lyapunov function~\cite{peet2009exponentialStability,peet2012converseSOS_LF,leth2017rational_stability}.} to be found using semidefinite programming~\cite{parrilo2000SOS,ahmadi2011converseSOS} -- establishing such an SOS formulation of the Lyapunov condition is not always trivial.
Specifically, the rational rate performance conditions we obtain are not linear in the Lyapunov function (a decision variable in SOS programming). Additionally, rates for finite-time stability require non-polynomial vector fields and Lyapunov functions. This raises the question of how such conditions might be enforced using SOS without sacrificing accuracy.

Perhaps the most significant progress toward the development of testable performance metrics can be found in~\cite{grune2002KLD_functions,grune2004KLD} under the unassuming name of ``Input to State Dynamic Stability''. Specifically, this work considers a class of ``$\mathcal{KLD}$'' comparison functions which are then associated with certain rates of convergence of solutions (along with bounds on the effect of disturbances). Remarkably, this work was able to establish that for any given rate, there exists a (possibly discontinuous) Lyapunov function capable of establishing this rate. This work was ahead of its time in that it predated much of the development of SOS-based algorithms which might have been used to find such Lyapunov functions and then been applied to specific performance metrics. Indeed, many of the results described in this paper parallel the developments found in~\cite{grune2002KLD_functions,grune2004KLD}, albeit here developed in a more constructive way and focused on the use of SOS algorithms to quantify specific metrics of performance -- including exponential, rational, and finite-time metrics of  performance.

The work of~\cite{grune2002KLD_functions,grune2004KLD} was particularly insightful in that a distinction was made between basic stability (showing that there \textit{exists} some metric of performance) and the evaluation of that performance metric. That is, while traditional Lyapunov conditions establish the existence of a performance metric, the result in~\cite{grune2002KLD_functions,grune2004KLD} provides Lyapunov conditions which can determine what that level of performance actually is (albeit in some region away from the origin).
The goal of this paper, then, is to build on the results from~\cite{grune2002KLD_functions,grune2004KLD} to formulate equivalent Lyapunov conditions for performance of nonlinear systems, including exponentially, rationally, and finite-time stable systems. Computationally, of course, we would also like to be able to find solutions to the proposed Lyapunov characterization using SOS programming.

Before establishing such equivalent characterizations, in Section~\ref{sec:beta_stability}, we first define performance metrics which admit an equivalent Lyapunov characterization. In particular, we focus on generalized notions of rate performance.
To define such rates, we use some of the language of $\mcl{KL}$ stability, wherein for a nonlinear system $\dot x=f(x)$, the solution map $\phi_f(x,t)$ satisfies $\phi_f(x,0)=x$ and $\partial_t\phi_f(x,t)=f(\phi_f(x,t))$ and where for class-$\mcl K$ functions $\alpha_1,\alpha_2$, stability implies existence of a class-$\mcl{KL}$ function $\beta$ such that $\alpha_1(\norm{\phi_f(x,t)})\le \beta(\alpha_2(\norm{x}),t)$~\cite{kellett2014KL}. For simplicity, however, we alter this slightly by fixing $\alpha$ and requiring $\alpha(\norm{\phi_f(x,t)})\le M\beta(\alpha(\norm{x}),kt)$ for some $M$, $k$.
Then, e.g., rational stability may be interpreted as $\mcl{KL}$ stability with additional restriction that $\alpha(y)=y^p$ and $\beta(y,t)=\frac{y}{1+yt}$. Furthermore, for any fixed $\alpha$ and $\beta$, we define rate performance  as the largest $k$ such that $\alpha(\norm{\phi_f(x,t)})\le M\beta(\alpha(\norm{x}),kt)$ for some $M$. However, for this notion of rate performance to be physically meaningful, we must add some additional restriction on the structure of $\beta$. 

First note that, apart from rate performance, an obvious notion of gain performance of $\phi_f$ is the smallest $M$ such that $\alpha(\norm{\phi_f(x,t)})\le M\alpha(\norm{x})$ for all $t\geq 0$.
However, for this gain performance to be consistent with the condition $\alpha(\norm{\phi_f(x,t)})\le M\beta(\alpha(\norm{x}),kt)$, we require that $\beta$ be normalized, so that $\beta(y,0)=y$. This ensures that all notions of stability have the same definition of gain performance. Second, we require time-invariance which ensures
the pointwise-in-time bound does not depend on the history of the state, e.g. if $\beta(y_1,t_{1})=\beta(y_2,t_{2})$, then also $\beta(y_1,t_{1}+t)=\beta(y_2,t_{2}+t)$ for all $t\geq 0$. This also ensures that the ratio $\rho(y):=\frac{\partial_t\beta(y,t)}{\partial_y \beta(y,t)}$ does not vary in time -- and since $\beta$ is normalized, we have $\partial_y \beta(y,0)=1$, so that $\rho(y)=\partial_t\beta(y,t)\vert_{t=0}$ (this implies that time-invariance is also equivalent to the $\mcl{KLD}$ property in~\cite{grune2002KLD_functions}). These conditions are not particularly restrictive and hold for the definitions of exponential, rational, and finite-time stability with: $\beta_{\tnf{e}}(y,t)=ye^{-t}$; $\beta_{\tnf{r}}(y,t)=\frac{y}{1+yt}$ and $\beta_{\tnf{f}}(y,t)=y-t$, respectively. In each case, there is a clear notion of rate of convergence where, for given $\alpha$, a system is exponentially, rationally, or finite-time stable with rate $k$ if it is $\beta$-stable with $\beta(y,t):=\beta_{\tnf{e},\tnf{r},\tnf{f}}(y,kt)$.


Another advantage of normalized, time-invariant $\beta$-metrics is that,
 as will be shown in Subsection~\ref{subsec:LF_characterization_sufficiency}, any such metric may be equivalently characterized using comparison functions. Specifically, any normalized, time-invariant function $\beta$ is generated by a scalar differential equation $\dot y=\rho(y)$ where $\rho(y):=\partial_t \beta(y,t)\vert_{t=0}$. Then for exponential, rational, and finite-time stability with rate $k$, we have $\rho_{\tnf{e}}(y)=-ky$, $\rho_{\tnf{r}}(y)=-ky^2$, and $\rho_{\tnf{f}}(y)=-k$, respectively. This characterization of $\beta$-metrics allows for the ready construction of Lyapunov characterizations as $\dot V\leq k\rho(V)$.

To establish that the proposed class of Lyapunov conditions may be used to precisely characterize performance rates (and gains), in Subsection~\ref{subsec:LF_characterization:necessity} we construct a Kurzweil-Yoshizawa type converse Lyapunov function as $V(x)=\inf_{y\in\mcl{W}(x)}y$ where $\mcl{W}(x):=\{y\in\R_{+}\mid \alpha(\norm{\phi_{f}(x,t)})\leq M\beta(\alpha(\norm{x}),kt),~\forall t\geq 0\}$ (similar to the Lyapunov function from~\cite{grune2002KLD_functions,grune2004KLD}) to show that (for given $\alpha$) the system solution satisfies $\alpha(\norm{\phi_f(x,t)})\le M\beta (\alpha(\norm{x}),kt)$ if and only if there exists a Lyapunov function $V$ satisfying $\alpha(\norm{x})\leq MV(x)\leq M\alpha(\norm{x})$ and $V(\phi_{f}(x,t))\leq \beta(V(x),kt)$ where if $V$ is differentiable, $\dot V(x)\le k\rho(V(x))$. Conditions for differentiability are provided in Thms.~\ref{thm:LF_necessity_diff_FT} and~\ref{thm:LF_necessity_diff}. For exponential, rational, and finite-time stability with rate $k$, we have the equivalent Lyapunov characterizations: $\dot V(x)\le -k V(x)$; $\dot V(x)\le -k V(x)^2$ and $\dot V(x)\le -k $, respectively. 

Having defined rate (and gain) performance, and having equivalently characterized this performance using Lyapunov conditions, in Section~\ref{sec:SOS_tests},
we tighten these conditions to SOS constraints.
For exponential rate performance, the SOS constraints are relatively easy: we require $V$ and $(-\nabla V(x)^Tf(x)-k V(x))$ to be SOS and use a standard Positivstellensatz extension for rates on a semialgebraic domain. For rational and finite-time stability, however, we must slightly manipulate the Lyapunov conditions in order to test them using SOS.

For rational stability, we observe that the condition $\dot V(x)\le-kV(x)^2$ is nonlinear in decision variable $V$. However, since $0\leq V(x)\le \alpha(\norm{x})$, we may instead enforce $\dot V(x)\le -kV(x)\alpha(\norm{x})\le-kV(x)^2$. While this may introduce conservatism, we prove that the obtained rate performance will deviate from the true rate performance by a factor of at most $\frac{1}{M}$, for gain performance $M$. This is a substantial improvement over classical Lyapunov conditions, $\dot{V}(x)\leq -k\alpha(\norm{x})^2$, for which the rate performance may deviate by a factor of $\frac{1}{M^2}$.

For finite-time stability, the polynomial bounds on $\dot{V}$ are problematic since the vector field $f$ will not be polynomial. In this case, however, we may make a change of variables $\tilde{V}(z)=V(z(x))$ and $\tilde{f}(z)=f(z(x))|z(x)|^{1-q}$ for $z(x)=\text{sign}(x)|x|^{q}$ such that $\dot{\tilde{V}}=\nabla\tilde{V}^T \tilde{f}$ is polynomial for $\tilde{V}$ polynomial.

For exponential, rational, and finite-time stability, extensive numerical testing is performed in Section~\ref{sec:examples} to evaluate the accuracy of the proposed rate and gain performance characterization of several nonlinear systems. The bounds on rate performance are compared with numerical simulation-based estimates to verify accuracy.
For the case of rational stability, the bounds are also compared to lower bounds on rate performance obtained with standard Lyapunov conditions for rational stability, verifying that the proposed Lyapunov conditions substantially reduce conservatism, and adhere to the conjectured $\frac{1}{M}$ scaling rule.
Finally, we examine the effect of the domain size on degradation of the performance measure.

\section{Notation}%
Let $\R_{+}:=[0,\infty)$.
We say that $\alpha:\R\to\R_{+}$ is monotonically decreasing or monotonically nonincreasing if $y_{1}>y_{2}$ implies $\alpha(y_{1})<\alpha(y_{2})$ or $\alpha(y_{1})\leq\alpha(y_{2})$, respectively. We say that $\alpha:\R^{n}\to\R_{+}$ is positive definite if $\alpha(x)=0$ if and only if $x=0$, and coercive if $\norm{x}_{2}\to\infty$ implies $\alpha(x)\to\infty$.
For $p>0$, denote by $\|.\|_{p}$ the $p$-(quasi)norm on $\R^{n}$.
For a vector field $f:\Omega\to\R^{n}$ on some domain $\Omega\subseteq\R^{n}$, let $\phi_{f}:\Omega\times\R\to\R^{n}$ denote the solution map to the corresponding ODE, so that 
\begin{equation*}
	\partial_{t}\phi_{f}(x,t)=f(\phi_{f}(x,t)),\quad \phi_{f}(x,0)=x,\quad \forall x\in \Omega,\,t\in\R_{+}.
\end{equation*}
Throughout the paper, we assume the solution map $\phi_{f}$ to be well-defined, unique, and jointly continuous in $x$ and $t$. Such well-posedness can be guaranteed using relatively mild smoothness conditions on $f$, see 
e.g. Cors.~3 and 4  in Sec.~7 of~\cite{arnold1992ODEs}.
An exception is made for finite-time stable systems where uniqueness at the origin is relaxed using the weakened conditions provided in~\cite{bhat2000FiniteTimeStability}, Section 2.
Under these conditions, the solution map $\phi_{f}$ satisfies the semigroup property
\begin{equation*}
	\phi_{f}(\phi_{f}(x,t),s)=\phi_{f}(x,t+s),\qquad \forall t,s\in\R_{+}.
\end{equation*}
We say that $G\subseteq\Omega$ is forward invariant for $f$ if $\phi_{f}(x,t)\in G$ for all $x\in G$ and $t\geq 0$.


\section{Characterizing Pointwise-in-Time Stability}\label{sec:beta_stability}

The main technical result of this paper is a converse Lyapunov theorem which provides an equivalent Lyapunov characterization of any suitably well-posed notion of pointwise-in-time rate performance. To that end, our first goal is to define what constitutes a well-posed metric of pointwise-in-time rate performance. In this section, we suppose that we are presented with a time-invariant vector field, $f$, with associated solution map, $\phi_f:\Omega \times \R_{+} \rightarrow \Omega$ where $\Omega \subset \R^n$ is forward invariant. A pointwise-in-time stability criterion is one which bounds some measure, $\alpha_1:\Omega\rightarrow \R_{+}$, of the state of the system, $\phi_{f}(x,t)$, in terms of some measure, $\alpha_2:\Omega\rightarrow \R_{+}$, of the size of the initial condition, $x$ --- and furthermore, that this bound (defined by $\beta:\R_{+} \times \R_{+} \to \R_{+}$) holds at every point in time, so that
\begin{equation}\label{eq:PTS_bound}
	\alpha_1(\phi_{f}(x,t))\leq \beta(\alpha_2(x),t),\quad \forall x\in \Omega,~t\geq 0.
\end{equation}
This definition is essentially the same as the concept of stability with respect to 2-measures~\cite{kellett2023TwoMeasureISS}. The only difference is that we have not yet imposed the constraint that $\beta$ be class $\mcl{KL}$, although this will be imposed implicitly.

Having separated the measures of state, $\alpha_1,\alpha_2$, from the pointwise bound, $\beta$, we note that for any given notion of pointwise-in-time stability these measures and bounds are not uniquely defined. Indeed, as discussed in~\cite{grune1999asymptotic_exponential,jongeneel2024asymptotic_exponential}, even for fixed $\beta$, exponential and asymptotic stability are equivalent modulo a change in $\alpha$. Furthermore, even with class $\mcl K, \mcl{KL}$ restrictions, there is significant ambiguity in $\alpha_{i},\beta$. To illustrate, for the notion of exponential stability, $\norm{\phi_{f}(x,t)}_2 \le Me^{-k t} \norm{x}_2$, the pointwise-in-time bound can be equivalently characterized by several definitions of $\alpha_1,\alpha_2,\beta$, including e.g. $\alpha_1(x)=\norm{x}_2$, $\alpha_2(x)=M\norm{x}_2$ and $\beta(y,t)=ye^{-k t}$, or, $\alpha_1(x)=\norm{x}^4_2$, $\alpha_2(x)=\norm{x}^4_2$ and $\beta(y,t)=M^4 ye^{-4k t}$.
While we do not claim to propose any method for uniquely defining $\alpha_{i},\beta$, we do claim that this ambiguity allows some flexibility to impose certain structure on $\alpha_{i}$ and $\beta$ which can then be used to define quantifiable notions of rate and gain performance.

To start, we will assume that $\alpha_{1}(x)=\frac{1}{M}\alpha_{2}(x)=\alpha(x)$ for some $M\geq 1$. This will allow us to define a notion of gain performance for a given choice of  $\alpha$ as the smallest value of $M\geq 1$ such that $\alpha(\phi_{f}(x,t))\leq M\alpha(x)$ for all $t\geq 0$. Note that this also implies that gain performance is not explicitly a function of the domain size, and allows for gain performance measures which hold globally (e.g. when the system is linear). Now, because we also want this definition of gain performance to be consistent with $\beta$, we further impose a normalization condition that $\beta(y,0)=y$ -- indicating that $\beta$ is not, itself, a measure of initial condition or state but rather bounds the decay rate of the measures of state defined by $\alpha$.
Finally, we also presume that the notion of stability is time-invariant or memoryless -- implying that $\beta$ is strictly a measure of the state and does not depend on the history of the state.
This definition is formally stated as follows.

\begin{defn}\label{defn:time-invariant}
	We say that $\beta:\R_{+} \times \R_{+}\to\R_{+}$ is \tbf{normalized} if $\beta(y,0)=y$, and \tbf{time-invariant} if $\beta$ is continuous and for all $y,z,t_{0}\geq0$, $\beta(y,t_{0})=\beta(z,0)$ implies $\beta(y,t_{0}+t)=\beta(z,t)$ for all $t\geq 0$.
\end{defn}

\begin{figure}[t]
	\centering
	\includegraphics[width=1.0\linewidth]{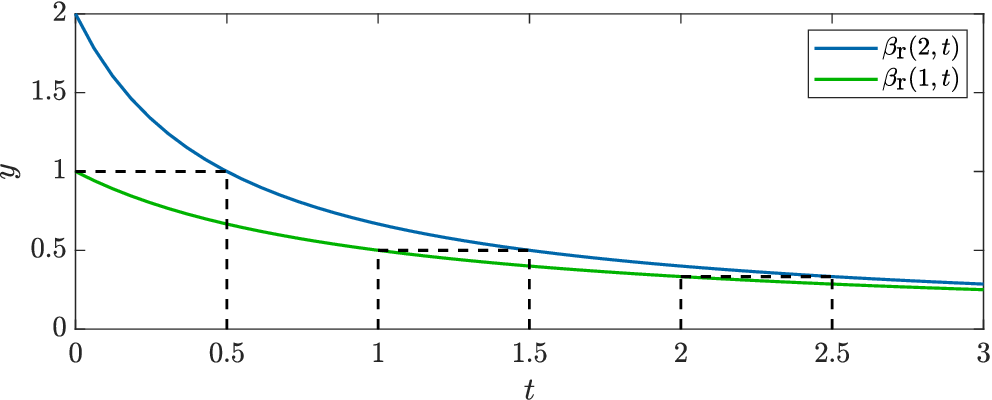}
	\vspace*{-0.5cm}
	\caption{Rational stability bound $\beta_{\tnf{r}}(y,t):=\frac{y}{1+yt}$ for $y=2$ and $y=1$. Since $\beta_{\tnf{r}}(2,0.5)=\beta_{\tnf{r}}(1,0)$, time-invariance of $\beta_{\tnf{r}}$ implies that also $\beta_{\tnf{r}}(2,t+0.5)=\beta_{\tnf{r}}(1,t)$ for all $t\geq 0$.}
	\label{fig:rational_stability}	
	\vspace*{-0.3cm}
\end{figure}

The time-invariance property is illustrated in Fig.~\ref{fig:rational_stability} for the function $\beta_{\tnf{r}}(y,t):=\frac{y}{1+yt}$, which we will use to characterize rational stability in Subsection~\ref{subsec:beta_stability:examples}.
A normalized, time-invariant function, $\beta$, can now be used to define the notion of $\beta$-stability as follows.

\begin{defn}\label{defn:betastability}
For given $\alpha:\R^{n}\to\R_{+}$ and normalized, time-invariant $\beta:\R_{+}\times\R_{+}\to\R_{+}$, we say that $f:\Omega \rightarrow \R^n$ (with associated solution map $\phi_f$) is $\boldsymbol{\beta}$\tbf{-stable} on $G \subseteq \Omega$ with respect to $\alpha$ and with rate $k\geq 0$ and gain $M\geq 1$ if\\[-0.75\baselineskip]
\begin{equation*}
	\alpha(\phi_{f}(x,t))\leq M\beta(\alpha(x),kt),\qquad \forall x\in G,~t\geq 0.
\end{equation*}
\end{defn}
\vspace*{2mm}

Naturally, the notion of $\beta$-stability motivated here is strongly related to traditional notions of $\mcl{KL}$ stability, stability with respect to 2 measures, and $\mcl{KLD}$ stability. The primary distinction, of course, is that for given $\alpha_1,\alpha_2$, $\mcl{KL}$ stability requires $\alpha_{1}(\phi_{f}(x,t))\leq\beta(\alpha_{2}(x),t)$ for \textit{some} $\beta$, whereas Defn.~\ref{defn:betastability} requires this bound to hold for \textit{given} $\beta$. As a result, in the $\mcl{KL}$ framework, there is no clear distinction between different notions of stability, in that strategic choices of the measures $\alpha_{i}$ (potentially not differentiable at $x=0$) allow for equivalence between e.g. rational and exponential stability (by~\cite{sontag1998IISS}, Prop.~7).
The proposed notion of $\beta$-stability avoids such ambiguity, by requiring both the measures $\alpha_{i}$ and the function $\beta$ to be fixed (with $\alpha_{1}=\frac{1}{M}\alpha_{2}$).
This distinction also allows us to define rate and gain performance of the system.

The second distinction between $\beta$-stability and $\mcl{KL}$ stability lies in the characterization of the $\beta$ function as normalized and time-invariant. In particular, note that a class-$\mcl{KL}$ function may not be either normalized (e.g. $\beta(y,t)=2ye^{-t}$) or time-invariant (e.g. $\beta(y,t)=\frac{y}{1+t}$). Conversely, however, any normalized $\beta$ is of class $\mcl{K}$ at $t=0$, since $\beta(y,0)=y$ is monotonically increasing and $\beta(0,0)=0$. If $\beta$ is also time-invariant, this then implies that $\beta(y,t)$ is nondecreasing in $y$ for all $t\ge 0$, as seen in the following lemma.

\begin{lem}\label{lem:radially_increasing}
	If $\beta$ is normalized and time-invariant, then $\beta(y,t)$ is monotonically nondecreasing in $y$ for all $t\geq 0$.
\end{lem}
\begin{proof}
	Suppose for contradiction that there exists some $t\geq 0$ and $y_{1}< y_{2}$ such that $\beta(y_{1},t)>\beta(y_{2},t)$. Since $\beta$ is normalized, $\beta(y_{1},0)=y_{1}<y_{2}=\beta(y_{2},0)$, and hence (by continuity of $\beta$) there exists some $t^*\in(0,t)$ such that $\beta(y_{1},t^*)=\beta(y_{2},t^*)$. Since $\beta$ is time-invariant,  $\beta(y_{1},t^*+s)=\beta(y_{2},t^*+s)$ for all $s\geq 0$. Choosing $s=t-t^*>0$, this implies $\beta(y_{1},t)=\beta(y_{2},t)$, posing a contradiction.  
\end{proof}

Given the proposed notion of $\beta$-stability, in the following subsection, we show how several classical and quantifiable stability notions can be defined by normalized and time-invariant functions $\beta$.

\vspace*{-0.2cm}
\subsection{Illustrations of $\beta$-Stability}\label{subsec:beta_stability:examples}

In this subsection, we examine several basic pointwise-in-time stability notions and show how the $\beta$-stability framework is used to define associated metrics of performance. In particular, we consider exponential, rational, and finite-time stability.

\subsubsection{Exponential Stability}

Perhaps the most well-studied quantifiable stability notion is that of exponential stability.
\begin{defn}\label{defn:exponential_stability}
	For given $G\subseteq\Omega$, we say that the vector field $f\in\Omega\to\R^{n}$ is \tbf{exponentially stable} on $G$ with rate $k\geq 0$ and gain $M\geq 1$ if\\[-0.95\baselineskip]
	\begin{equation*}
		\norm{\phi_{f}(x,t)}_{2}\leq Me^{-kt}\norm{x}_{2},\qquad \forall x\in G,~t\geq 0.
	\end{equation*}
\end{defn}
\vspace*{2mm}

This definition of exponential stability can be readily expressed in the $\beta$-stability framework as follows.

\begin{lem}\label{lem:exponential_stability}
Let $\alpha(x):=\norm{x}_{2}$ and $\beta_{\tnf{e}}(y,t):=ye^{-t}$. Then $\beta_{\tnf{e}}$ is normalized and time-invariant, and $f$ is exponentially stable on $G$ with rate $k$ and gain $M$ if and only if it is $\beta_{\tnf{e}}$-stable on $G$ with respect to $\alpha$ and with rate $k$ and gain $M$.
\end{lem}
\begin{proof}
First, $\beta_{\tnf{e}}$ is normalized since $\beta_{\tnf{e}}(y,0)=y$, and time-invariant since $\beta_{\tnf{e}}(y,t_{0})=ye^{-t_0}=z$ implies
\begin{equation*}
	\beta_{\tnf{e}}(y,t_{0}+t)=ye^{-(t_{0}+t)}=ye^{-t_{0}}e^{-t}=ze^{-t}=\beta_{\tnf{e}}(z,t).
\end{equation*}
Since $\alpha(\phi_{f}(x,t))=\norm{\phi_{f}(x,t)}_{2}$ and $ M\beta_{\tnf{e}}(\alpha(x),kt)=Me^{-kt}\norm{x}_{2}$, equivalence of exponential stability and $\beta_{\tnf{e}}$-stability follows immediately.
\end{proof}
\vspace*{2mm}

\subsubsection{Rational Stability}

Exponential stability is relatively uncommon for systems which are substantially nonlinear (e.g. $\dot x=-x^3$). A more common property of nonlinear systems is that of rational stability (also called weakly intensive behaviour in~\cite{hahn1963LFs}), which can be formulated as follows.

\begin{defn}\label{defn:rational_stability}
	Given $G\subseteq\Omega$, and $p\in \N$, we say that the vector field $f:\Omega\to\R^{n}$ is \tbf{rationally stable} on $G$ with rate $k\ge 0$ and gain $M\ge 1$ if\\[-0.75\baselineskip]
	\begin{equation*}
		\norm{\phi_{f}(x,t)}_{2}^{p}\leq M\frac{\norm{x}_{2}^{p}}{1+\norm{x}_{2}^{p}kt} ,\quad \forall x\in G,~t\geq 0.
	\end{equation*}
\end{defn}
\vspace*{2mm}

Rational stability can be associated with performance metrics using the $\beta$-stability framework as follows.

\begin{lem}\label{lem:rational_stability}
	Let $p\in\N$, $\alpha(x):=\norm{x}_{2}^{p}$, and $\beta_{\tnf{r}}(y,t):=\frac{y}{1+yt}$.
	Then $\beta_{\tnf{r}}$ is normalized and time-invariant, and $f$ is rationally stable on $G$ with rate $k$ and gain $M$ if and only if it is $\beta_{\tnf{r}}$-stable on $G$ with respect to $\alpha$ and with rate $k$ and gain $M$.
\end{lem}
\begin{proof}
	First, $\beta_{\tnf{r}}$ is normalized, since $\beta_{\tnf{r}}(y,0)=y$. Furthermore, for all $y>0$ and $t_{0}\geq 0$, we have $\beta_{\tnf{r}}(y,t_{0})=(y^{-1}+t_{0})^{-1}$. Therefore, if $y>0$, $\beta_{\tnf{r}}(y,t_{0})=z$ implies that
	\begin{align*}
		\beta_{\tnf{r}}(y,t_{0}+t)&=(y^{-1}+[t+t_{0}])^{-1}	\\
		&
		=(\beta_{\tnf{r}}(y,t_{0})^{-1}+t)^{-1}=\beta_{\tnf{r}}(z,t),\quad \forall t\geq 0.
	\end{align*}
	Alternatively, if $y=0$, then $z=\beta_{\tnf{r}}(y,t_{0})=0$ implies $\beta_{\tnf{r}}(y,t_{0}+t)=0=\beta_{\tnf{r}}(z,t)$ for all $t\geq 0$.
	It follows that $\beta_{\tnf{r}}$ is time-invariant.
	Finally, equivalence holds since $\alpha(\phi_{f}(x,t))=\norm{\phi_{f}(x,t)}_{2}^{p}$ and $M\beta_{\tnf{r}}(\alpha(x),kt)=M\norm{x}_{2}^{p}/(1+\norm{x}_{2}^{p}kt)$.
\end{proof}
\vspace*{2mm}

\subsubsection{Finite-Time Stability}

Unlike exponential and rational stability (where convergence is asymptotic), finite-time stable systems reach an equilibrium in finite time. Although finite-time stability is typically characterized by a ``settling time function'', $T(x)$, where $\lim_{t\rightarrow T(x)}\phi_f(x,t)=0$, we suppose that performance of finite-time stable systems may also be characterized by a growth bound on $T(x)$ as $T(x)\le \frac{1}{k}\norm{x}^\eta_p$. This leads to the following definition.

\begin{defn}\label{defn:finite_time_stability}
	For given $G\subseteq\Omega$ and $p,\eta>0$, we say that the vector field $f:\Omega\to\R^{n}$ is \tbf{finite-time stable} on $G$ with rate $k\ge 0$ and gain $M\ge 1$ if\\[-0.75\baselineskip]
	\begin{equation*}
		\norm{\phi_{f}(x,t)}_{p}^{\eta}\leq \begin{cases}
			M \bl(\norm{x}_{p}^{\eta}-kt\br),	&	t\leq \frac{1}{k}\norm{x}_{p}^{\eta},\\
			0,		&	\tnf{else},
		\end{cases},\quad \forall x\in G.
	\end{equation*}
\end{defn}
\vspace*{2mm}

Finite-time stability and performance metrics are now expressed using the $\beta$-stability framework as follows.

\begin{lem}\label{lem:finite_time_stability}
	For given $p,\eta>0$, let $\alpha(x):=\norm{x}_{p}^{\eta}$ and $\beta_{\tnf{f}}(y,t):=\max\{y-t,0\}$.
	Then $\beta_{\tnf{f}}$ is normalized and time-invariant, and the vector field $f$ is finite-time stable on $G$ with rate $k$ and gain $M$ if and only if it is $\beta_{\tnf{f}}$-stable on $G$ with respect to $\alpha$ and with rate $k$ and gain $M$.
\end{lem}
\begin{proof}
	First, $\beta_{\tnf{f}}$ is normalized since $\beta_{\tnf{f}}(y,0)=y$. For time-invariance, let $y,t_{0}\geq 0$ and  $z=\beta_{\tnf{f}}(y,t_{0})$. In the case $t_{0}\geq y$, we have $z=0$, and hence $\beta_{\tnf{f}}(y,t_{0}+t)=0=\beta_{\tnf{f}}(z,t)$ for all $t\geq 0$. Otherwise, if $t_{0}<y$, then for $t\in[0,z)$
	\begin{equation*}
		\beta_{\tnf{f}}(y,t_{0}+t)
		=y-[t+t_{0}]
		=\beta_{\tnf{f}}(y,t_{0})-t	
		=\beta_{\tnf{f}}(z,t),
	\end{equation*}
	and for $t\geq z$, we again have $\beta_{\tnf{f}}(y,t_{0}+t)=0=\beta_{\tnf{f}}(z,t)$.
	Thus, $\beta_{\tnf{f}}$ is time-invariant.
	Finally, equivalence of finite-time stability and $\beta_{\tnf{f}}$-stability follows since $\alpha(\phi_{f}(x,t))=\norm{\phi_{f}(x,t)}_{p}^{\eta}$ and $M\beta_{\tnf{f}}(\alpha(x),kt)=M(\norm{x}_{p}^{\eta}-kt)$.
\end{proof}

These examples illustrate how many classical stability notions can be associated with performance metrics and expressed using the $\beta$-stability framework.
Moreover, each of these stability notions is parameterized by a rate, $k$, offering a way of quantifying rate performance in terms of largest rate $k$ for which $\beta$-stability holds.

\subsection{Notions of Rate and Gain Performance}

Having motivated, defined, and illustrated the $\beta$-stability framework, let us now formally define the associated performance metrics and discuss how one might accurately determine such metrics of performance. Specifically, the two obvious metrics are rate and gain performance.

First we consider gain performance, which is rather trivially defined as the maximum overshoot/amplification with respect to a given choice of $\alpha$:
$M^*:=\sup_{x\in G}\sup_{t\in[0,\infty)}\frac{\alpha(\phi_{f}(x,t))}{\alpha(x)}$.
Note that the gain performance does not depend on $\beta$ and may be equivalently defined for any normalized/time-invariant $\beta$ as the smallest gain such that $f$ is $\beta$-stable on $G$ with rate $k=0$. Indeed, recall that the motivation for normalization and time-invariance of $\beta$ was to ensure that gain performance does not depend on the choice of $\beta$. Beyond motivating the normalization and time-invariance properties, however, gain performance is not particularly new or interesting -- corresponding to Lipschitz continuity of $\phi_f(x,\cdot)$ with respect to $\norm{\cdot}_{L_\infty}$ at $x=0$ and thus implying stability in the sense of Lyapunov (uniform continuity of $\phi_f$).

Next, let us consider the notion of rate performance. Since nonlinear systems exhibit distinct convergence behaviours, for any fixed choice of $\alpha$, a system may be rationally stable, but not exponentially stable or finite-time stable.
This complicates the task of defining a consistent notion of rate performance which is valid for any nonlinear ODE. The $\beta$-stability framework simplifies this problem by providing a universal definition of rate performance. Specifically, we have the following.
\begin{defn}
	For given $\alpha:\R^{n}\to\R_{+}$ and normalized, time-invariant $\beta$, we define the \tbf{rate performance} of $f:\Omega\to\R^{n}$ with respect to $\alpha$ and $\beta$ on $G\subseteq\Omega$ as the largest value of $k\geq 0$ for which there exists $M\geq 1$ such that $f$ is $\beta$-stable on $G$ with respect to $\alpha$ and with rate $k$ and gain $M$.
\end{defn}

We note that, while not all nonlinear systems will admit a rate performance metric for a \textit{given} $\beta$, every asymptotically stable such system will admit a rate performance metric for \textit{some} $\beta$, $\alpha$, as established in~\cite{grune1999asymptotic_exponential} (using $\beta=\beta_{\tnf{e}}$).

Having defined a suitable notion of rate performance, we now face the main challenge of testing the rate performance for a given vector field and choice of $\beta$. To resolve this challenge, in the following section, we propose an equivalent Lyapunov characterization of $\beta$-stability with respect to any given measure, so that $\beta$-stability implies existence of a Lyapunov function certifying this notion of stability. Using this result, the challenge of computing the rate performance may then be equivalently formulated as that of finding a suitable Lyapunov function, which can be posed as an optimization problem, as examined in Section~\ref{sec:SOS_tests}.


\section{A Necessary and Sufficient Lyapunov Function Characterization of $\beta$-Stability}\label{sec:LF_characterization}

Given the notion of $\beta$-stability, in this section we show that for any vector field $f$, $\beta$-stability of $f$ with rate $k$ on some forward invariant set $G$ is equivalent to the existence of a Lyapunov function, $V:G\to\R_{+}$, such that for all $x\in G$, $M^{-1}\alpha(x)\leq V(x)\leq \alpha(x)$ and
(if $V$ is differentiable)
\begin{equation}\label{eq:LF_diff_negativity}
	\nabla V(x)^T f(x)\!\leq\! k\rho(V(x)),\enspace\text{where}~ \rho(y)\!:=\! \partial_t\beta(y,t)\vert_{t=0}.
\end{equation}
To establish this result, Subsection~\ref{subsec:LF_characterization_sufficiency} first addresses sufficiency via Thm.~\ref{thm:LF_sufficiency}, Lem.~\ref{lem:comparison} and Thm.~\ref{thm:LF_sufficiency_diff}. Then, in Subsection~\ref{subsec:LF_characterization:necessity}, we establish the associated necessary counterparts to these results in Thm.~\ref{thm:LF_necessity}, Lem.~\ref{lem:comparison_N} and Thm.~\ref{thm:LF_necessity_diff_FT}.

\subsection{A Sufficient Lyapunov Condition for $\beta$-Stability}\label{subsec:LF_characterization_sufficiency}

In this subsection, we derive a sufficient condition for $\beta$-stability as feasibility of the Lyapunov inequality in~\eqref{eq:LF_diff_negativity}. Specifically, Thm.~\ref{thm:LF_sufficiency} provides a sufficient condition for $\beta$-stability in terms of existence of a $V$ for which $\beta$-stability holds with $\alpha=V$. Lemma~\ref{lem:comparison} then recalls the comparison principle, with application to Thm.~\ref{thm:LF_sufficiency}. Finally, Thm.~\ref{thm:LF_sufficiency_diff} provides sufficient Lyapunov conditions for $\beta$-stability.

In this section, we consider a slight generalization of the notion of $\beta$-stability by allowing for 2 measures, $\alpha_{1},\alpha_{2}:\R^{n}\to \R$. Furthermore, to simplify the presentation, we subsume the rate parameter, $k$, into the $\beta$ function as $\beta(y,kt)\mapsto\beta(y,t)$. Specifically, we define the following notion of $\beta$-stability with respect to two measures.
\begin{defn}
	Given $\alpha_{1},\alpha_{2}:\R^{n}\to\R_{+}$ and normalized, time-invariant $\beta$, we say that $f:\Omega\to\R^{n}$ is $\mbs{\beta}$\tbf{-stable} \tbf{with respect to two measures}, $\alpha_{1},\alpha_{2}$, on domain $G\subseteq\Omega$, if\\[-0.75\baselineskip]
	\begin{equation*}
		\alpha_{1}(\phi_{f}(x,t))\leq \beta(\alpha_{2}(x),t),\qquad \forall x\in G,~t\geq 0.
	\end{equation*}
\end{defn}
\vspace*{2mm}

We note that $\beta$-stability with respect to a measure $\alpha$ and with rate $k$ and gain $M$ as per Defn.~\ref{defn:betastability} implies $\hat{\beta}$-stability with respect to two measures, setting $\alpha_{1}(x)=M^{-1}\alpha(x)$, $\alpha_{2}(x)=\alpha(x)$, and $\hat{\beta}(y,t):=\beta(y,kt)$. Explicit dependency on the rate parameter is recovered in Cor.~\ref{cor:LF_sufficiency_rate} and Cor.~\ref{cor:LF_necessity_rate}.

We now start with the following relatively straightforward conditions for $\beta$-stability in terms of a Lyapunov function whose evolution is upper bounded by $\beta$.

\begin{thm}\label{thm:LF_sufficiency}
	For $\Omega \subseteq \R^n$ and $f:\Omega\to\R^{n}$, let $G\subseteq\Omega$ be forward invariant, and $\alpha_{1},\alpha_{2}:G\to\R_{+}$ be continuous. For any normalized and time-invariant $\beta$, if there exists $V:G\to\R_{+}$ such that for all $x\in G$,
	\begin{align}\label{eq:V_ineq1}
		\alpha_{1}(x)\leq V(x)&\leq \alpha_{2}(x),	\\
		V(\phi_{f}(x,t))&\leq \beta(V(x),t),\qquad \forall t\geq 0,\notag
	\end{align}
	then $f$ is $\beta$-stable on $G$ with respect to $\alpha_{1},\alpha_{2}$.
\end{thm}
\begin{proof}
	Suppose that $V$ satisfies~\eqref{eq:V_ineq1}. Since $\beta$ is normalized and time-invariant, by Lemma~\ref{lem:radially_increasing}, $\beta(y,t)$ is monotonically nondecreasing in $y$. Since forward invariance implies $\phi_{f}(x,t)\in G$ for all $x\in G$ and $t\geq 0$, it follows that
	\begin{equation*}
		\alpha_{1}(\phi_{f}(x,t))
		\leq V(\phi_{f}(x,t))	
		\leq \beta(V(x),t)
		\leq \beta(\alpha_{2}(x),t),
	\end{equation*}
	for all $x\in G$ and $t\geq 0$.
\end{proof}

Although Thm.~\ref{thm:LF_sufficiency} provides simple sufficient conditions for $\beta$-stability, these conditions are not readily verifiable without knowledge of the solution map, $\phi_{f}$. To express these conditions in terms of $\dot V$, then, we may use the comparison principle.
We recall the comparison principle in the following lemma, following immediately from Lemma~3.4 in~\cite{khalil2002nonlinear}. This lemma allows for continuous but not necessarily differentiable comparison functions, $v:\R\to\R$, using $D_{t}^{+} v$ to denote the upper right-hand (Dini) derivative of $v(t)$. However, for symmetry with the converse result in Lemma~\ref{lem:comparison_N}, we explicitly define $v(t):=V(\phi_f(x,t))$.

\begin{lem}[Comparison Principle]\label{lem:comparison}
	For $\rho:\R_{+}\to\R$, let $\beta$ be the unique continuous function satisfying $\beta(y,0)=y$ and $\partial_{t}\beta(y,t)=\rho(\beta(y,t))$ for all $t,y\geq 0$.
	For $\Omega\subseteq\R^{n}$ and $f:\Omega\to\R^{n}$, let $G$ be forward invariant for $f$, and let $V:G\to\R_{+}$ be such that $V(\phi_{f}(x,t))$ is continuous in $t\geq 0$ for all $x\in G$. Define $\dot{V}(x):=D_{t}^{+}V(\phi_{f}(x,t))|_{t=0}$.

	If $\dot{V}(x)\leq \rho(V(x))$ for all $x\in G$,
	then $V(\phi_{f}(x,t))\leq \beta(V(x),t)$ for all $x\in G$ and $t\geq 0$.
\end{lem}

The use of the comparison principle is not conservative since, as will be shown in Lemma~\ref{lem:comparison_N}, a bound of the form $V(\phi_{f}(x,t))\leq\beta(V(x),t)$ implies that $\dot V(x)\leq  \rho(V(x))$.

Of course, application of the comparison principle requires $\beta(y,0)=y$ and $\partial_{t}\beta(y,t)=\rho(\beta(y,t))$ for some $\rho\;:\;\R_+\rightarrow \R$. However, this is equivalent to normalization and time-invariance of $\beta$, as shown in the following proposition.

\begin{prop}\label{prop:ODE2flow}
	If $\beta(y,t)$ is differentiable in $t$, normalized, and time-invariant, then
	\begin{align}\label{eq:ODE_scalar}
		\partial_{t}\beta(y,t)&=\rho(\beta(y,t)),	&	\beta(y,0)&=y,	&
		\forall y,t\geq 0,
	\end{align}
	where $\rho(y):=\partial_{t}\beta(y,t)\vert_{t=0}$. Conversely, for any $\rho:\R_{+}\to\R$, if there exists a unique and continuous function $\beta$ that satisfies~\eqref{eq:ODE_scalar}, then $\beta$ is normalized and time-invariant.
\end{prop}
\begin{proof}
	Suppose that $\beta$ is normalized and time-invariant.
	Then, for all $y,t\geq 0$, letting $z=\beta(z,0)=\beta(y,t)$ we have
	\begin{equation*}
		\beta(\beta(y,t),s)=\beta(z,s)=\beta(y,t+s),\qquad \forall s\geq 0.
	\end{equation*}
	It follows that, for all $y,t\geq 0$,
	\begin{align*}
		\partial_{t}\beta(y,t)
		&=\partial_{s}\beta(y,t+s)|_{s=0}	\\
		&=\partial_{s}\beta(\beta(y,t),s)|_{s=0}
		=\rho(\beta(y,t)).
	\end{align*}
	Conversely, if $\beta$ is the unique and continuous function satisfying~\eqref{eq:ODE_scalar}, then $\beta$ is normalized by definition, and satisfies the semigroup property, $\beta(\beta(y,t),s)=\beta(y,t+s)$ for all $y,t,s\geq 0$. It follows that, if $\beta(y,t)=\beta(z,0)=z$ for some $y,t\geq 0$, then also $\beta(y,t+s)=\beta(\beta(y,t),s)=\beta(\beta(z,0),s)=\beta(z,s)$ for all $s\geq 0$, and thus $\beta$ is time-invariant.
\end{proof}

We now apply Prop.~\ref{prop:ODE2flow} and Lemma~\ref{lem:comparison} to tighten the condition $V(\phi_{f}(x,t))\leq \beta(V(x),t)$ in Thm.~\ref{thm:LF_sufficiency} to $\dot{V}(x)\leq \rho(V(x))$, obtaining the following sufficient Lyapunov conditions for $\beta$-stability with respect to two measures.

\begin{thm}\label{thm:LF_sufficiency_diff}
	For $\,\Omega \subseteq \R^n$ and $f:\Omega\to\R^{n}$, let $G\subseteq\Omega$ be forward invariant, and $\alpha_{1},\alpha_{2}:G\to\R_{+}$ be continuous. For $\rho:\R_{+}\to\R$, let $\beta$ be the unique continuous function satisfying $\beta(y,0)=y$ and $\partial_{t}\beta(y,t)=\rho(\beta(y,t))$.	
	If there exists a continuous function $V:G\to\R_{+}$ which
	satisfies
	\begin{align*}
		\alpha_{1}(x)\leq V(x)&\leq \alpha_{2}(x),	\\
		\dot{V}(x)&\leq \rho(V(x)),\qquad \forall x\in G,
	\end{align*}
	where $\dot{V}(x):=D_{t}^{+}V(\phi_{f}(x,t))|_{t=0}$,
	then
	$f$ is $\beta$-stable on $G$ with respect to $\alpha_{1},\alpha_{2}$. 
\end{thm}
\begin{proof}
	Suppose that $V:G\to \R_{+}$ satisfies the proposed conditions.
	Then,
	by Lemma~\ref{lem:comparison}, $V(\phi_{f}(x,t))\leq \beta(V(x),t)$ for all $x\in G$, and all $t\geq 0$. Since, by Prop.~\ref{prop:ODE2flow}, $\beta$ is normalized and time-invariant, it follows by Thm.~\ref{thm:LF_sufficiency} that $f$ is $\beta$-stable on $G$ with respect to $\alpha_{1},\alpha_{2}$.
\end{proof}

Having now obtained the main sufficiency results for $\beta$-stability with respect to two measures, we interpret this result using the original definition of $\beta$-stability using single-measure rate-performance.

\begin{cor}\label{cor:LF_sufficiency_rate}
	For $\Omega \subseteq \R^n$ and $f:\Omega\to\R^{n}$, let $G\subseteq\Omega$ be forward invariant, and $\alpha:G\to\R_{+}$ be continuous. For $\rho:\R_{+}\to\R$, let $\beta$ be the unique, continuous function satisfying $\beta(y,0)=y$ and $\partial_{t}\beta(y,t)=\rho(\beta(y,t))$ for all $t,y\geq 0$. For any $k\in\R$ and $M\geq 1$, if there exists a continuous function $V:G\to\R_{+}$ such that, for all $x\in G$,
	\begin{equation*}
		M^{-1}\alpha(x)\leq V(x)\leq \alpha(x),	\quad\text{and}\quad
		\dot{V}(x)\leq k\rho(V(x)),
	\end{equation*}
	where $\dot{V}(x):=D_{t}^{+}V(\phi_{f}(x,t))|_{t=0}$,
	then $f$ is $\beta$-stable on $G$ with respect to $\alpha$, with rate $k$ and gain $M$. 
\end{cor}
\begin{proof}
The proof follows directly from Thm.~\ref{thm:LF_sufficiency_diff} using $\alpha_{1}(x):=M^{-1}\alpha(x)$, $\alpha_{2}(x):=\alpha(x)$, and $\hat{\beta}(y,t):=\beta(y,kt)$, where we note that $\partial_t \beta(y,kt)\vert_{t=0} =k \partial_t \beta(y,t)\vert_{t=0}$.
\end{proof}

Having derived sufficient conditions for $\beta$-stability, in the next subsection, we show that these conditions are also necessary for $\beta$-stability with a given rate and gain.

\subsection{A Necessary Lyapunov Condition for $\beta$-Stability}\label{subsec:LF_characterization:necessity}

In the previous subsection, Thm.~\ref{thm:LF_sufficiency}, Lem.~\ref{lem:comparison} and Thm.~\ref{thm:LF_sufficiency_diff} established a sufficient condition for $\beta$-stability in terms of existence of a Lyapunov function.
In this subsection, we provide converse counterparts of those results, showing that $\beta$-stability of a vector field always implies existence of a Lyapunov function certifying this property.
First, Thm.~\ref{thm:LF_necessity} shows that if $f$ is $\beta$-stable, then there exists a $V$ which is upper bounded by $\beta$ -- i.e. $V(\phi_f(x,t))\le \beta (V(x),t)$. Next, Lem.~\ref{lem:comparison_N} shows that if this $V$ is continuous, then $\dot V(x) \le \rho (V(x))$ (a converse comparison principle). Finally, Thm.~\ref{thm:LF_necessity_diff_FT} provides conditions under which $V$ is continuous -- implying necessity of Thm.~\ref{thm:LF_sufficiency_diff}. Conditions under which $V$ is differentiable are provided in Thm.~\ref{thm:LF_necessity_diff}.

To start, Thm.~\ref{thm:LF_necessity} provides a converse of Thm.~\ref{thm:LF_sufficiency}. This shows that if $\phi_f$ is $\beta$-stable, then there exists a $V(x)$ whose evolution along trajectories is upper-bounded by $\beta$ -- satisfying the conditions of Thm.~\ref{thm:LF_sufficiency}.


\begin{thm}\label{thm:LF_necessity}
	For $\Omega\subseteq\R^{n}$ and $f:\Omega\to\R^{n}$, let $G\subseteq\Omega$ be forward invariant for $f$, and let $\alpha_{1},\alpha_{2}:G\to\R_{+}$. For any normalized and time-invariant $\beta$, if $f$ is $\beta$-stable on $G$ with respect to $\alpha_{1},\alpha_{2}$,
	 then there exists $V:G\to\R_{+}$ such that
	\begin{align}\label{eq:V_ineq2}
		\alpha_{1}(x)\leq V(x)&\leq \alpha_{2}(x),	&	&\forall x\in G,	\\
		V(\phi_{f}(x,t))&\leq \beta(V(x),t),	&	&\forall t\geq			\notag 0.
	\end{align}
\end{thm}
\vspace{2mm}
\begin{proof}
	Suppose that $f$ is $\beta$-stable on $G$ with respect to $\alpha_{1},\alpha_{2}$. To begin, we define the set-valued map
\begin{equation}\label{eq:Wset}
		\mcl{W}(x):=\bl\{y\in\R_{+}\mid \alpha_{1}\bl(\phi_{f}(x,t)\br)\leq \beta(y,t),~\forall t\geq 0\br\},
	\end{equation}
and the converse Lyapunov function 
	\begin{equation}\label{eq:Vconverse}
		V(x):=\inf_{y \in \mcl W(x)} y.
	\end{equation}
While this construction is somewhat awkward, when $\beta$ has a well-defined backwards-in-time continuation, it simplifies to the expression in Eqn.~\eqref{eq:V_sup}. Fig.~\ref{fig:LF_illustration} illustrates this construction and relationship to Eqn.~\eqref{eq:V_sup}.

First, note that since $f$ is $\beta$-stable on $G$, we have $\alpha_{2}(x)\in \mcl{W}(x)$ for all $x\in G$, and therefore $\mcl{W}(x)$ is non-empty. Since also $y\geq 0$ for all $y\in\mcl{W}(x)$ (by definition), the infimum is lower-bounded and hence well-defined for all $x\in G$.

	Next, since $\alpha_{2}(x)\in \mcl{W}(x)$, we have $V(x)\leq\alpha_{2}(x)$ for all $x\in G$. Furthermore, since $\beta$ is normalized, for any  $y\in\mcl{W}(x)$ we have $y=\beta(y,0)\geq\alpha_{1}(\phi_{f}(x,0))=\alpha_{1}(x)$, and therefore also $V(x)\geq\alpha_{1}(x)$ for all $x\in G$.
	
	Finally, to see that $V(\phi_{f}(x,t))\leq \beta(V(x),t)$, fix arbitrary $x\in G$. Since $\beta(y,t)$ is continuous in $y$, the set $\mcl{W}(x)$ is closed, and thus $V(x)=\inf_{y\in \mcl{W}(x)} y\in\mcl{W}(x)$. It follows that $\alpha_{1}(\phi_{f}(x,t))\leq \beta(V(x),t)$ for all $t\geq 0$, and therefore
	\begin{align*}
		\alpha_{1}\bl(\phi_{f}(\phi_{f}(x,t),s)\br)
		&= \alpha_{1}\bl(\phi_{f}(x,t+s)\br)	\\
		&\leq \beta(V(x),t+s)
		=\beta\bl(\beta(V(x),t),s\br),
	\end{align*}
	for all $t,s\geq 0$.
	By definition of $\mcl{W}(x)$, then, $\beta(V(x),t)\in\mcl{W}(\phi_{f}(x,t))$ and thus $V(\phi_{f}(x,t))\leq \beta(V(x),t)$ for $t\geq 0$.
\end{proof}

\begin{figure}[t]
	\centering
	\includegraphics[width=1.0\linewidth]{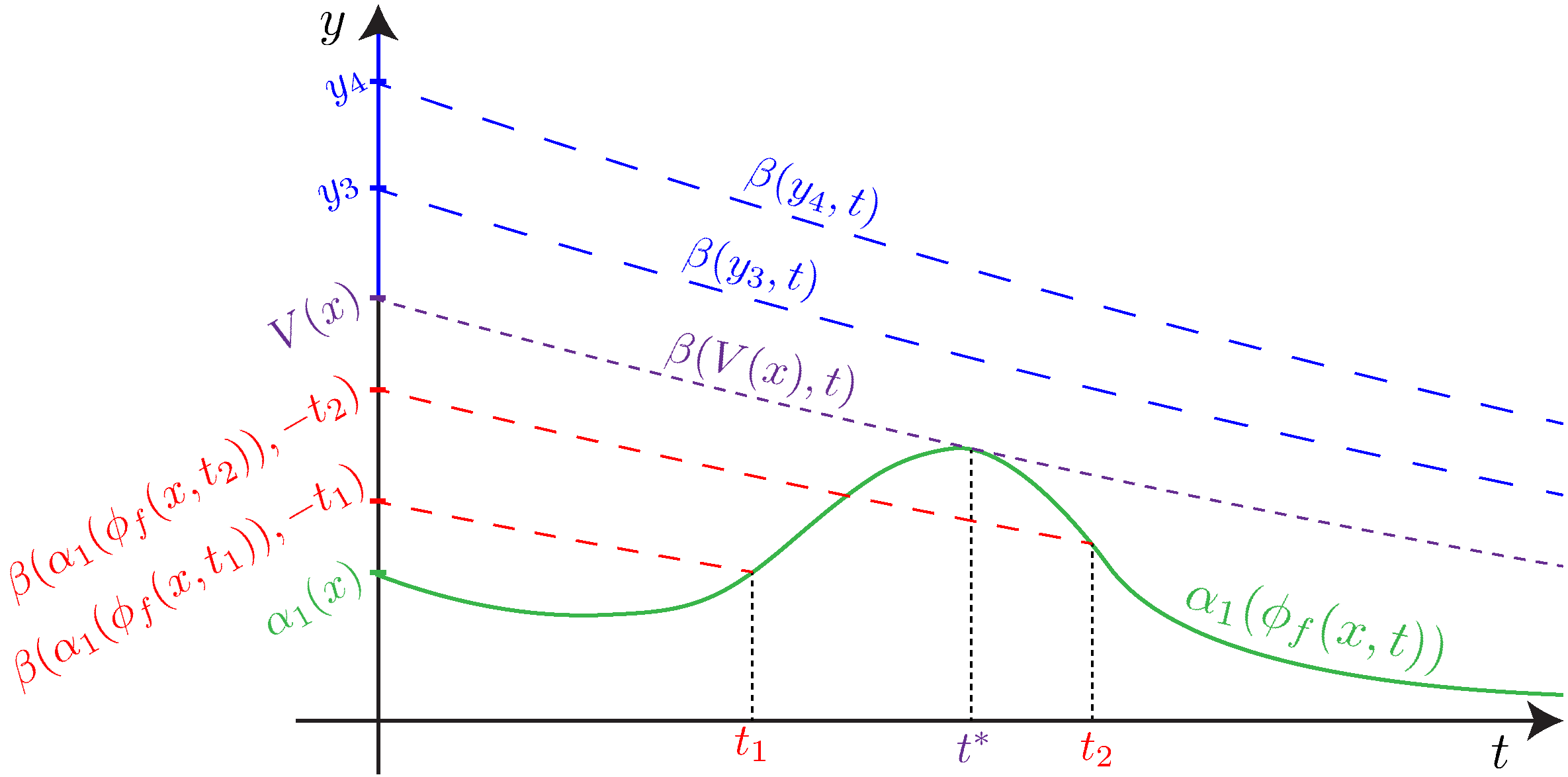}
	\caption{Illustration of the set\\
	$\mcl{W}(x):=\{y\in\R_{+}\mid \alpha_{1}(\phi_{f}(x,t))\leq\beta(y,t),~\forall t\in\R_{+}\}$ (in blue) for a particular curve $\alpha_{1}(\phi_{f}(x,t))$. The converse Lyapunov function in the proof of Thm.~\ref{thm:LF_necessity} is then given by $V(x):=\inf_{y\in\mcl{W}(x)}y$. Note that if $\beta(y,-t)$ is well-defined and $\beta(\beta(y,t),-t)=y$, then $\beta(\alpha_{1}(\phi_{f}(x,t)),-t)\leq y$ for any $y\in\mcl{W}(x)$ and $t\geq 0$. It follows that, in this case, $V(x):=\sup_{t\in[0,\infty)}\beta(\alpha_{1}(\phi_{f}(x,t)),-t)$.}
	\label{fig:LF_illustration}	
	\vspace*{-0.4cm}
\end{figure}

Thm.~\ref{thm:LF_necessity} proves that the sufficient Lyapunov conditions from Thm.~\ref{thm:LF_sufficiency} are also necessary for $\beta$-stability with respect to two measures. In particular, $\beta$-stability implies that the function $V(x):=\inf_{y\in\mcl{W}(x)}y$ with $\mcl{W}(x)$ as in~\eqref{eq:Wset} satisfies the Lyapunov conditions of Thm.~\ref{thm:LF_sufficiency}. Here, we note that if $\beta(y,t)$ is well-defined and time-invariant at negative times (so that $\beta(\beta(y,t),-t)=y$ for all $y,t\geq 0$) then this converse Lyapunov function can be equivalently expressed as
\begin{equation}\label{eq:V_sup}
	V(x):=\sup_{t\in[0,T_{f}(x))} \beta\bl(\alpha_{1}(\phi_{f}(x,t)),-t\br),
\end{equation}
where $T_{f}(x):=\sup\{t\geq0\mid \alpha_{1}(\phi_{f}(x,t))>0\}$. Sufficient conditions under which this construction may be used are proven in Lemma~\ref{lem:V_sup_appx} in Appx.~\ref{appx:proofs}.

As discussed in Subsection~\ref{subsec:LF_characterization_sufficiency}, to tighten the condition $V(\phi_{f}(x,t))\leq\beta(V(x),t)$ in Thm.~\ref{thm:LF_sufficiency} to $\dot{V}(x)\leq \rho(V(x))$ in  Thm.~\ref{thm:LF_sufficiency_diff} (where $\rho(y):=\partial_{t}\beta(y,t)|_{t=0}$), we use the comparison principle in Lemma~\ref{lem:comparison}. In order to prove the converse of Thm.~\ref{thm:LF_sufficiency_diff}, therefore, we first prove the converse of Lemma~\ref{lem:comparison} -- showing that $\dot{V}(x)\leq \rho(V(x))$ is also necessary for $V(\phi_{f}(x,t))\leq\beta(V(x),t)$.

\begin{lem}[Converse Comparison Principle]\label{lem:comparison_N}
	Let $\beta$ be normalized and time-invariant, and $\rho(y):=\partial_{t}\beta(y,t)|_{t=0}$.
	For $\Omega\subseteq\R^{n}$ and $f:\Omega\to\R^{n}$, let $G$ be forward invariant for $f$. Let $V:G\to\R_{+}$ be such that $V(\phi_{f}(x,t))$ is continuous in $t\geq 0$ for all $x\in G$, and define $\dot{V}(x):=D_{t}^{+}V(\phi_{f}(x,t))|_{t=0}$.

 	If $V(\phi_{f}(x,t))\leq \beta(V(x),t)$ for all $x\in G$ and $t\geq 0$,
	then $\dot{V}(x)\leq \rho(V(x))$ for all $x\in G$.
\end{lem}
\begin{proof}
	Suppose that $V(\phi(x,t))\leq \beta(V(x),t)$ for all $x\in G$ and $t\geq 0$. Then, for all $x\in G$, using the fact that $V(\phi_{f}(x,0))=V(x)=\beta(V(x),0)$, we find
	\begin{align*}
		\dot{V}(x)&=D_{t}^{+}V(\phi_{f}(x,0))	\\
		&=\limsup_{\Delta t\to 0^{+}}\frac{1}{\Delta t}\bbl[V(\phi_{f}(x,\Delta t))-V(\phi_{f}(x,0))\bbr]	\\
		&\leq \limsup_{\Delta t\to 0^{+}}\frac{1}{\Delta t}\bbl[\beta(V(x),\Delta t)-\beta(V(x),0)\bbr]	\\
		&=\partial_{t}\beta\bl(V(x),t\br)\br|_{t=0}		
		=\rho\bl(V(x)\br). \\[-2\baselineskip]
	\end{align*}
\end{proof}

Lemma~\ref{lem:comparison_N} offers a converse of the comparison principle in Lemma~\ref{lem:comparison}, showing that if $V$ is continuous, then $V(\phi_{f}(x,t))\leq\beta(V(x),t)$ implies $\dot{V}(x)\leq \rho(V(x))$ -- a condition which is significantly easier to test than the inequality in~\eqref{eq:V_ineq2}. 
To ensure that the converse Lyapunov function in the proof of Thm.~\ref{thm:LF_necessity} is indeed continuous, we may impose the condition that $\beta$ is monotonically decreasing and, in the case of finite-time stability, that the settling time function is continuous. With these conditions, the following theorem shows that there exists a \textit{continuous} $V$ which satisfies the condition $\dot{V}(x)\leq\rho(V(x))$.
In particular, the result shows that we may define this converse Lyapunov function using the construction in Eqn.~\eqref{eq:V_sup}.

\begin{thm}\label{thm:LF_necessity_diff_FT}
	For $\,\Omega \subseteq \R^n$ and $f:\Omega\to\R^{n}$ with $f(0)=0$, let $G\subseteq\Omega$ be forward invariant, and $\alpha_{1},\alpha_{2}:G\to\R_{+}$ be continuous. For $\rho:\R_{+}\to\R$ continuous on $(0,\infty)$, let $\beta$ be the unique continuous function satisfying $\beta(y,0)=y$ and $\partial_{t}\beta(y,t)=\rho(\beta(y,t))$.
	Suppose that $\alpha_{1}$ is positive definite, and $\beta(y,t)$ is monotonically decreasing in $t$ whenever $\beta(y,t)>0$. Finally, suppose that $T_{f}(x):=\sup\,\{t\geq 0\mid \alpha_{1}(\phi_{f}(x,t))>0\}$ (possibly infinite) is continuous for $x\in G\setminus\{0\}$.
	
	If $f$ is $\beta$-stable on $G$ with respect to $\alpha_{1},\alpha_{2}$,
	then there exists a continuous function $V:G\to\R_{+}$ such that\\[-1.25\baselineskip]
	\begin{align*}
		\alpha_{1}(x)\leq V(x)&\leq \alpha_{2}(x),	&	
		\dot{V}(x)&\leq \rho(V(x)),	&	&\forall x\in G, 
	\end{align*}
	where $\dot{V}(x):=D_{t}^{+}V(\phi_{f}(x,t))|_{t=0}$.
\end{thm}

\vspace{2mm}

\begin{proof}
	A sketch of the proof is as follows. 
	Define $V(x)$ as in Eqn~\eqref{eq:V_sup}.
	By continuity of $\beta$, $\alpha_{1}$, $\phi_{f}$, and $T_{f}$, it follows that $V$ is continuous. In addition, we can show that $V$ is equivalent to the converse Lyapunov function from Lemma~\ref{thm:LF_necessity}. Using that lemma as well as Lemma~\ref{lem:comparison_N}, we find that $V$ satisfies the proposed properties. 
	A formal proof is given in Appx~\ref{appx:proofs:necessity_diff_FT}.
\end{proof}

Thm.~\ref{thm:LF_necessity_diff} shows that, if $\beta$ is monotonically decreasing and (in the case of finite-time stability) the settling time function, $T_{f}(x)$, is continuous, stability of $f$ with respect to given $\beta$ implies existence of a continuous Lyapunov function satisfying $
D_{t}^{+}V(\phi_{f}(x,t))|_{t=0}\leq \rho(V(x))$. Here, if $V(x)$ is differentiable, this Lyapunov condition can be equivalently expressed as $\nabla V(x)^T f(x)\leq \rho(V(x))$ --- a condition that can be tested without knowledge of the solution map, $\phi_{f}(x,t)$. 
To ensure this differentiability of the converse Lyapunov function in Thm.~\ref{thm:LF_necessity_diff_FT}, we may impose the additional conditions that $f$ and $\beta$ are locally Lipschitz (thereby precluding the possibility of finite-time stability). With these additional conditions, the following theorem shows that for any $\epsilon\in(0,1)$, there exists a \textit{locally Lipschitz continuous} $V_{\epsilon}$ which satisfies the weakened conditions $\nabla{V}_{\epsilon}(x)^T f(x)\leq(1-\epsilon)\rho(V_{\epsilon}(x))$.

\begin{thm}\label{thm:LF_necessity_diff}
	For $\,\Omega \subseteq \R^n$ and $f:\Omega\to\R^{n}$ with $f(0)=0$, let $G\subseteq\Omega$ be forward invariant, and $\alpha_{1},\alpha_{2}:G\to\R_{+}$ be continuous. For $\rho:\R_{+}\to\R$ continuous on $(0,\infty)$, let $\beta$ be the unique continuous function satisfying $\beta(y,0)=y$ and $\partial_{t}\beta(y,t)=\rho(\beta(y,t))$.
	Suppose that $\alpha_{1}$ is positive definite and coercive, $\beta(y,t)$ is monotonically decreasing in $t$ for $y>0$, and that $f$, $\alpha_{1}$, and $\beta$ are locally Lipschitz continuous.
	
	If $f$ is $\beta$-stable on $G$ with respect to $\alpha_{1},\alpha_{2}$,
	then for every $\epsilon\in(0,1)$ there exists a continuous function $V_{\epsilon}:G\to\R_{+}$ which is locally Lipschitz continuous on $G\setminus\{0\}$ and differentiable almost everywhere and satisfies\\[-1.3\baselineskip]
	\begin{align*}
		\alpha_{1}(x)\leq V_{\epsilon}(x)&\leq \alpha_{2}(x),	&	&\forall x\in G,	\\
		\nabla V_{\epsilon}(x)^T f(x)&\leq (1-\epsilon)\rho(V_{\epsilon}(x)),	&	&\text{for a.e. } x\in G. 
	\end{align*}
\end{thm}
\vspace*{2mm}
\begin{proof}
	A sketch of the proof is as follows. Define \\[-0.8\baselineskip]
	\begin{equation*}
		V_{\epsilon}(x):=\sup_{t\in[0,\infty)} \beta\bl(\alpha_{1}\bl(\phi_{f}(x,t)\br)-[1-\epsilon]t\br),
	\end{equation*}
	\ \\[-0.75\baselineskip]
	for $\epsilon\in(0,1)$. Then, by Thm.~\ref{thm:LF_necessity_diff_FT}, $\dot{V}_{\epsilon}(x)\leq (1-\epsilon)\rho(V_{\epsilon}(x))$, for all $x\in G$. Furthermore, since $f$, $\alpha_{1}$, and $\beta$ are locally Lipschitz, $V_{\epsilon}$ is locally Lipschitz on $G\setminus\{0\}$, and therefore differentiable almost everywhere. It follows that $\nabla V_{\epsilon}(x)^T f(x)\leq (1-\epsilon)\rho(V_{\epsilon}(x))$ for all $x\in G$ for which $\nabla V_{\epsilon}(x)$ is well-defined.
	A formal proof is provided in Appx.~\ref{appx:proofs:necessity_diff}.
\end{proof}

Thm.~\ref{thm:LF_necessity_diff} shows that, if $f,\beta$ are locally Lipschitz and $\beta$ is monotonically decreasing, $\beta$-stability of $f$ implies existence of a converse Lyapunov function which is differentiable almost everywhere.
Of course, the constraint that $f$ and $\beta$ be locally Lipschitz is violated for finite-time stable systems, in which case differentiability of $V$ will need to be verified explicitly.


To conclude this section, we combine Thm.~\ref{thm:LF_necessity_diff} and Cor.~\ref{cor:LF_sufficiency_rate} to obtain the following necessary and sufficient conditions for $\beta$-stability with respect to a single measure and with a given rate and gain performance, as in Definition~\ref{defn:betastability}.

\begin{cor}\label{cor:LF_necessity_rate}
	For $\Omega\subseteq\R^{n}$, let $G\subseteq\Omega$, $f:\Omega\to \R^{n}$, $\rho$, $\beta$, and $\alpha_{1}$ satisfy the conditions of Thm.~\ref{thm:LF_necessity_diff}, and let $\alpha:=\alpha_{1}$.
	Then, $f$ is $\beta$-stable on $G$ with respect to $\alpha$ and with rate $k$ and gain $M$ if and only if for every $\epsilon\in(0,1)$ there exists a continuous function $V_{\epsilon}:G\to\R_{+}$ which is locally Lipschitz on $G\setminus\{0\}$ and satisfies\\[-1.25\baselineskip]
	\begin{align}\label{eq:LF_conditions_rate}
		M^{-1}\alpha(x)&\leq V_{\epsilon}(x)\leq \alpha(x),	&	&\forall x\in G,	\\
		\nabla V_{\epsilon}(x)^T f(x)&\leq (1-\epsilon)k\rho(V_{\epsilon}(x)),	&	&\text{for a.e. }x\in G.\notag
	\end{align}
\end{cor}
\vspace*{2mm}
\begin{proof}
	Necessity follows immediately from Thm.~\ref{thm:LF_necessity_diff}, using $\alpha_{1}=M^{-1}\alpha(x)$, $\alpha_{2}(x)=\alpha(x)$, and $\beta(y,t)\mapsto \beta(y,kt)$, so that $\rho(y)\mapsto k\rho(y)$.
	For sufficiency, suppose $V_{\epsilon}$ satisfies the conditions in Eqn.~\eqref{eq:LF_conditions_rate}, for any $\epsilon\in(0,1)$. Then for all $x\in G$, 
	\begin{align*}
		&\dot{V}_{\epsilon}(x)
		:=\limsup_{\Delta t\to 0^{+}}\frac{1}{\Delta t}\bbl[V_{\epsilon}(\phi_{f}(x,t)) -V_{\epsilon}(\phi_{f}(x,0))\bbr]	\\	
		&=\limsup_{\Delta t\to 0^{+}}\frac{1}{\Delta t}\int_{0}^{\Delta t}\nabla V_{\epsilon}(\phi_{f}(x,t))^T f(\phi_{f}(x,t))dt	\\
		&\leq \limsup_{\Delta t\to 0^{+}}\frac{(1-\epsilon)k}{\Delta t}\!\!\int_{0}^{\Delta t}\!\!\!\rho\bl(V_{\epsilon}(\phi_{f}(x,t))\br)dt
		=\!(1-\epsilon)k\rho(V_{\epsilon}(x)).
	\end{align*}
	By Cor.~\ref{cor:LF_sufficiency_rate}, it follows that $f$ is $\beta$-stable on $G$ with respect to $\alpha$ and with rate $(1-\epsilon)k$ and gain $M$.
	Letting $\epsilon\to 0$, we conclude that $f$ is $\beta$-stable on $G$ with rate $k$.
%
\end{proof}

Cor.~\ref{cor:LF_necessity_rate} implies that if a system is $\beta$-stable with rate $k$, the Lyapunov conditions in Eqn.~\eqref{eq:LF_conditions_rate} will be feasible for any $\epsilon \in (0,1)$ --- thereby establishing $\beta$-stability with rate $(1-\epsilon)k$. Since $\epsilon$ may be made arbitrarily small, this implies that any algorithm capable of testing the Lyapunov conditions in Eqn.~\eqref{eq:LF_conditions_rate} can be combined with a bisection search to quantify the rate performance of the system with arbitrary accuracy. In the following section, we use sum-of-squares optimization to propose such algorithms for the $\beta$ functions associated with exponential, rational, and finite-time stability. In Section~\ref{sec:SOS_tests}, we will then apply these algorithms to quantify local and global performance of a range of representative nonlinear systems.

\section{SOS Conditions for Rate Analysis of Exponential, Rational, and Finite-Time Stability}\label{sec:SOS_tests}

Having established that $\beta$-stability rate performance can be equivalently characterized by Lyapunov conditions, we now show how such conditions can be tested numerically for three classical notions of stability: exponential stability (Subsec.~\ref{subsec:SOS_tests:exponential}), rational stability (Subsec.~\ref{subsec:SOS_tests:rational}), and finite-time stability (Subsec.~\ref{subsec:SOS_tests:finite_time}). For each stability notion, we apply Cor.~\ref{cor:LF_necessity_rate} to derive necessary and sufficient Lyapunov conditions for stability with a given rate (and gain) on a given domain.

For exponential stability, the resulting Lyapunov conditions match the classical conditions for testing exponential stability, $M^{-1}\norm{x}_{2}^2\leq V(x)\leq \norm{x}_{2}^2$ and $\dot{V}(x)\leq -kV(x)$, and these conditions have a straightforward implementation using Sum-Of-Squares (SOS) programming when the vector field is polynomial. For rational stability, we require $M^{-1}\norm{x}_{2}^{p}\leq V(x)\leq \norm{x}_{2}^{p}$ and $\dot{V}(x)\leq -kV(x)^2$ (similar to the sufficient conditions used in~\cite{jammazi2013rational}). These conditions can be enforced using SOS by noting that $V(x)\leq \norm{x}_{2}^{p}$ and $\dot{V}(x)\leq -kV(x)\norm{x}_{2}^{p}$ imply $\dot{V}(x)\leq -kV(x)^2$. Finally, for finite-time stability, Cor.~\ref{cor:LF_necessity_rate} yields conditions of the form $M^{-1}\norm{x}_{p}^{\eta}\leq V(x)\leq \norm{x}_{p}^{\eta}$ and $\dot{V}(x)\leq -k$ (similar to those in~\cite{roxin1966FTS}). To enforce these conditions using SOS, we apply the variable transformation approach in~\cite{kisole2025FiniteTimeSOS} to obtain conditions of the form $M^{-\frac{2}{\eta}}\norm{z}_{2}^{2d}\leq \tilde{V}(z) \leq \norm{z}_{2}^{2d}$ and $\dot{\tilde V}(z) \leq -\frac{2k}{\eta}\norm{z}_{2}^{(2-\eta) d}$.

\noindent\textbf{Notation:} Throughout this section, we denote by $\R_{d}^{m}[x]$ the set of vector-valued polynomials of degree at most $d\in\N$ in $x\in\R^{n}$. We will tighten polynomial positivity conditions (e.g. $p(x)\ge 0$ for all $x\in \R^n$) to SOS conditions, using the notation $p \in \Sigma_{s,d}$ for even $d \in \N$ to denote the constraint $p(x)=Z_{d/2}(x)^T P Z_{d/2}(x)$ for some $P\succeq 0$ and where $Z_{d/2}$ is the vector of monomials of degree $d/2$ or less. For a semialgebraic region, $\Omega := \{x \mid g_i(x)\ge 0\}$ with $g_{i}\in\R_{d_{i}}[x]$, we tighten local polynomial positivity conditions (e.g. $p(x)\ge 0$ for all $x\in \Omega$) to Putinar-based Positivstellensatz conditions~\cite{putinar1993psatz}, so that the constraint $p \in \Sigma_{s,d}[\Omega]$ means $p(x)=s_{0}(x)+\sum_i s_i(x)g_i(s)$ for some $s_0 \in \Sigma_{s,d}$ and $s_i\in \Sigma_{s,d-d_{i}}$. For $V:\R^{n}\to\R_{+}$ and $c \in \R_{+}$ we use $S_c(V)$  to denote the sublevel set $S_{c}(V):=\{x\in\R^{n}\mid V(x)\leq c\}$. 

\subsection{Conditions for Exponential Decay Rate}\label{subsec:SOS_tests:exponential}

Consider first the notion of exponential stability as in Defn.~\ref{defn:exponential_stability}, which is equivalent to $\beta_{\tnf{e}}$-stability with $\beta_{\tnf{e}}(y,t)=e^{-t}y$.
Applying Cor.~\ref{cor:LF_necessity_rate}, we obtain the following Lyapunov conditions for exponential stability rate and gain performance, $k,M$.

\begin{cor}\label{cor:exponential_stability_LF}
	For $\Omega\subseteq\R^{n}$, let $f:\Omega\to\R^{n}$ with $f(0)=0$ be locally Lipschitz continuous, and $G\subseteq \Omega$ be forward invariant for $f$. Then, $f$ is exponentially stable on $G$ with rate $k\geq 0$ and gain $M\geq 1$ if and only if for all $\epsilon\in(0,1)$ there exists a continuous function $V:G\to\R_{+}$ which is locally Lipschitz continuous on $G\setminus\{0\}$ and satisfies\\[-1.25\baselineskip]
	\begin{align*}
		M^{-1}\norm{x}_{2}&\leq V(x) \leq \norm{x}_{2},		&	&\forall x\in G,	\\
		\nabla V(x)^T f(x)&\leq -(1-\epsilon)kV(x),	&	&\text{for a.e. }x\in G. 
	\end{align*}
\end{cor}
\vspace*{2mm}
\begin{proof}
	Define $\rho_{\tnf{e}}(y):=\partial_{t}\beta_{\tnf{e}}(y,t)|_{t=0}=-y$. Then, $\beta_{\tnf{e}}(y,t)$ is the unique and continuous function $\beta$ satisfying $\partial_{t}\beta(y,t)=\rho_{\tnf{e}}(\beta(y,t))$ with $\beta(y,0)=y$.
	Since $-(1-\epsilon)kV(x)=(1-\epsilon)k\rho_{\tnf{e}}(V(x))$, the result now follows immediately from Cor.~\ref{cor:LF_necessity_rate} and Lemma~\ref{lem:exponential_stability}.
%
\end{proof}

The Lyapunov conditions in Cor.~\ref{cor:exponential_stability_LF} are similar to those typically used for proofs of exponential stability (in e.g.~\cite{khalil2002nonlinear,grune2004KLD}). However, Cor.~\ref{cor:exponential_stability_LF} is slightly stronger than typical Lyapunov conditions for exponential stability in that it shows that, using the proposed conditions, the exponential rate performance on a region $G$ can be established with arbitrary accuracy.


While SOS is often used to test basic notions of exponential stability, and while it has been shown that the use of polynomial Lyapunov functions is not conservative for such tests~\cite{peet2009exponentialStability}, the equivalent rate conditions of Cor.~\ref{cor:exponential_stability_LF} are atypical in that we have a fixed degree-1 upper and lower bound on the Lyapunov function (which ensure normalization of gain). To allow for higher degree polynomial Lyapunov functions without altering the definition of exponential stability, we may use a change of variables $V(x)\mapsto V(x)^{2d}$ which allows these conditions to be equivalently represented as $\norm{x}_{2}^{2d}\leq V(x) \leq M^{2d} \norm{x}_{2}^{2d}$ and $\nabla V(x)^T f(x)\leq -2dkV(x)$. These conditions are combined with forward invariance of sublevel sets to obtain Prop.~\ref{prop:exponential_stability_LF} which may be tested using SOS.

\begin{prop}\label{prop:exponential_stability_LF}
	Let $\Omega\subseteq\R^{n}$ and $f:\Omega\to\R^{n}$. If for some $k\geq 0$ and $M\geq 1$ there exist $d>0$ and a differentiable function $V:\Omega\to\R_{+}$ such that for all $x\in\Omega$,
	\begin{align*}
		\norm{x}_{2}^{2d}\leq V(x) &\leq M^{2d} \norm{x}_{2}^{2d},	&
		\nabla V(x)^T f(x)&\leq -2dkV(x),
	\end{align*}
then for any $c \in \R_{+}$ such that $S_{c}(V) \subseteq \tnf{int}(\Omega)$, 	$f$ is exponentially stable on $S_{c}(V)$ with rate $k$ and gain $M$.
\end{prop}
\begin{proof}
	Suppose that $V:\Omega\to\R_{+}$ satisfies the proposed conditions, and $c\in\R_{+}$ is such that $S_{c}(V) \subseteq \tnf{int}(\Omega)$. Since $\nabla V(x)^T f(x)\leq 0$ for all $x\in S_{c}(V)\subseteq\Omega$, and $S_{c}$ is compact, the sublevel set $S_{c}(V)$ is forward invariant for $f$. Now, define $\tilde{V}(x)=M^{-\frac{1}{2d}}V(x)^{\frac{1}{2d}}$. Then for all $x\in S_{c}(V)$, we have $M^{-1}\norm{x}_{2}\leq \tilde{V}(x)\leq \norm{x}_{2}$ and
	\begin{align*}
		\nabla\tilde{V}(x)^T f(x)
		&= \frac{1}{2d}V(x)^{\frac{1}{2d}-1}\nabla V(x)^T f(x)	
		\leq-k\tilde{V}(x).
	\end{align*}
	By Cor.~\ref{cor:exponential_stability_LF}, it follows that $f$ is exponentially stable on $S_{c}(V)$ with rate $k$ and gain $M$.
\end{proof}

The SOS conditions associated with Prop.~\ref{prop:exponential_stability_LF} are now readily stated as follows.

\begin{cor}\label{cor:exponential_stability_SOS}
	Let $d\in\N$, $f\in \R^{n}_{d_{f}}[x]$, and $g_{i} \in \R_{d_{i}}[x]$, and define $\Omega:=\{x\in\R^{n}\mid g_{i}(x)\geq 0\}$. If $V$ solves
\begin{align}\label{eq:exponential_stability_SOSP}
		\max_{k\geq 0,~\gamma>0,~V}&	&	&k ,&	\\
		\tnf{s.t.}&		&	& V(x)-(x^T x)^{d}&\in\Sigma_{s,2d}[\Omega],	\notag\\				
		&		&		&\gamma(x^T x)^{d} -V(x)&\in\Sigma_{s,2d}[\Omega],	\notag\\
		&		&		&\hspace*{-1.5cm}-2dk V(x)-\nabla V(x)^T f(x) &\in\Sigma_{s,d'}[\Omega],\notag
	\end{align}
	where $d'=\max\{2d,2d-1+d_{f}\}$, then for any $c\in\R_{+}$ such that $S_{c}(V)\subseteq\tnf{int}(\Omega)$, $f$ is exponentially stable on $S_{c}(V)$ with rate $k$ and gain $M=\gamma^{\frac{1}{2d}}$. 
\end{cor}
\begin{proof}
	Suppose the optimization program is feasible. Then the solution $V$ satisfies all conditions of Prop.~\ref{prop:exponential_stability_LF} with $M=\gamma^{\frac{1}{2d}}$. By that proposition, $f$ is exponentially stable on $S_{c}(V)$ with rate $k$ and gain $M$.
\end{proof}

Note that, for any fixed $k\geq 0$, the optimization program in Cor.~\ref{cor:exponential_stability_SOS} is linear in $\gamma$ and $V$, and feasibility can be tested using semidefinite programming. Since the objective of the optimization problem is monotone in $k$, bisection on the parameter $k$ can then be used to find the largest such $k$.

\subsection{Conditions for Rational Decay Rate}\label{subsec:SOS_tests:rational}

Consider now the notion of rational stability from Defn.~\ref{defn:rational_stability}, characterized as $\beta_{\tnf{r}}$-stability in Lemma~\ref{lem:rational_stability}, with $\beta_{\tnf{r}}(y,t)=\frac{y}{1+yt}$. Applying Prop.~\ref{prop:ODE2flow} and Cor.~\ref{cor:LF_necessity_rate}, gain and rate performance can be determined using the following Lyapunov conditions.

\begin{cor}\label{cor:rational_stability_LF}
	For $\Omega\subseteq\R^{n}$, let $f:\Omega\to\R^{n}$ with $f(0)=0$ be locally Lipschitz continuous, and $G\subseteq \Omega$ be forward invariant for $f$. Then, for $p \in \N$,  $f$ is rationally stable on $G$ with rate $k\geq 0$ and gain $M\geq 1$ if and only if for all $\epsilon\in(0,1)$, there exists a continuous function $V:G\to\R_{+}$ which is locally Lipschitz on $G\setminus\{0\}$ and satisfies\\[-1.25\baselineskip]
	\begin{align*}
		M^{-1}\norm{x}_{2}^{p}&\leq V(x) \leq \norm{x}_{2}^{p},	&	&\forall x\in G,	\\
		\nabla V(x)^T f(x)&\leq -(1-\epsilon)kV(x)^2,&	&\text{for a.e. }x\in G.\notag 
	\end{align*}
\end{cor}
\vspace*{2mm}
\begin{proof}
	Define $\rho_{\tnf{r}}(y):=\partial_{t}\beta_{\tnf{r}}(y,t)|_{t=0}=-y^2$ for $y\geq0$. Then, $\beta_{\tnf{r}}(y,t)$ is the unique and continuous function $\beta$ satisfying $\partial_{t}\beta(y,t)=\rho_{\tnf{r}}(\beta(y,t))$ with $\beta(y,0)=y$.
	By Cor.~\ref{cor:LF_necessity_rate}, it follows that $f$ is $\beta_{\tnf{r}}$-stable on $G$ with respect to $\alpha(x):=\norm{x}_{2}^{p}$ and with rate $k$ and gain $M$ if and only if there exists a function $V$ satisfying the proposed conditions. The result then follows immediately from Lemma~\ref{lem:rational_stability}.
\end{proof}

Cor.~\ref{cor:rational_stability_LF} provides an equivalent Lyapunov characterization of rational rate performance. These conditions are similar to the sufficient conditions for rational stability from e.g.~\cite{jammazi2013rational}, proving that those conditions are also necessary for rational stability with a particular rate. However, the obtained Lyapunov conditions are substantially different from the more standard necessary and sufficient conditions for a relaxed definition of rational stability in~\cite{bacciotti2005LF_book}, which are given as
\begin{equation}\label{eq:rational_classical_LF}
	C_{1}\norm{x}_{2}^{r_{1}}\le V(x)\le C_{2}\norm{x}_{2}^{r_{2}},	\quad
	\dot V(x) \le -C_{3} \norm{x}_{2}^{r_{3}},
\end{equation}
for some $r_{3}>r_{2}$ and $r_{1}\geq r_{2}$, and where we take $r_{1}=r_{2}$ in order to ensure normalization of gain. These standard Lyapunov conditions have the advantage that they are linear in the decision variable, $V$, providing a straightforward implementation using SOS.
The disadvantage of these conditions, however, is that they may yield conservative maximum lower bounds on the rate performance. In the following lemma, we quantify this conservatism by establishing a direct relationship between the two Lyapunov characterizations,
showing how any $\{V,C_{i},r_{1}=r_{2},r_{3}\}$ satisfying the standard conditions in~\eqref{eq:rational_classical_LF} may be mapped to $\{\hat{V},M,k,p\}$ satisfying the conditions in Cor.~\ref{cor:rational_stability_LF}, and vice versa.
We also provide an intermediate Lyapunov characterization which is linear in $V$, but is less conservative than the standard conditions.

\begin{lem}\label{lem:rational_LFs_equivalence}
	For $G\subseteq\R^{n}$, $f:G\to\R^{n}$, and $V,\tilde{V},\hat{V}:G\to\R_{+}$, consider the following conditions:
	\begin{enumerate}
		\item[\tnf{(i)}]
		There exist $C_{1},C_{2}>0$, $C_{3}\geq 0$, $q>r>0$ such that
		\begin{align*}
			C_{1}\norm{x}_{2}^{r}\leq V(x)&\leq C_{2}\norm{x}_{2}^{r},	\\
			\nabla V(x)^T f(x) &\leq -C_{3} \norm{x}_{2}^{q},\quad \forall x\in G.\notag
		\end{align*}
		\item[\tnf{(ii)}]
		There exist $k\geq 0$, $\gamma,p,r>0$ such that
		\begin{align*}
			\gamma^{-1}\norm{x}_{2}^{r}\leq \tilde{V}(x) &\leq \norm{x}_{2}^{r},	\\
			\nabla \tilde{V}(x)^T f(x) &\leq -c\tilde{V}(x)\norm{x}_{2}^{p},\quad \forall x\in G.\notag
		\end{align*}
		\item[\tnf{(iii)}]
		There exist $k\geq 0$, $M,p>0$ such that
		\begin{align*}
			M^{-1}\norm{x}_{2}^{p}\leq \hat{V}(x) &\leq \norm{x}_{2}^{p},	\\
			\nabla \hat{V}(x)^T f(x)&\leq -k\hat{V}(x)^2,\quad \forall x\in G.\notag
		\end{align*}
	\end{enumerate}
	Then, the following statements hold:
	\begin{enumerate}
		\item \tnf{(i)} implies \tnf{(ii)} with $p=q-r$, $\gamma=\frac{C_{2}}{C_{1}}$, $c=\frac{C_{3}}{C_{2}}$.
		\item \tnf{(ii)} implies \tnf{(i)} with $C_{1}=\frac{1}{\gamma}$, $C_{2}=1$, $C_{3}=\frac{c}{\gamma}$, $q=r+p$.
		\item \tnf{(ii)} implies \tnf{(iii)} with $M=\gamma^{\frac{p}{r}}$, $k=\frac{cp}{r}$.
		\item \tnf{(iii)} implies \tnf{(ii)} with $r=p$, $\gamma=M$, $c=\frac{k}{M}$.		
		\item \tnf{(i)} implies \tnf{(iii)} with $p=q-r$, $M=(\frac{C_{2}}{C_{1}})^{\frac{q}{r}-1}$, $k=\bl(\frac{q}{r}-1\br)\frac{C_{3}}{C_{2}}$.
		\item \tnf{(iii)} implies \tnf{(i)} with $C_{1}=\frac{1}{M}$, $C_{2}=1$, $C_{3}=\frac{k}{M^2}$, $r=p$, $q=2p$.
	\end{enumerate}
\end{lem}
\begin{proof}
	For Statement 1), suppose $C_{i}$, $r$, $q$, and $V$ satisfy (i) and $\{p,\gamma,c\}$ are as specified. Let $\tilde{V}(x):=C_{2}^{-1}V(x)$. Then $\gamma^{-1}\norm{x}_{2}^{r}=\frac{C_{1}}{C_{2}}\norm{x}_{2}^{r} \leq \tilde{V}(x)$ and $\tilde{V}(x)\leq \frac{C_{2}}{C_{2}}\norm{x}_{2}^{r}
	=\norm{x}_{2}^{r}$.
	It follows that also $-\norm{x}_{2}^{r}\leq -\tilde{V}(x)$, and therefore
	\begin{align*}
		-\norm{x}_{2}^{q}
		&=-\norm{x}_{2}^{r}\norm{x}_{2}^{p}	
		\leq-\tilde{V}(x)\norm{x}_{2}^{p}.
	\end{align*}
	Given that $\nabla V(x)^T f(x)\leq -C_{3}\norm{x}_{2}^{q}$ and $c=\frac{C_{3}}{C_{2}}$, this implies
	\begin{equation*}
		\nabla\tilde{V}(x)^T f(x)
		\leq -\frac{C_{3}}{C_{2}}\norm{x}_{2}^{q}
		\leq-c\tilde{V}(x)\norm{x}_{2}^{p}.
	\end{equation*}
	We find that $\{p,r,\gamma,c,\tilde{V}\}$ satisfy (ii).
	
	For Statement 2), suppose $\{p,r,\gamma,c,\tilde{V}\}$ satisfy (ii) and $\{C_{i},r,q\}$ are as specified. Let $V(x)=\tilde{V}(x)$.
	Then $C_{1}\norm{x}_{2}^{r}\leq V(x)\leq C_{2}\norm{x}_{2}^{r}$.
	Furthermore, we find $-V(x)=-\tilde{V}(x)\leq -\gamma^{-1}\norm{x}_{2}^{r}$ and therefore
	\begin{align*}
		\nabla V(x)^T f(x)
		\leq -cV(x)\norm{x}_{2}^{p}
		\leq -c\gamma^{-1}\norm{x}_{2}^{p+r}
		=-C_{3}\norm{x}_{2}^{q}.
	\end{align*}
	Thus, $\{C_{i},r,q,V\}$ satisfy (i).
	
	For Statement 3), suppose $\{p,r,\gamma,c,\tilde{V}\}$ satisfy (ii) and $M,k$ are as specified. Let $\hat{V}(x)=\tilde{V}(x)^{\frac{p}{r}}$.  Then 
	\begin{equation*}
		M^{-1}\norm{x}_{2}^{p}
		=(\gamma^{-1}\norm{x}_{2}^{r})^{\frac{p}{r}}
		\leq \hat{V}(x)
		\leq \norm{x}_{2}^{p}.
	\end{equation*}
	Since $-\norm{x}_{2}^{p}\leq-\hat{V}(x)$, we also find
	\begin{align*}
		\nabla\hat{V}(x)^T f(x)
		&=\frac{p}{r}\tilde{V}(x)^{\frac{p}{r}-1}\,\nabla \tilde{V}(x)^Tf(x)	\\[-0.4em]
		&\qquad \leq -\frac{cp}{r}\tilde{V}(x)^{\frac{p}{r}}\norm{x}_{2}^{p}
		\leq-k\hat{V}(x)^2.
	\end{align*}
	Thus, $\{p,M,k,\hat{V}\}$ satisfy (iii).
	
	For Statement 4), suppose $\{p,M,k,\hat{V}\}$ satisfy (iii), and let $r=p$, $\gamma=M$, $c=\frac{k}{M}$, and $\tilde{V}=\hat{V}$. Then $\gamma^{-1}\norm{x}_{2}^{r}=M^{-1}\norm{x}_{2}^{p}\leq \tilde{V}(x)$ and $\tilde{V}(x)\leq \norm{x}_{2}^{p}=\norm{x}_{2}^{r}$.
	Furthermore, since $-\hat{V}(x)\leq -M^{-1}\norm{x}_{2}^{p}$, we have
	\begin{align*}
		\nabla \hat{V}(x)^T f(x)	
		&\leq -k\hat{V}(x)^2	
		\leq-c\tilde{V}(x)\norm{x}_{2}^{p}.
	\end{align*}
	Thus, $\{p,r,\gamma,c,\tilde{V}\}$ satisfies (ii).
	
	Finally, Statement 5) follows immediately from Statement 1) and Statement 3) (using $\hat{V}(x)=(C_{2}^{-1}V(x))^{\frac{q}{r}-1}$), and Statement 6) follows immediately from Statement 4) and Statement 2) (using $V(x)=\hat{V}(x)$).
\end{proof}
\vspace*{2mm}

\noindent \textbf{Conservatism and Conjectured Scaling Rule:}
Lemma~\ref{lem:rational_LFs_equivalence} presents three equivalent Lyapunov characterizations of the basic notion of rational stability. Of these conditions, however, only Conditions~(iii) (obtained from Cor.~\ref{cor:rational_stability_LF}) provide an exact characterization of the rate performance, $k$. To see this, consider the equivalence of (ii) and (iii). If the system is stable with rate $k$, then (iii) will be feasible using that $k$. This then implies (ii) is feasible with constant $c=k/M$, which in turn, implies (iii) is feasible with rate $k'=c=k/M$ -- certifying stability with only rate $k/M$. Thus, unless we have gain $M=1$, the optimal rate $c$ from (ii) may be a conservative bound on the decay rate. This same logic applies to Conditions (i), where the best provable rate from (i) would be $k/M^2$. This conservatism is observed numerically in Subsec.~\ref{subsec:examples:rational} and seems to obey the conjectured scaling rule.
Note, however, that this $1/M$ scaling rule is not definitive in that there may exist some $\tilde{V}$ satisfying~(ii) with $c=kM$, and some $V$ satisfying~(i) with $C_{3}=kM^2$, thereby still certifying stability with the true rate $k$.
Nevertheless, we can conclude that the conservatism in these rate tests cannot be greater than the factor $1/M$ or $1/M^2$, respectively.

Lemma~\ref{lem:rational_LFs_equivalence} allows us to conclude that while previous rational stability conditions in the form of (i) are linear in the decision variable and hence testable using SOS, they may provide suboptimal maximum lower bounds on the rational rate performance. However, Lemma~\ref{lem:rational_LFs_equivalence} also offers less conservative intermediate Lyapunov conditions in the form of (ii), which are still linear in $V$ and hence testable using SOS. We describe such an SOS implementation in the following corollary.

\begin{cor}\label{cor:rational_stability_SOS}
	Let $r,d,p\in\N$, $f\in\R_{d_{f}}^{n}[x]$ and $g_{i}\in\R_{d_{i}}[x]$, and define $\Omega:=\{x\in\R^{n}\mid g_{i}(x)\geq 0\}$. If $V$ solves
	\begin{align}\label{eq:rational_stability_SOSP}
		\max_{k\geq0,~\gamma>0,~V}&	&	&k,&	\\
		\tnf{s.t.}&		&	& V(x)-(x^T x)^{r/2} &\in\Sigma_{s,2d}[\Omega],	\notag\\				
		&		&		&\gamma(x^T x)^{r/2} -V(x) &\in\Sigma_{s,2d}[\Omega],	\notag\\[-0.4em]
		&		&		&\hspace*{-2.7cm}-\frac{kr}{p} V(x)(x^T x)^{p/2}-\nabla V(x)^T f(x) &\in\Sigma_{s,d'}[\Omega],\notag
	\end{align}
	where $d'=\max\{2d+p,2d-1+d_{f}\}$, then for any $c\in\R_{+}$ such that $S_{c}(V)\subseteq\tnf{int}(\Omega)$, $f$ is rationally stable on $S_{c}(V)$ with rate $k$ and gain $M:=\gamma^{\frac{p}{r}}$.  
\end{cor}
\begin{proof}
	Suppose the optimization program is feasible, and let $M=\gamma^{\frac{p}{r}}$. For any $c\in\R_{+}$ such that $S_{c}(V)\subseteq \tnf{int}(\Omega)$, we have that $\nabla V(x)^T f(x)\leq 0$ for all $x\in S_{c}(V)$, implying $S_{c}(V)$ is forward invariant for $f$. Define $\tilde{V}(x):=\gamma^{-1}V(x)$. Since $V$ satisfies~\eqref{eq:rational_stability_SOSP}, $\tilde{V}$ will satisfy $\gamma^{-1}\norm{x}_{2}^{r}\leq \tilde{V}(x)\leq \norm{x}_{2}^{r}$ and $\nabla \tilde{V}(x)^T f(x)\leq -\frac{kr}{p} \tilde{V}(x)\norm{x}^{p}$
	for all $x\in S_{c}(V)\subseteq \Omega$.
	By Prop.~\ref{lem:rational_LFs_equivalence}, it follows that there exists a function $\hat{V}:S_{c}(V)\to\R_{+}$ that satisfies $M^{-1}\norm{x}_{2}^{p}\leq \hat{V}(x)\leq \norm{x}_{2}^{p}$ as well as $\nabla \hat{V}(x)^T f(x)\leq -k\hat{V}(x)^2$.
	By Cor.~\ref{cor:rational_stability_LF}, this implies that $f$ is rationally stable on $S_{c}(V)$ with rate $k$ and gain $M$.
\end{proof}

The SOS conditions of Cor.~\ref{cor:rational_stability_SOS} will be applied in the numerical experiments described in Subsection~\ref{subsec:examples:rational}.

\subsection{Conditions for Finite-Time Decay Rate}\label{subsec:SOS_tests:finite_time}

Finally, consider the pointwise-in-time notion of finite-time stability from Defn.~\ref{defn:finite_time_stability}, characterized as $\beta_{\tnf{f}}$-stability in Lemma~\ref{lem:finite_time_stability}. While $\beta_{\tnf{f}}(y,t)$ is not differentiable everywhere, $D_{t}^{+}\beta_{\tnf{f}}(y,t)|_{t=0}$ exists, and we may define
\begin{equation}\label{eq:finite_time_rho}
	\rho_{\tnf{f}}(y):=D_{t}^{+}\beta_{\tnf{f}}(y,t)|_{t=0}
	=\begin{cases}
		-1,	&	y>0,	\\
		0,	&	y=0.
	\end{cases}
\end{equation}
Using this definition, we can apply Thm.~\ref{thm:LF_sufficiency_diff} and Thm.~\ref{thm:LF_necessity_diff_FT} to obtain the following necessary and sufficient conditions for finite-time stability.

\begin{cor}\label{cor:finite_time_stability_LF}
	For $\Omega\subseteq\R^{n}$, let $f:\Omega\to\R^{n}$ be continuous, $f(0)=0$, and $G\subseteq \Omega$ be forward invariant for $f$. Suppose $T_{f}(x):=\sup\{t\geq 0\mid \norm{\phi_{f}(x,t)}_{p}^{\eta}>0\}$ is continuous for $x\in G\setminus\{0\}$. Then, $f$ is finite-time stable on $G$ with rate $k\geq 0$ and gain $M\geq 1$ if and only if there exist $p>0$, $\eta\in(0,1)$ and a continuous function $V:G\to\R_{+}$ such that
	\begin{align*}
		M^{-1}\norm{x}_{p}^{\eta}&\leq V(x) \leq \norm{x}_{p}^{\eta},	&	&\forall x\in G,	\\
		\dot{V}(x)&\leq -k,&	&\forall x\in G\setminus\{0\},
	\end{align*}
	where $\dot{V}(x):=D_{t}^{+}V(\phi_{f}(x,t))|_{t=0}$.
\end{cor}
\begin{proof}
	Define $\rho_{\tnf{f}}(y)$ as in Eqn.~\eqref{eq:finite_time_rho}. Then, $\beta_{\tnf{f}}$ is the unique and continuous function $\beta$ satisfying $\partial_{t}\beta(y,t)=\rho_{\tnf{f}}(\beta(y,t))$ with $\beta(y,0)=y$. Since $-k=k\rho_{\tnf{f}}(V(x))$ for all $x\in G\setminus\{0\}$, and $\dot{V}(0)=D_{t}^{+}V(\phi_{f}(0,t))|_{t=0}=0=\rho_{\tnf{f}}(0)$ (since $f(0)=0$), by Thm.~\ref{thm:LF_sufficiency_diff} and Thm.~\ref{thm:LF_necessity_diff_FT}, it follows that $f$ is $\beta_{\tnf{f}}$-stable on $G$ with respect to $\alpha(x):=\norm{x}_{p}^{\eta}$ and with rate $k$ and gain $M$ if and only if there exists a function $V$ satisfying the proposed conditions. By Lemma~\ref{lem:finite_time_stability}, the result follows.
\end{proof}

Cor.~\ref{cor:finite_time_stability_LF} presents necessary and sufficient Lyapunov conditions for testing finite-time rate performance. These conditions are essentially identical to certain classical Lyapunov conditions for finite-time stability from~\cite{roxin1966FTS}, which in recent years have largely been supplanted by the conditions from~\cite{bhat2000FiniteTimeStability}, taking the form $\dot{V} \leq -\frac{1}{1-\eta} V^{\eta}$ for $\eta \in (0,1)$. Both sets of conditions are necessary and sufficient, and equivalence may be established using a transformation of the form $V(x)\mapsto V(x)^{\eta}$.
However, testing either of these Lyapunov conditions using SOS is complicated by the fact that finite-time stable systems are not defined by polynomial vector fields. To resolve this, we use the following result from~\cite{kisole2025FiniteTimeSOS}, performing a variable substitution to express the conditions in terms of polynomials. To formulate this result, let $\sign(x)$ and $|x|$ denote the element-wise sign and absolute value of $x\in\R^{n}$, respectively, and denote by $\cdot$ the element-wise product, so that e.g. $\sign(x)\cdot|x|=x$. For $p\in\R$, let $x^{p}$ denote the elementwise power of $x$, so that $[x^{p}]_{i}=x_{i}^{1/p}$.

\begin{lem}[Thm.~6 from~\cite{kisole2025FiniteTimeSOS}]\label{lem:finite_time_stability_sengi}
	Let $f:\R^{n}\to\R^{n}$ be continuous, and $\Omega\subseteq\R^{n}$. For $r\in\N$, let $\tilde{\Omega}:=\{\sign(x)\cdot|x|^{1/r}\mid x\in \Omega\}$, and suppose there exists $\tilde{p},\tilde{k},\tilde{M},\tilde{r}\in\R_{+}$ and a differentiable function $\tilde{V}:\tilde{\Omega}\to\R_{+}$ such that $\tilde{V}(0)=0$ and $0<\tilde{V}(z)\leq\tilde{M}\norm{z}_{2}^{2r}$ for all $z\in\tilde{\Omega}\setminus\{0\}$, and furthermore
	\begin{align*}
		\ \\[-1.1\baselineskip]
		\nabla \tilde{V}(z)^T \tilde{f}(z)\leq -\tilde{k}\norm{z}_{2}^{\tilde{p}},\quad \forall z\in\tilde{\Omega},
	\end{align*}
	where $\tilde{f}(z)=\frac{1}{r}f(\sign(z)\cdot|z|^{r})\cdot|z|^{1-r}$.
	Then, $V(x):=\tilde{V}(\sign(x)\cdot|x|^{1/q})$ is differentiable on $\Omega\setminus\{0\}$ and satisfies
	\begin{align*}
		\ \\[-1.15\baselineskip]
		\nabla V(x)^Tf(x)\leq -(\tilde{k}/\tilde{M}) V(x)^{\tilde{p}/\tilde{r}},\qquad \forall x\in\Omega\setminus\{0\}.
	\end{align*}
\end{lem}
\vspace*{2mm}

Using this result, we obtain the following sufficient conditions for finite-time rate performance.

\begin{prop}\label{prop:finite_time_stability_LF_sengi}
	For $\Omega\subseteq\R^{n}$, $f:\Omega\to\R^{n}$, and $r>0$, let $\tilde{\Omega}:=\{\sign(x)\cdot|x|^{\frac{1}{r}}\mid x\in \Omega\}$ and define $\tilde{f}_{r}:\tilde{\Omega}\to\R^{n}$ by $\tilde{f}_{r}(z):=\frac{1}{r}f(\sign(z)\cdot|z|^{r})\cdot|z|^{1-r}$. For given $\eta\in(0,1)$, $k\geq 0$ and $M\geq 1$, if there exists $\tilde{V}:\tilde{\Omega}\to\R_{+}$ such that
	\begin{align*}
		M^{-\frac{2}{\eta}}\norm{z}_{2}^{2r}\leq \tilde{V}(z) &\leq \norm{z}_{2}^{2r},	\\
		\nabla\tilde{V}(z)^T \tilde{f}_{r}(z)&\leq -\frac{2k}{\eta}\norm{z}_{2}^{(2-\eta) r},\qquad \forall z\in\tilde{\Omega},
	\end{align*}
	then, for any $c\in\R_{+}$ such that $S_{c}(V)\subseteq\tnf{int}(\Omega)$ where $V(x):=\tilde{V}(\sign(x)\cdot|x|^{\frac{1}{r}})$, $f$ is finite-time stable on $S_{c}(V)$ with rate $k$ and gain $M$.
\end{prop}
\begin{proof}
	Suppose that $\tilde{V}$ satisfies the proposed conditions, and let the associated $V$ be as defined. Then, by Lemma~\ref{lem:finite_time_stability_sengi} (letting $\tilde{p}=(2-\eta)r$, $\tilde{k}=\frac{2k}{\eta}$, $\tilde{r}=2r$, $\tilde{M}=1$), this $V$ satisfies $\nabla V(x)^Tf(x)\leq -\frac{2k}{\eta} V(x)^{1-\frac{\eta}{2}}$ for all $x\in \Omega\setminus\{0\}$.
	Defining $z(x):=\sign(x)\cdot|x|^{\frac{1}{r}}$ for $x\in \Omega$, it follows that
	\begin{equation*}
		V(x)=\tilde{V}(z(x))\leq\norm{z(x)}_{2}^{2r}
		=\bl\||x|^{1/r}\br\|_{2}^{2r}
		=\norm{x}_{2/r}^2,
	\end{equation*}
	and, similarly, $V(x)\geq M^{-\frac{2}{\eta}}\norm{x}_{2/r}^2$.
	Now, let $\hat{V}(x):=V(x)^{\frac{\eta}{2}}$. Then, $M^{-1}\norm{x}_{2/r}^{\eta}\leq \hat{V}(x)\leq \norm{x}_{2/r}^{\eta}$ for all $x\in\Omega$. Furthermore, for all $x\in\Omega\setminus\{0\}$, we have
	\begin{align*}
		\nabla\hat{V}(x)^T f(x)
		&=\frac{\eta}{2}V(x)^{\frac{\eta}{2}-1}\nabla V(x)^T f(x)	\\
		&\quad\leq -kV(x)^{\frac{\eta}{2}-1}V(x)^{1-\frac{\eta}{2}}
		=-k.
	\end{align*}
	Finally, for any $c\in \R_{+}$ such that $S_{c}(V)\subseteq\tnf{int}(\Omega)$, we have that $\nabla V(x)^T f(x)\leq 0$ for all $x\in S_{c}(V)$, and thus $S_{c}(V)$ is forward invariant for $f$.
	By Cor.~\ref{cor:finite_time_stability_LF}, it follows that $f$ is finite-time stable on $S_{c}(V)$ 
	with rate $k$ and gain $M$.
\end{proof}

Although the Lyapunov conditions in Prop.~\ref{prop:finite_time_stability_LF_sengi} do not always admit an obvious SOS formulation, they can often be posed as an SOS program through multiplication by a suitable factor $h(z)$, as in the following corollary.

\begin{cor}\label{cor:finite_time_stability_SOS}
	For given $\eta\in(0,1)$ and $f:\R^{n}\to\R^{n}$, suppose there exist $r\in\N$ and $\lambda\in\R^{n}$ such that, letting $\tilde{f}_{r}(z):=\frac{1}{r}f(\sign(z)\cdot|z|^{r})\cdot|z|^{1-r}$ and $h(z):=\prod_{i=1}^{n}|z_{i}|^{\lambda_{i}}$, we have $h(z)\tilde{f}_{r}(z)\in\R^{n}_{d_{f}}[z]$ and $h(z)\norm{z}_{2}^{(2-\eta)r}\in\R_{d_{h}}[z]$. Furthermore, let $\Omega:=\{x\in\R^{n}\mid g_{i}(x)\geq 0\}$ be such that $\tilde{\Omega}:=\{z\in\R^{n}\mid \tilde{g}_{i}(z)\geq 0\}$ is semi-algebraic with $\tilde{g}_{i}(z):=g_{i}(\sign(z)\cdot|z|^{r})$. For any $d\in\N$, if $\tilde{V}$ solves
	\begin{align}\label{eq:finite_time_stability_SOSP}
		\max_{k\geq0,~\gamma>0,~\tilde{V}}&	&	&k, &	\\
		\tnf{s.t.}&		&	& \gamma \tilde{V}(z)-(z^T z)^{r}   &\in\Sigma_{s,2d}[\tilde{\Omega}],	\notag\\				
		&		&		& \hspace*{0cm}(z^T z)^{r} -\tilde{V}(z)  &\in\Sigma_{s,2d}[\tilde{\Omega}],	\notag\\[-0.4em]
		&		&		& \hspace*{-2cm}h(z)\bbl[-\frac{2k}{\eta} (z^Tz)^{(2-\eta)r/2}-\nabla \tilde{V}(z)^T \tilde{f}_{r}(z)\bbr] &\in\Sigma_{s,d'}[\tilde{\Omega}],	\notag
	\end{align}
	where $d':=\max\{d_{h},2d-1+d_{f}\}$, then, letting $V(x)=\tilde{V}(\sign(x)\cdot|x|^{\frac{1}{r}})$, for any $c\in\R_{+}$ with $S_{c}(V)\subseteq\tnf{int}(\Omega)$, $f$ is finite-time stable on $S_{c}(V)$ with rate $k$ and gain $M=\gamma^{\frac{\eta}{2}}$.
%
\end{cor}
\begin{proof}
	Suppose the optimization program is feasible. Then the solution $\tilde{V}$ satisfies all conditions of Prop.~\ref{prop:finite_time_stability_LF_sengi} with $M=\gamma^{\frac{\eta}{2}}$.
	By that proposition, we conclude that $f$ is finite-time stable on $S_{c}(V)$ with rate $k$ and gain $M$.
\end{proof}


Cor.~\ref{cor:finite_time_stability_SOS} shows how finite-time stability of a class of rational vector fields
can be tested using SOS programming. In the following section, we will apply this SOS program, as well as those from Cor.~\ref{cor:exponential_stability_SOS} and Cor.~\ref{cor:rational_stability_SOS}, to verify exponential, rational, and finite-time stability of several examples, and compute lower bounds on the associated rate performance.

\section{Numerical Analysis of Local and Global Rate Performance using SOS}\label{sec:examples}

In Section~\ref{sec:SOS_tests}, we have proposed SOS implementations of the necessary and sufficient Lyapunov conditions for rate and gain performance in Thm.~\ref{thm:LF_necessity_diff_FT} as applied to exponential, rational, and finite-time performance metrics.
Now, we evaluate the efficacy and accuracy of these SOS Programs (SOSPs) for computing maximum lower bounds on both global and local rate performance, considering separately the cases of exponential, rational, and finite-time stability.

In each case, we first consider a globally $\beta$-stable vector field for the corresponding function $\beta$.
Using the associated SOSP from Section~\ref{sec:SOS_tests} and performing a bisection search on the parameter $k$, we find a largest lower bound on the global rate performance.
We verify through simulation that this bound is tight.
For the rational stability example, we compare the obtained lower bound on the rate performance to that established using the more standard Lyapunov conditions for rational stability (Conditions~(i) in Lemma~\ref{lem:rational_LFs_equivalence}), verifying the conjectured conservatism in this bound from Subsec.~\ref{subsec:SOS_tests:rational}.

Having addressed global stability, we next show that the proposed SOSPs may also be used for local stability analysis -- computing the rate performance of a vector field on a given semialgebraic domain. We apply each test to a reverse-time van der Pol system (modified for the finite-time stability case), verifying stability on a semi-algebraic set parameterized by a radius parameter, $R$. Computing a maximal lower bound on the rate performance for several values of $R$, we verify that this bound decreases to zero as $R$ increases. For exponential stability, we also show that the largest lower bound on the rate performance matches the rate of decay of the linearized system as $R$ approaches $0$ -- as expected.

Finally, we show that the proposed SOS tests may also be used to compute regions of $\beta$-stability rate performance -- i.e. forward invariant sets on which the vector field is $\beta$-stable with a given rate. Specifically, considering again the (modified) van der Pol equation, we plot the largest closed sublevel sets of the Lyapunov function certifying each stability notion, for several values of the rate $k$. We show that, as $k$ tends to zero -- and $\beta$-stability reduces to asymptotic stability -- the associated region of performance also approaches the region of asymptotic attraction of the system.

The SOSP for each example in this section is parsed using the Matlab toolbox SOSTOOLS~\cite{papachristodoulou2021SOSTOOLS}, and solved using the semidefinite programming solver Mosek~\cite{mosek}. All simulations are performed using the Matlab function \texttt{ode23}.



\subsection{Exponential Stability}\label{subsec:examples:exponential}

We start with exponential stability metrics. As discussed in Subsec.~\ref{subsec:SOS_tests:exponential}, we anticipate that the SOS implementation should have little, if any conservatism. To test this hypothesis, we first apply the SOSP from Cor.~\ref{cor:exponential_stability_SOS} to verify global stability of a modified Lorenz system, obtaining accurate bounds on the rate and gain performance. We then use the SOSP to test the local rate performance of a van der Pol equation, and compute regions of exponential performance.

\subsubsection{Global Exponential Stability of Lorenz System}

As a first example, consider a Lorenz system defined by
\begin{equation}\label{eq:Lorenz_ODE}
	\dot{x}(t)=f(x(t)):=\bmat{8(x_{2}(t)-x_{1}(t))\\ x_{1}(t)(0.5-x_{3}(t))-x_{2}(t)\\ x_{1}(t)x_{2}(t) -4 x_{3}(t)},
\end{equation}
where we have selected parameters for which the system is globally exponentially stable.
Solving the SOSP in Cor.~\ref{cor:exponential_stability_SOS} with $\Omega = \R^n$ and $d \in \{1,\hdots,8\}$,  global exponential stability can be verified with greatest lower bound $k=0.4688$ on the rate, which does not improve when increasing $d$. Fixing this rate, the associated least upper bounds on gain performance, $M_{d}$, are presented in Table~\ref{tab:exponential_stability_Lorenz}.

\begin{table}[b]
	\setlength{\tabcolsep}{5.25pt}
	\vspace*{-0.4cm}
	\begin{tabular}{c|cccccccc}
		$d$     & 1      & 2      & 3      & 4      & 5      & 6      & 7      & 8\\\hline
		$M_{d}$ & 3.952 & 1.732 & 1.504 & 1.469 & 1.452 & 1.422 & 1.415 & 1.415
	\end{tabular}
	\caption{Gains $M_{d}$ for which global exponential stability of the Lorenz system in~\eqref{eq:Lorenz_ODE} can be verified with rate $k=0.4688$ using the SOSP~\eqref{eq:exponential_stability_SOSP}, for several values of $d$.}\label{tab:exponential_stability_Lorenz}
\end{table}

For comparison, the solution to the Lorenz system is also simulated from $t=0$ up to $t=10$, starting with initial state $x(0)=\tnf{e}_{2}:=(0,1,0)^T$. The norm of this solution is displayed with logarithmic scale in Fig.~\ref{fig:exponential_stability_Lorenz_bound}, along with the associated exponential bounds $M_{d}e^{-0.4688t}$ for several values $M_{d}$ as in Tab.~\ref{tab:exponential_stability_Lorenz}. From Fig.~\ref{fig:exponential_stability_Lorenz_bound} and Tab.~\ref{tab:exponential_stability_Lorenz}, it is clear that the bound on gain performance becomes tighter as the degree $d$ of the polynomial Lyapunov function is increased. Moreover, the obtained decay rate $k=0.4688$ also closely matches the average decay rate of the solution from $t=8$ to $t=10$, computed as $k_{\tnf{sim}}=\frac{\ln(\norm{\phi_{f}(\tnf{e}_{2},8)}_{2})-\ln(\norm{\phi_{f}(\tnf{e}_{2},10)}_{2})}{10-8}=0.4689$.


\begin{figure}[t]
	\centering
	\includegraphics[width=1.0\linewidth]{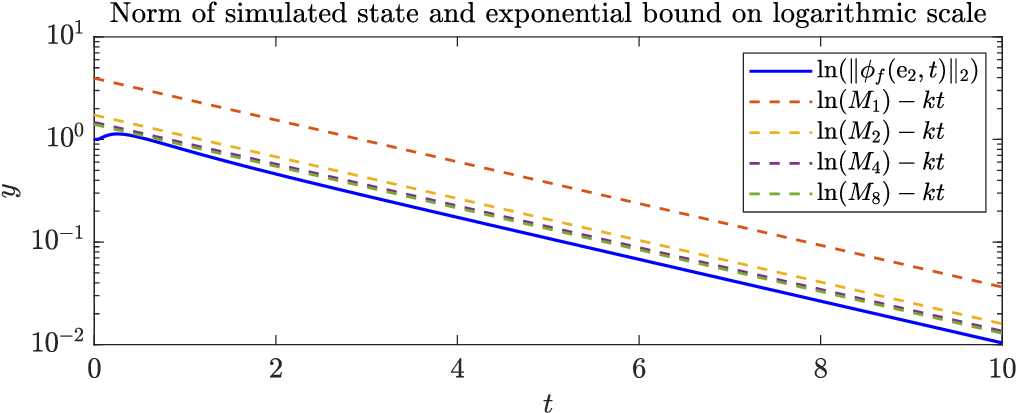}
	\vspace*{-0.6cm}
	\caption{Norm $\norm{\phi_{f}(x,t)}_{2}$ of simulated solution to the Lorenz system in~\eqref{eq:Lorenz_ODE} on logarithmic scale, starting with $x=\tnf{e}_{2}:=(0,1,0)^T$. The exponential bounds $M_{d}e^{-0.4688t}$ for $M_{d}$ as in Tab.~\ref{tab:exponential_stability_Lorenz} are also plotted for $d\in\{1,2,4,8\}$.}
	\label{fig:exponential_stability_Lorenz_bound}	
	\vspace*{-0.4cm}
\end{figure}

\subsubsection{Exponentially Stable Regions of van der Pol System}
As a second example, consider the reverse-time van der Pol system,
\begin{equation}\label{eq:vdP_ODE}
	\dot{x}(t)=f(x(t))=\bmat{-x_{2}(t)\\
		-(1-x_{1}(t)^2)x_{2}(t)+x_{1}(t)}.
\end{equation}
We examine exponential stability of this vector field using the conditions of SOSP~\eqref{eq:exponential_stability_SOSP} with $\Omega_{R}:=\{x\in\R^{2}\mid R^2-x^T x\geq0\}$
 and $d=8$. For $R$ ranging from $0.01$ to $1.75$ (with spacing $\Delta R=0.015$), we use bisection to find the greatest lower bound on rate $k_R$ while leaving $M$ unconstrained. For each such $k_R$, we then compute the least upper bound on gain performance, $M_R$, for which the given rate is feasible. The resulting values of $k_R$ and $M_R$ are displayed in Fig.~\ref{fig:exponential_stability_vdP_rate} as a function of the radius, $R$. Fig.~\ref{fig:exponential_stability_vdP_rate} shows that, for $R$ close to 0, exponential decay can be verified with rate $k$ close to 0.5, matching the decay rate of the linearized system. As $R$ approaches the inner radius of the van der Pol limit cycle, the decay rate decreases to 0.

\begin{figure}[b]
	\centering
	\vspace*{-0.3cm}
	\includegraphics[width=1.0\linewidth]{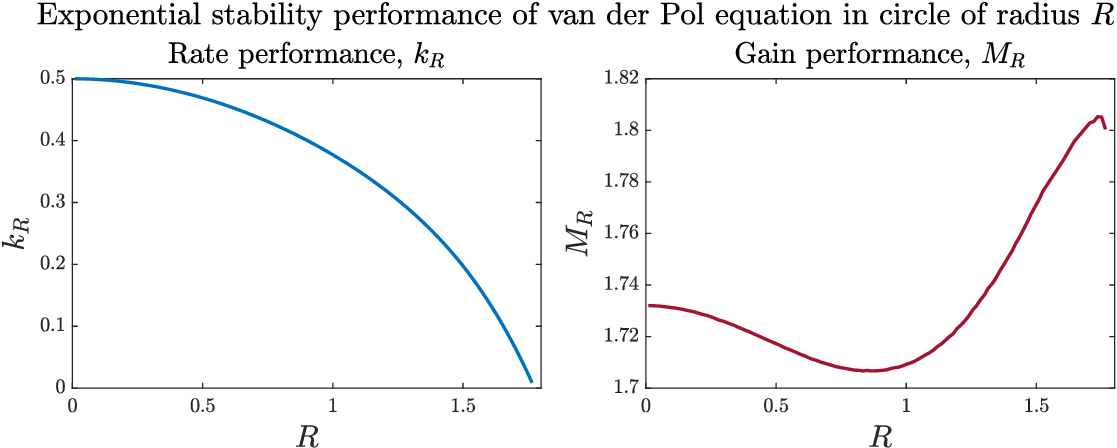}
	\vspace*{-0.6cm}
	\caption{Greatest lower bound on rate $k_R$ and associated least upper bound on gain $M_R$ for exponential stability of the van der Pol system in~\eqref{eq:vdP_ODE} as a function of radius, $R$, of the centered circular domain on which the conditions were verified as computed by SOSP~\eqref{eq:exponential_stability_SOSP} with $d=8$.}
	\label{fig:exponential_stability_vdP_rate}	
\end{figure}

For $k\in\{0,0.025,0.1,0.2,0.35,0.48\}$, the greatest closed level sets of the computed Lyapunov function certifying stability with rate $k$ are displayed in Fig.~\ref{fig:exponential_stability_vdP_ROA}, corresponding to forward-invariant sets of the van der Pol equation with rate performance at least $k$. The level set of a Lyapunov function for the value of $k=0$ is displayed as well, along with a simulation-based estimate of the region of attraction of the van der Pol equation (black dashed line). Fig.~\ref{fig:exponential_stability_vdP_ROA} shows that the region of performance grows as the rate decreases --- approaching the region of asymptotic attraction as $k$ approaches $0$.

\subsection{Rational Stability}\label{subsec:examples:rational}

Having demonstrated that the SOSP for exponential-type performance in Cor.~\ref{cor:exponential_stability_SOS} has no apparent conservatism, we now consider the case of rational-type rate and gain metrics. In the rational case, we will use the sufficient SOS conditions for rational rate and gain performance obtained from Lemma~\ref{lem:rational_LFs_equivalence}~(ii) as implemented in Cor.~\ref{cor:rational_stability_SOS}. As discussed in Subsection~\ref{subsec:SOS_tests:rational}, we anticipate that these conditions for rate performance may be conservative with a scaling factor of up to $1/M$ where $M$ is the gain performance. To evaluate these scalings, we will compare the obtained bounds on rate performance to both numerical simulation and to bounds obtained from the ``naive'' SOS test obtained from Lemma~\ref{lem:rational_LFs_equivalence}~(i).
The Lyapunov Conditions (i) are tightened to SOS constraints as was done in Cor.~\ref{cor:rational_stability_SOS}, obtaining a constraint on the derivative of the variable $V$ as
\begin{equation}\label{eq:rational_SOS_classical}
	-\frac{r}{p}M^{\frac{r}{p}} k(x^Tx)^{(r+p)/2}-\nabla V(x)^T f \in \Sigma_{d'}[\Omega],
\end{equation}
where we will use the same values of $M,p,r,d,d'$ as used to evaluate rate performance with Lyapunov Conditions~(ii). For both Lyapunov conditions, bisection is used to find the largest value of $k$ for which the respective SOSP is feasible.

As in the case of exponential stability, we first find a greatest lower bound on the rate performance of a globally rationally stable system, and subsequently test the local rate performance and compute regions of performance for a van der Pol system.

\begin{figure}[t]
	\centering
	\includegraphics[width=1.0\linewidth]{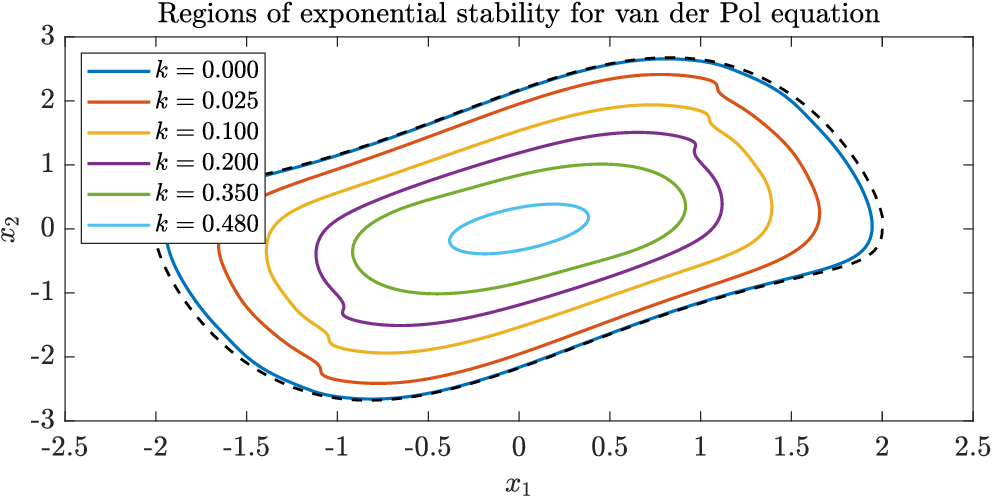}
	\vspace*{-0.6cm}
	\caption{Forward-invariant domains of the reverse-time van der Pol Equation~\eqref{eq:vdP_ODE} in which exponential stability can be verified with several rates $k$ by solving the SOSP~\eqref{eq:exponential_stability_SOSP}. The simulated region of asymptotic attraction is displayed using a black dashed line.}
	\label{fig:exponential_stability_vdP_ROA}	
	\vspace*{-0.4cm}
\end{figure}

\subsubsection{A Globally Rationally Stable Vector Field}

Consider the following ODE\\[-0.75\baselineskip]
\begin{equation}\label{eq:rational_stability_example_ODE}
	\dot{x}(t)=f(x(t)):=\bmat{-x_{1}(t)^3-x_{2}(t)^3 \\ x_{1}(t)^3 -x_{2}(t)^3 -3x_{3}(t) \\ 2x_{2}(t)-x_{3}(t)^3}.
\end{equation}
Similar to the case of exponential stability, we use bisection to find the largest $k$ such that SOSP~\eqref{eq:rational_stability_SOSP} is feasible on $\Omega=\R^{n}$, using $p=2$ and $r=2d$ for $d\in\{1,\hdots,9\}$.
This approach
establishes global rational stability with greatest lower bounds on rate performance $k_{\tnf{(ii)}}$ as presented in Table~\ref{tab:rational_stability_3D}. For each value of $d$, fixing the associated rate performance $k=k_{\tnf{(ii)}}$, we obtain a least upper bound of $M=1.5$ on the gain performance, which does not change for different values of $d$.

\begin{table}[b]
	\setlength{\tabcolsep}{4pt}
		\vspace*{-0.4cm}
	\begin{tabular}{c|ccccccccc}
		$d$     & $1$ & $2$ & $3$ & $4$ & $5$ & $6$ & $7$ & $8$ & 9   \\\hline
		$k_{\tnf{(i)}}$ & 0.418 & 0.306 & 0.261 & 0.195 & 0.145 & 0.107 & 0.076 & 0.053 & 0.038
		\\
		$k_{\tnf{(ii)}}$ & 0.563 & 0.616 & 0.622 & 0.622 & 0.625 & 0.625 & 0.625 & 0.626 & 0.626
	\end{tabular}
	\caption{Greatest lower bounds on rational rate performance of ODE~\eqref{eq:rational_stability_example_ODE} for $\Omega:=\R^{3}$, established with Lyapunov Conditions~(\tnf{i}) ($k_{\tnf{(i)}}$) and (\tnf{ii}) ($k_{\tnf{(ii)}}$) of Lemma~\ref{lem:rational_LFs_equivalence}. In both cases, the associated SOSP is solved with $r=2d$, $p=2$, and $M=1.5$. The simulation-based estimate of the rate performance is $k_{\tnf{sim}}=1.048$.}\label{tab:rational_stability_3D}
\end{table}

Now, to test the conservatism scaling rules conjectured in Subsection~\ref{subsec:SOS_tests:rational}, we fix $M=1.5$ and apply Lyapunov Conditions~(i) for rational stability from Lemma~\ref{lem:rational_LFs_equivalence}, obtaining greatest lower bounds on rate performance $k_{\tnf{(i)}}$ as in Table~\ref{tab:rational_stability_3D}.
Notably, the lower bound on the rate computed with Conditions~(i) decreases with the degree $d$. We conjecture that this decrease in performance may be due to the global nature of the test and associated restrictions on the degree structure of the polynomials used in the test (e.g. $V$ must be a homogeneous polynomial). However, comparing the greatest lower bound $k_{\tnf{(i)}}=0.418$ (for $d=1$) to the greatest lower bound $k_{\tnf{(ii)}}=0.626$ (for $d=9$), we find that the conservatism introduced by Conditions~(i) compared to Conditions~(ii) indeed manifests as a scaling of roughly $\frac{1}{M}=\frac{2}{3}$ of the bound on the rate performance.

To evaluate conservatism of the rate bound $k_{\tnf{(ii)}}=0.626$, we now compare this rate to that obtained from numerical simulation. Specifically, we simulate the ODE for 2500 distinct initial conditions along a sphere of radius $R=10^{4}$ centered at the origin. Using bisection to find the largest value of $k$ for which all simulated solutions satisfy $\norm{x(t)}_{2}^{2}\leq M\beta_{\tnf{r}}(\norm{x(0)}_{2}^{2},kt)$ for $M=1.5$ and $t\in[0,1000]$, we obtain an estimate of the rate performance as $k_{\tnf{sim}}=1.048$. Compared to this estimate, the lower bound $k_{\tnf{(ii)}}=0.626$ appears to be slightly more conservative than the conjectured scaling factor of $\frac{1}{M}=\frac{2}{3}$.

\subsubsection{Rationally Stable Regions of van der Pol System}

Next, we examine local bounds on rate performance for the same reverse-time van der Pol Equation~\eqref{eq:vdP_ODE} used for exponential stability. To initialize this test, we solve the SOSP in Cor.~\ref{cor:rational_stability_SOS} using $p=2$, $r=1$, $d=9$, $M=3$, and $\Omega:=\{x\mid x_{1}^2-x_{1}x_{2}+x_{2}^2\leq R^2\}$ with $R\in\{0.3,0.5,...,1.7\}$. Associated greatest values $k$ for which SOSP~\eqref{eq:rational_stability_SOSP} is feasible are given in Table~\ref{tab:rational_stability_vdP}. In addition, fixing $k=0$, feasibility of SOSP~\eqref{eq:rational_stability_SOSP} is also verified for $R=2.365$. For each $k$ and $R$, the associated largest closed level set, $G_{k}$, of the obtained Lyapunov function is computed -- certifying rate performance of at least $k$ on $G_{k}$. A subset of these regions is displayed in Fig.~\ref{fig:rational_stability_vdP_ROA}, along with a simulation-based approximation of the region of asymptotic attraction (black dashes). From Fig.~\ref{fig:rational_stability_vdP_ROA} it is clear that, as expected, the region of rational stability with rate performance $k$ approaches the region of asymptotic attraction as $k\to 0$.

Next, having established that each of the $G_k$ are forward invariant, we may obtain improved estimates of rate performance on these regions by repeating the set of initial tests, but using $\Omega=G_{k}$. For each such $k$ and $G_k$, the resulting maximum lower bound on the rate performance, $k_{\tnf{(ii)}}$, is presented in Table~\ref{tab:rational_stability_vdP}, showing uniform improvement in computed maximum lower bound.
In addition, to test the conjectured conservatism scaling rule, greatest lower bounds on the rate performance on each region are also computed using Lyapunov Conditions~(i) from Lemma~\ref{lem:rational_LFs_equivalence}. The obtained largest lower bounds on the rate performance, $k_{\tnf{(i)}}$, are included in Table~\ref{tab:rational_stability_vdP}.

\begin{table}[b]
	\setlength{\tabcolsep}{5pt}
	%
	%
	\vspace*{-0.1cm}
	\begin{tabular}{c|cccccccc}
		$R$     & $0.3$ & $0.5$      & $0.7$      & $0.9$      & $1.1$      & $1.3$      & $1.5$ & $1.7$    \\\hline
		$k$ & 16.886 & 5.970 & 2.895 & 1.562 & 0.870 & 0.510 & 0.335 & 0.198 \\\hline
		$k_{\tnf{(i)}}$ & 10.312 & 3.675 & 1.793 & 1.025 & 0.646 & 0.515 & 0.321 & 0.167 \\
		$k_{\tnf{(ii)}}$ & 18.429 & 6.561 & 3.195 & 1.806 & 1.133 & 0.904 &	0.558 & 0.290 \\
		$k_{\tnf{sim}}$ & 36.674 & 13.060 & 6.327 & 3.509 & 2.184 & 1.743 & 1.054 & 0.535
	\end{tabular}
	\caption{Greatest values of $k$ for which SOSP~\eqref{eq:rational_stability_SOSP} is feasible for ODE~\eqref{eq:vdP_ODE} on $\Omega:=\{x\mid x_{1}^2-x_{1}x_{2}+x_{2}^2\leq R^2\}$, using $r=1$, $d=9$, $p=2$, and $M=3$. Greatest lower bounds on the rational stability rate performance on the associated invariant set $G_{k}$ in Fig.~\ref{fig:rational_stability_vdP_ROA} are provided as well, established with Lyapunov Conditions (\tnf{i}) ($k_{\tnf{(i)}}$) and (\tnf{ii}) ($k_{\tnf{(ii)}}$) of Lemma~\ref{lem:rational_LFs_equivalence}.
		Simulation-based estimates ($k_{\tnf{sim}}$) of the rate performance on $G_{k}$ are also given.}\label{tab:rational_stability_vdP}
\end{table}

Finally, to verify accuracy of the established lower bounds on the rate performance, for each of the computed regions $G_{k}$, we numerically simulate the ODE in Eqn.~\eqref{eq:vdP_ODE} for 200 distinct initial conditions placed at intervals along the boundary of this region. For each region, we then find the largest rate $k_{\tnf{sim}}$ such that all of the simulated solutions satisfy  $\norm{x(t)}_{2}^{2}\leq M\norm{x(0)}_{2}^{2}/(1+k_{\tnf{sim}}\norm{x(0)}_{2}^2 t)$ for $t\in[0,100]$. The resulting values of $k_{\tnf{sim}}$ are given in Table~\ref{tab:rational_stability_vdP}, providing a numerical estimate of the true rate performance on the region $G_{k}$, for each $k\in\{0.3,0.5\hdots,1.7\}$.

The results suggest that the conservatism in $k_{\tnf{(ii)}}$ scales as roughly $\frac{1}{2}$ compared with $k_{\tnf{sim}}$ -- slightly less than the conjectured scaling of $\frac{1}{M}=\frac{1}{3}$. In addition, the conservatism in $k_{\tnf{(i)}}$ compared to $k_{\tnf{(ii)}}$ likewise scales roughly as $\frac{1}{2}$, again slightly less than the conjectured scaling factor of  $\frac{1}{3}$.

\begin{figure}[t]
	\centering
	\includegraphics[width=1.0\linewidth]{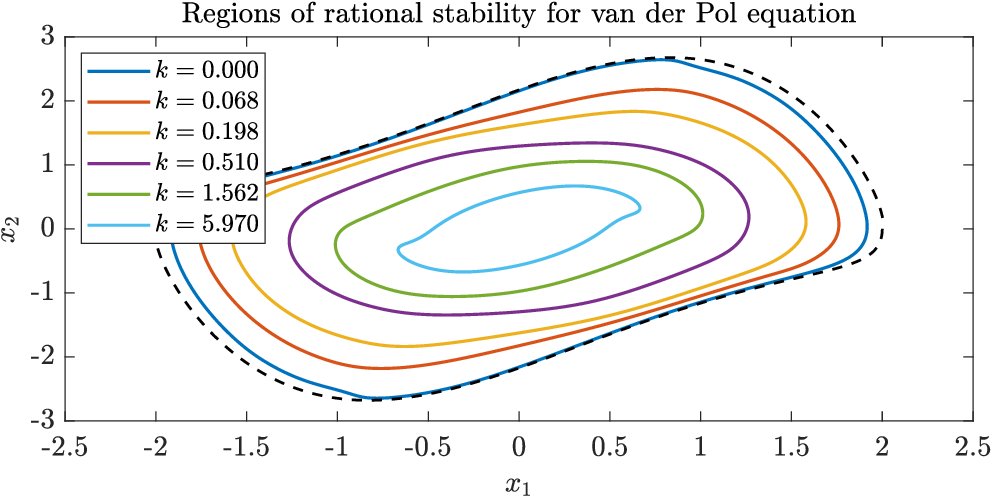}
	\vspace*{-0.6cm}
	\caption{Forward-invariant regions, $G_{k}$, of the reverse-time van der Pol Equation~\eqref{eq:vdP_ODE}, in which rational stability can be verified with $M=3$ and rate $k$ by solving SOSP~\eqref{eq:rational_stability_SOSP}, for several values of $k$ from Table~\ref{tab:rational_stability_vdP}.
		The simulated region of (asymptotic) attraction is displayed in black dashes.}
	\label{fig:rational_stability_vdP_ROA}	
	\vspace*{-0.4cm}
\end{figure}

\subsection{Finite-Time Stability}\label{subsec:examples:finite_time}

Finally, we consider the notion of finite-time stability, applying the SOSP from Cor.~\ref{cor:finite_time_stability_SOS}. We first verify that this SOSP may indeed be used to compute tight lower bounds on the rate performance (and thus settling time) of a scalar-valued finite-time stable system, and then apply the SOSP to test the local rate performance of a modified van der Pol system, plotting associated regions of finite-time attraction.

\subsubsection{Global Finite-Time Stability}

First, we illustrate accuracy of SOSP~\eqref{eq:finite_time_stability_SOSP} for finite-time stability by considering a class of scalar ODEs of the form
\begin{equation*}
	\dot{x}(t)=f(x(t)):=-\frac{1}{\eta} \sign(x(t))|x(t)|^{1-\eta},
\end{equation*}
for $\eta\in(0,1)$.
Solutions to this ODE are given by $x(t)=x(0)\sqrt[\eta]{1-|x(0)|^{-\eta}t}$, so that for any $\eta\in(0,1)$, the vector field is globally finite-time stable with rate $k=1$, and gain $M=1$, corresponding to a settling time function of $T(x)=|x|^\eta$. Introducing the variable substitution as in Cor.~\ref{cor:finite_time_stability_SOS} with $r=\frac{1}{\eta}$, we get the modified vector field
\begin{equation*}
	\tilde{f}_{r}(z)=\frac{1}{r}f(\sign(z)|z|^{r})|z|^{1-r}=-\sign(z).
\end{equation*}
Solving SOSP~\eqref{eq:finite_time_stability_SOSP} for this system using $d=r$, $h(z)=|z|$ and $\Omega=\R$, global finite-time stability can be verified tightly with rate $k=1$ and gain $M=1$ for each $\eta\in\{\frac{1}{2},\frac{1}{3},\frac{1}{4},\hdots,\frac{1}{10}\}$. Specifically, for each value of $\eta$, we find 
$V(x) = |x|^{\eta}$.

\subsubsection{Finite-Time Stable Regions of Modified van der Pol System}

As a second illustration, consider the ODE defined by the following vector field, corresponding to a van der Pol oscillator with variable substitution $x\rightarrow \sign(x)|x|^{1/3}$:
\begin{align}\label{eq:vdP_FT_ODE}
	\dot{x}
	\!:=\!\bmat{-\sign(x_{2})|x_{2}|^{\frac{1}{3}}\\
	                   2(|x_{1}|^{\frac{2}{3}}-1)\sign(x_{2})|x_{2}|^{\frac{1}{3}}+\sign(x_{1})|x_{1}|^{\frac{1}{3}}\!}\!.
\end{align}
We verify finite-time stability of this vector field using SOSP~\eqref{eq:finite_time_stability_SOSP} with $\Omega:=\{x\mid \norm{x}_{2}\leq R\}$, $\eta=\frac{2}{3}$, $r=3$, and $h(z)=z_{1}^2 z_{2}^2$. In this case, $\tilde{f}_{r}$ in SOSP~\eqref{eq:finite_time_stability_SOSP} is found to be
\begin{align*}
	\tilde{f}_{r}(z)q(z)
	&=\frac{1}{3}\bmat{-z_{2}^3\\-2(1-z_{1}^{2})z_{1}^2 z_{2}+z_{1}^3},
\end{align*}
with $\tilde{\Omega}=\{z\mid R^2-z_{1}^{6}-z_{2}^{6}\geq 0\}$.
Solving SOSP~\eqref{eq:finite_time_stability_SOSP} with $d=12$ and $R=1$, a greatest lower bound on the finite-time rate performance is obtained as $k=0.2976$, with associated upper bound on the gain performance as $M=1.9994$. This implies a settling time function of $T(x)=k^{-1}\norm{x}_{2/r}^{\eta}=3.36\norm{x}_{2/3}^{2/3}$ on any forward-invariant set contained in the unit circle.

\begin{figure}[b]
	\centering
	\vspace*{-0.4cm}
	\includegraphics[width=1.0\linewidth]{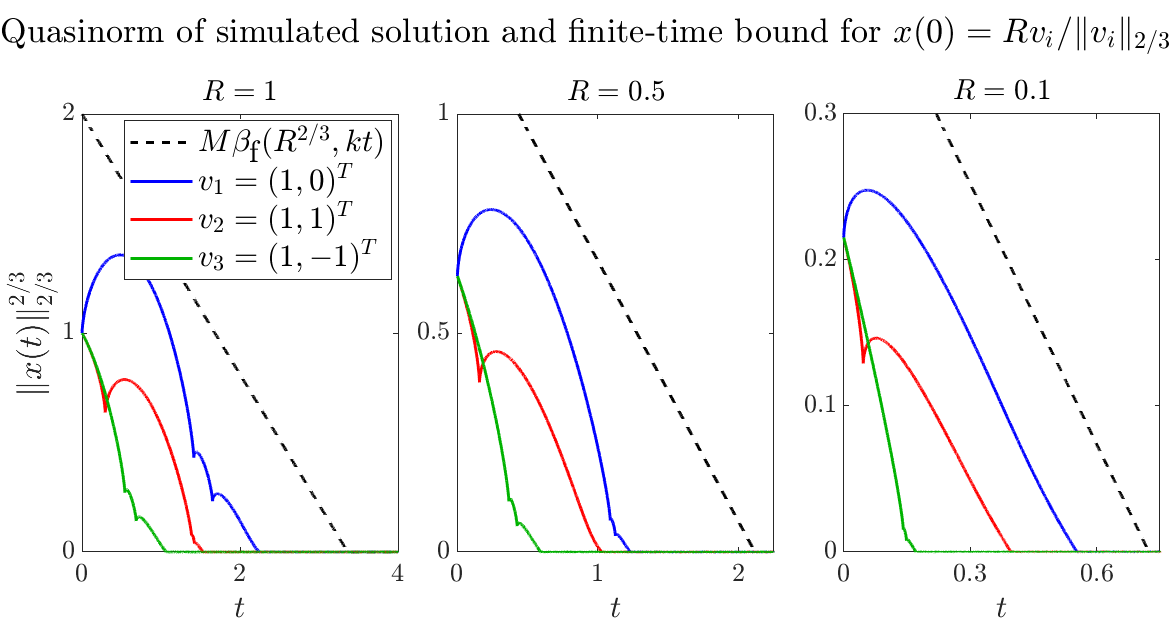}
	\vspace*{-0.6cm}
	\caption{Value of $\alpha(x(t))=\norm{x(t)}_{2/3}^{2/3}$ for numerical solutions, $x(t)$, to the ODE in Eqn.~\eqref{eq:vdP_FT_ODE}, for several initial states $x(0)$ with $\|x(0)\|_{2/3}=R$ for $R\in\{1,0.5,0.1\}$, along with the finite-time stability bound $M\beta_{\tnf{f}}(\alpha(x(0)),kt)=1.999(R^{2/3}-0.2976t)$.}
	\label{fig:finite_time_stability_vdP_bound}	
\end{figure}

To verify accuracy of the result, we simulate ODE~\eqref{eq:vdP_FT_ODE} for 1000 randomly generated initial conditions in the unit circle, and compute the settling time, $T_{\tnf{sim}}(x(0))$, for each initial condition, $x(0)$. For each simulation result, we then establish an estimate of the rate performance as $k_{\tnf{sim}}(x(0)):=\norm{x(0)}_{2/3}^{2/3}/T_{\tnf{sim}}(x(0))$. Taking the minimum of the values $k_{\tnf{sim}}(x(0))$ over all initial conditions, we obtain an estimate of the rate performance as $k_{\tnf{sim}}=0.3107$, which is slightly greater than the lower bound $k=0.2976$ computed with SOSP~\eqref{eq:finite_time_stability_SOSP}.
For visualization of rate, Fig.~\ref{fig:finite_time_stability_vdP_bound} depicts $\norm{x(t)}_{2/3}^{2/3}$ for several numerical solutions, $x(t)$, for initial conditions $x(0)=R\frac{v_{i}}{\norm{v_{i}}_{2/3}}$ for $v_{i}\in\{\smallbmat{1\\0},\smallbmat{1\\1},\smallbmat{1\\-1}\}$, along with the finite-time stability bound $M\beta_{\tnf{f}}(R^{\eta},kt)=M(R^{2/3}-kt)$. Fig.~\ref{fig:finite_time_stability_vdP_bound} shows that the obtained bound on the rate performance accurately approximates the slope of $\norm{x(t)}_{2/3}^{2/3}$ as solutions $x(t)$ approach the origin.

Fixing $M=2$, SOSP~\eqref{eq:finite_time_stability_SOSP} is also solved on $\Omega:=\{x\mid \norm{x}_{2}\leq R\}$ for $R\in\{1,2,2.5,2.75,2.85,2.9,2.95,2.96\}$, obtaining greatest lower bounds on the rate performance $k$ as in Table~\ref{tab:finite_time_stability_vdP}.
Fixing $k=0$ and performing bisection on the value of $R$, finite-time stability can be verified up to $R=2.962$.
For each value of $k$, the largest closed level set, $G_{k}$, of the resulting Lyapunov function is computed, several of which are displayed in Fig.~\ref{fig:finite_time_stability_vdP_ROA}.
This figure shows that the region of finite-time rate performance $k$ approaches the simulation-based estimate of the region of attraction (black dashes) as $k$ tends to $0$.

\begin{figure}[t]
	\centering
	\hspace*{-0.4cm}\includegraphics[width=1.1\linewidth]{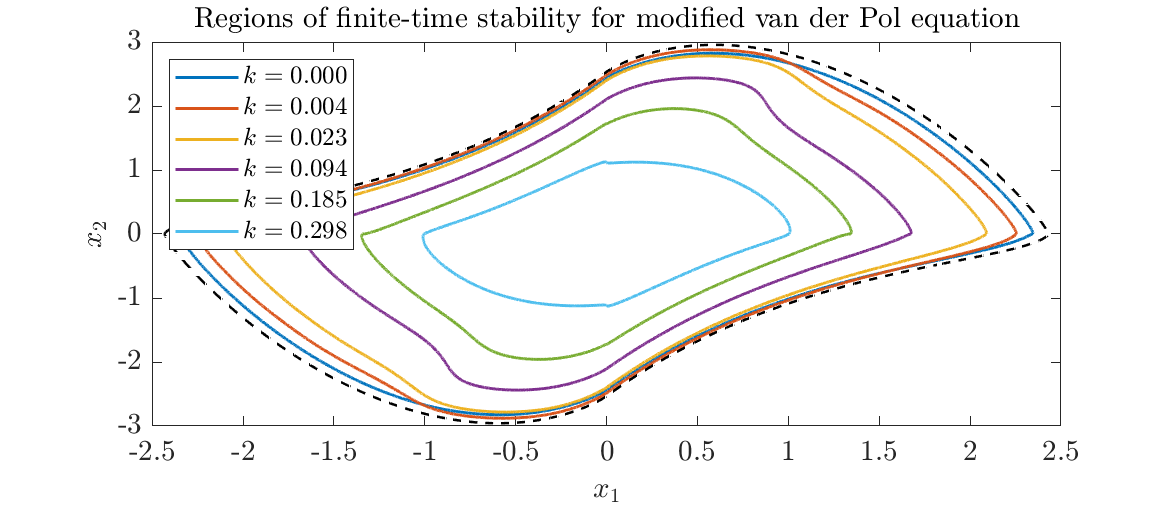}%
	\vspace*{-0.2cm}
	\caption{Forward-invariant sets, $G_{k}$, on which finite-time stability of the vector field in~\eqref{eq:vdP_FT_ODE} can be verified with different rates $k$ (and settling times $T(x)=\frac{1}{k}\norm{x}_{2/3}^{2/3}$) by solving SOSP~\eqref{eq:finite_time_stability_SOSP}. The simulated region of attraction is displayed as a black dashed line.}\label{fig:finite_time_stability_vdP_ROA}	
	\vspace*{-0.4cm}
\end{figure}

For each of the values of $k$ in Table~\ref{tab:finite_time_stability_vdP} and associated region of performance $G_{k}$, a tighter lower bound on the rate performance on this region, $k_{G_{k}}$, is computed by solving SOSP~\eqref{eq:finite_time_stability_SOSP} with $\Omega=G_{k}$, again setting $d=12$ and $M=2$. In addition, an estimate of the rate performance on each region $G_{k}$ is computed using simulation. Specifically, for each value of $k$, solutions to the ODE~\eqref{eq:vdP_FT_ODE} are simulated for 500 distinct initial conditions along the boundary of the associated region $G_{k}$. A largest value of $k_{\tnf{sim}}$ is then computed (through bisection) such that all of the simulated solutions on $G_{k}$ satisfy $\norm{x(t)}_{2/3}^{2/3}\leq M(\norm{x(0)}_{2/3}^{2/3}-k_{\tnf{sim}}t)$ for $t\in[0,100]$. The obtained value of $k_{\tnf{sim}}$ for each value of $k$ is given in Table~\ref{tab:finite_time_stability_vdP}. The results show that, for smaller values of $R$, the lower bounds on the rate performance, $k_{G_{k}}$, are relatively accurate compared to the simulated estimate of the rate performance, $k_{\tnf{sim}}$. However, this accuracy decreases as the considered region $G_{k}$ grows.

%

\begin{table}[t]
	\setlength{\tabcolsep}{3.5pt}
	\begin{tabular}{c|cccccccc}
		$R$     & 1 & 2& 2.5 & 2.75 & 2.85 & 2.9 & 2.95 & 2.96 \\\hline
		$k$ & 0.2976 & 0.1849 & 0.0945 & 0.0414 & 0.0225 & 0.0130 & 0.0041 & 0.0026 \\\hline
		$k_{G_{k}}$ & 0.2976 & 0.2612 & 0.1805 & 0.1654 & 0.1451 & 0.0397 &	0.0157 & 0.0219 \\
		$k_\tnf{sim}$ & 0.3107 & 0.3107 & 0.3107 & 0.2879 & 0.2431 & 0.2144 & 0.1793 & 0.1690		
	\end{tabular}
	\caption{Largest values of $k$ for which SOSP~\eqref{eq:finite_time_stability_SOSP} is feasible for the ODE in Eqn.~\eqref{eq:vdP_FT_ODE}, using $M=2$, $r=3$,  $\eta=\frac{2}{3}$ $d=12$, and $\Omega:=\{x\mid \norm{x}_{2}\leq R\}$. Greatest lower bounds on the rate performance, $k_{G_{k}}$, on the associated invariant regions $G_{k}$ as in Fig.~\ref{fig:finite_time_stability_vdP_ROA} are provided as well, along with simulated estimates of the rate performance on these regions, $k_{\tnf{sim}}$.}\label{tab:finite_time_stability_vdP}
	\vspace*{-0.6cm}
\end{table}

\section{Conclusion}

This paper provides a framework for characterization of rate and gain performance in nonlinear systems. Specifically, such rates are formally defined in terms of normalized, time-invariant $\beta$ functions. This framework then allows rate and gain performance to be quantified for a broad class of nonlinear systems --- encompassing previous notions of exponential, rational, and finite-time stability. Rate and gain performance using the $\beta$ function approach is further shown to admit an equivalent Lyapunov characterization using a combination of a converse comparison lemma and a novel converse Lyapunov construction. When applied to exponential, rational, and finite-time rate parameters ($k$), these equivalent Lyapunov characterizations are $\dot V\le -k V$, $\dot V \le -k V^2$ and $\dot V\le -k$, respectively. The results are formulated both locally and globally. Sum-of-squares implementations of these Lyapunov tests for exponential, rational, and finite-time rates are also given. In the case of rational stability, it is shown that the new conditions offer improved quantification of rate performance, with a conjectured scaling rule for conservatism given as $\frac{1}{M}$ where $M$ is the gain performance. The SOS conditions are verified using a battery of numerical tests for both local and global stability. The accuracy of the quantified greatest lower bounds on rate are verified by comparison to numerical simulation, the conjectured conservatism scaling rule in the rational case is confirmed, and regions of rate performance are constructed.

\bibliographystyle{IEEEtran}
\bibliography{bibfile}
%
%

\vspace*{-1.25cm}
\begin{IEEEbiography}[{\vspace*{-0.7cm}\includegraphics[width=1in,height=1in,clip,keepaspectratio]{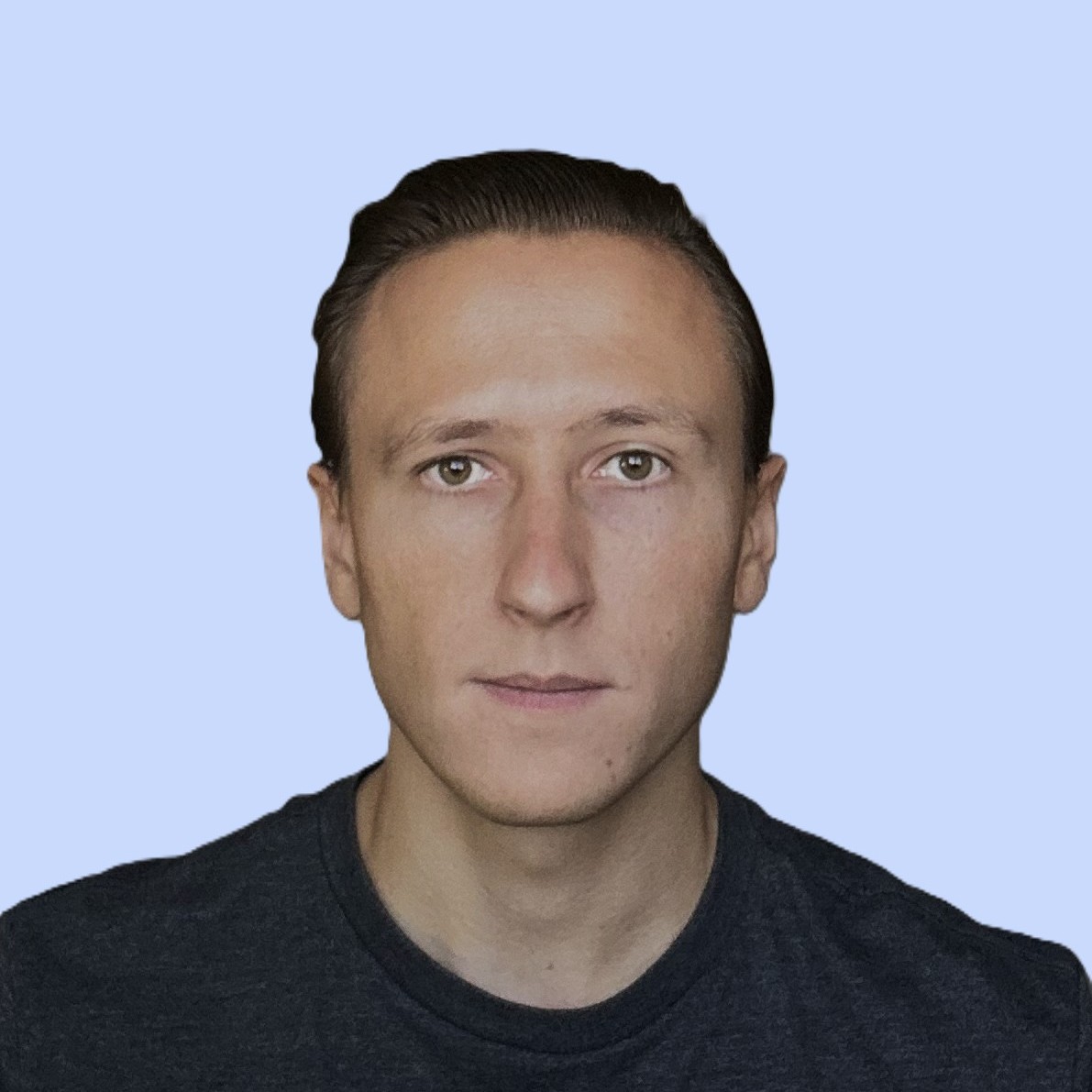}}]{Declan S. Jagt} received his B.S from Amsterdam University College (2017), and M.S in applied mathematics from Delft University of Technology (2020). Since 2020, he has been part of the Cybernetic Systems and Controls Lab at Arizona State University, working on convex optimization based methods for analysis and control of nonlinear and infinite-dimensional systems.
\end{IEEEbiography}

\vspace*{-2.2cm}
\begin{IEEEbiography}[{\vspace*{-0.5cm}\includegraphics[width=1in,height=1in,clip,keepaspectratio]{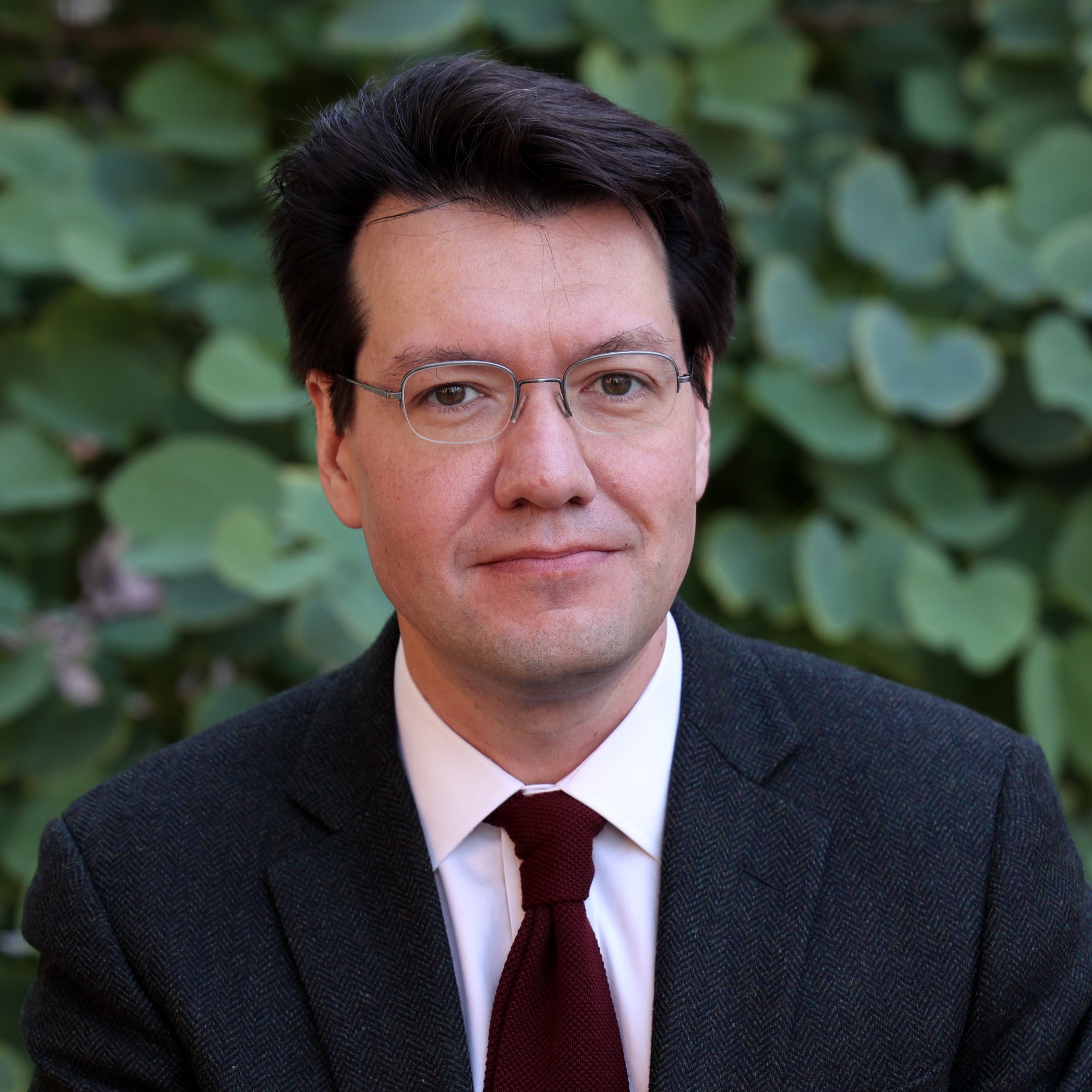}}]{Matthew M. Peet} received M.S. and Ph.D. degrees in aeronautics and astronautics from Stanford University (2000-2006). He was a postdoc at INRIA (2006-2008) and Asst. Professor at the Illinois Institute of Technology (2008-2012). Currently, he is Associate Professor of Aerospace Engineering at Arizona State University.\vspace*{-1.0cm}
\end{IEEEbiography}

\clearpage

\begin{appendices}
	
\section{Proof of Converse Lyapunov Theorems from Subsection~\ref{subsec:LF_characterization:necessity}}\label{appx:proofs}

In this appendix, we prove Thm.~\ref{thm:LF_necessity_diff_FT} and Thm.~\ref{thm:LF_necessity_diff}, showing that under suitable regularity conditions, $\beta$-stability of a vector field $f$ implies existence of a locally Lipschitz continuous Lyapunov function certifying $\beta$-stability with rate $(1-\epsilon)$, for all $\epsilon\in(0,1)$.

For these proofs, we first show in Lemma~\ref{lem:beta_negative_appx} that if $\beta$ is monotonically decreasing in $t$, then $\beta$ admits a natural extension to negative times, and which preserves the semigroup property, $\beta(\beta(y,t),-t)=y$. This extension will be used in Lemma~\ref{lem:V_sup_appx} to represent the proposed converse Lyapunov function as a supremum -- a representation which is then used in the proof of Thm.~\ref{thm:LF_necessity_diff_FT} and Thm.~\ref{thm:LF_necessity_diff}.

\begin{lem}\label{lem:beta_negative_appx}
	For $\rho:\R_{+}\to\R$ continuous on $(0,\infty)$, let $\beta:\R_{+}\times\R_{+}\to\R_{+}$ be the unique function satisfying $\partial_{t}\beta(y,t)=\rho(\beta(y,t))$ and $\beta(y,0)=y$. Suppose further that $\beta(y,t)$ is monotonically decreasing in $t$ whenever $\beta(y,t)>0$. 	
	Then, for any $z>0$, there exists a continuous function $\hat{\beta}_{z}:[0,z)\times\R\to[0,z]$ with the following properties
	\begin{enumerate}
		\item[(i)] $\hat{\beta}_{z}(\cdot,t)$ is monotonically nondecreasing for all $t\in\R$,
		
		\item[(ii)] $\hat{\beta}_{z}(\cdot,t)=\beta(\cdot,t)$ for all $t\geq 0$,
		
		\item[(iii)] $\hat{\beta}_{z}(\hat{\beta}_{z}(y,t),-s)=\hat{\beta}_{z}(y,t-s)$ for all $(y,t)  \in\{ (y,t)\in (0,z) \times \R_{+}\mid  0<\hat{\beta}_{z}(y,t)< z\}$ and $s\leq t$,
		
		\item[(iv)] $\hat{\beta}_{z_{1}}(y,t)=\hat{\beta}_{z_{2}}(y,t)$ for all $0<z_{1}\leq z_{2}$, and $(y,t)  \in\{ (y,t)\in (0,z_1) \times \R\mid  0<\hat{\beta}_{z_{1}}(y,t)< z_1\}$.
	\end{enumerate}
\end{lem}
\vspace*{2mm}
\begin{proof}	
	We first show that, for any $z>0$, there exists $T_{\max}(z)\in(0,\infty]$ such that $\beta(z,\cdot):[0,T_{\max}(z))\to(0,z]$ is surjective. To start, since $\rho$ is continuous, for any $z>0$ and $\epsilon\in(0,z)$, we can define $\kappa:=\min_{y\in[\epsilon,z]}|\rho(y)|>0$, so that $|\partial_{t}\beta(z,t)|=|\rho(\beta(z,t))|>\kappa$ for all $t$ for which $\beta(z,t)>\epsilon$. Since $\beta(z,t)$ is monotonically decreasing in $t$ whenever $\beta(z,y)>0$, it follows that $\partial_{t}\beta(z,t)<-\kappa$ whenever $\beta(z,t)>\epsilon$, and thus there exists a finite time $T_{\epsilon}\leq \frac{z-\epsilon}{\kappa}$ such that $\beta(z,t)<\epsilon$ for all $t\geq T_{\epsilon}$. Since this holds for arbitrary $\epsilon\in(0,z)$, this implies $\lim_{t\to\infty}\beta(z,t)=0$. As such, there exists some $T_{\max}(z):=\sup\{t\geq 0\mid \beta(z,t)>0\}$ (possibly infinite) such that $\beta(z,t)>0$ for all $t<T_{\max}(z)$ and $\lim_{t\to T_{\max}(z)}\beta(z,t)=0$. By continuity of $\beta$, it follows that $\beta(z,\cdot):[0,T_{\max}(z))\to(0,z]$ is surjective.

Having established that $\beta(z,\cdot)$ is monotonically decreasing and surjective for any $z> 0$, we use these properties to establish a backward-in-time extension of $\beta$. Specifically, for any $z > 0$, we have the existence of a monotonically decreasing inverse, $\mu(z,\cdot):=\beta(z,\cdot)^{-1}:(0,z]\to[0,T_{\max}(z))$, so that $\mu(z,y)$ indicates the time in which $\beta$ decreases from $z$ to $y$ -- i.e. $\beta(z,\mu(z,y))=y$ for all $y\in(0,z]$, and $\mu(z,\beta(z,t))=t$. Now, for $z\in (0,\infty)$, define $\hat{\beta}_{z}:[0,z)\times\R\to[0,z]$ as
	\begin{equation*}
		\hat{\beta}_{z}(y,t):=\begin{cases}
            0,                      & y=0,\; t\in \R,\\
			z,						& y>0,\, t\leq -\mu(z,y),		\\
			\beta(z,\mu(z,y)+t),	& y>0,\, t>-\mu(z,y).	
		\end{cases}
	\end{equation*}
Unlike $\beta$, for any $z \ge 0$, $\hat{\beta}_{z}$ is defined for all times, where $z$ bounds the value of $\hat \beta$ on negative times.
Since $\mu(z,\cdot)$ and $\beta(z,\cdot)$ are monotonically decreasing, it follows that for any $t\in\R$, $\beta(z,\mu(z,y)+t)$ (and hence $\hat{\beta}_{z}(y,t)$) is monotonically nondecreasing in $y$.
	In addition, for all $y\in[0,z)$ and $t\geq 0$,
	\begin{equation*}
		\hat{\beta}_{z}(y,t)
		=\beta(z,\mu(z,y)+t)
		=\beta(\beta(z,\mu(z,y)),t)
		=\beta(y,t).
	\end{equation*}
	Furthermore, for all $t,s\in\R_{+}$ for which $\hat{\beta}_{z}(y,t)\in(0,z)$ (i.e. $t\in(-\mu(z,y),T_{\max}(y))$) and $s\leq t$, we have
	\begin{align*}
		\hat{\beta}_{z}\bl(\hat{\beta}_{z}(y,t),-s\br)
		&=\beta\bbl(z,\mu\bl(z,\beta(z,\mu(z,y)+t)\br)-s\bbr)	\\
		&=\beta\bl(z,\mu(z,y)+t-s\br)	
		=\hat{\beta}_{z}\bl(y,t-s\br),
	\end{align*}
	where recall that $\mu(z,\beta(z,t))=t$ implies $\mu(z,\beta(z,\mu(z,y)+t))=\mu(z,y)+t$.
	Finally, for any $0<z_{1}\leq z_{2}$ and $y\in(0,z_{1})$, we remark that $\mu(z_{1},y)\leq\mu(z_{2},y)$ and
	\begin{align*}
		&\beta\bl(z_{2},\mu(z_{2},z_{1})+\mu(z_{1},y)\br)
		=\beta\bbl(\beta\bl(z_{2},\mu(z_{2},z_{1})\br),\mu(z_{1},y)\bbr)	\\
		&=\beta\bl(z_{1},\mu(z_{1},y)\br)
		=y
		=\beta\bl(z_{2},\mu(z_{2},y)\br),
	\end{align*}
	implying that $\mu(z_{2},y)=\mu(z_{2},z_{1})+\mu(z_{1},y)$. It follows that, for any $t\in\R$ such that $\hat{\beta}_{z_{1}}(y,t)\in(0,z_{1})$ (which implies $t\in(-\mu(z_{1},y),T_{\max}(y))$), we have
	\begin{align*}
		\hat{\beta}_{z_{2}}(y,t)
		&=\beta\bl(z_{2},\mu(z_{2},y)+t\br)	\\
		&=\beta\bl(z_{2},\mu(z_{2},z_{1})+\mu(z_{1},y)+t\br)	\\
		&=\beta\bbl(\beta\bl(z_{2},\mu(z_{2},z_{1})\br),\mu(z_{1},y)+t\bbr)	\\
		&=\beta\bl(z_{1},\mu(z_{1},y)+t\br)
		=\hat{\beta}_{z_{1}}(y,t).
	\end{align*}
	Thus, $\hat{\beta}_{z}(y,t)$ satisfies all the proposed properties.
\end{proof}

The $\hat{\beta}_{z}$ in Lemma~\ref{lem:beta_negative_appx} may be interpreted as allowing for extension of $\beta$ to negative times for which $\beta$ values are known to exist and are bounded by $z$. For negative times beyond this domain, the value of $\hat{\beta}_{z}$ is simply defined as $z$ and no longer coincides with $\beta$.
Using this restricted negative-time extension of $\beta$, the converse Lyapunov function in Thm.~\ref{thm:LF_necessity} can be equivalently expressed in the form in Eqn.~\ref{eq:V_sup}, as shown in the following lemma.

\begin{lem}\label{lem:V_sup_appx}
	For $\Omega \subseteq \R^n$ and $f:\Omega\to\R^{n}$ with $f(0)=0$, let $G\subseteq\Omega$ be forward invariant, and $\alpha_{1},\alpha_{2}:G\to\R_{+}$ be continuous and positive definite. For $\rho:\R_{+}\to\R$ continuous on $(0,\infty)$, let $\beta$ be the unique continuous function satisfying $\beta(y,0)=y$ and $\partial_{t}\beta(y,t)=\rho(\beta(y,t))$. Suppose further that $\beta(y,t)$ is monotonically decreasing in $t$ for all $y,t\in\R_{+}$ for which $\beta(y,t)>0$, and for any $z>0$, define the associated $\hat{\beta}_{z}\in[0,z)\times\R\to[0,z]$ as in Lemma~\ref{lem:beta_negative_appx}.
	
	If $f$ is $\beta$-stable on $G$ with respect to $\alpha_{1},\alpha_{2}$, then for any $\delta>0$, if $V:G\to\R_{+}$ is defined by $V(0)=0$ and
	\begin{equation*}
		V(x):=\!\sup_{t\in[0,T_{f}(x))}\!\hat{\beta}_{\alpha_{2}(x)+\delta}\bl(\alpha_{1}(\phi_{f}(x,t)),-t\br),\quad \forall x\in G\setminus\{0\},
	\end{equation*}
	where $T_{f}(x):=\sup\{t\geq 0\mid \alpha_{1}(\phi_{f}(x,t))>0\}$ (possibly infinite), we have that $V$ satisfies
	\begin{align*}
		\alpha_{1}(x)\leq V(x)&\leq \alpha_{2}(x),	&	&\forall x\in G,	\\
		V(\phi_{f}(x,t))&\leq \beta(V(x),t),		&	&\forall t\geq 0.
	\end{align*}	
\end{lem}
\vspace*{2mm}
\begin{proof}
	To prove the result, we show that if the conditions of the lemma are satisfied, the converse Lyapunov function as defined in the proof of Eqn.~\eqref{eq:Vconverse} of Thm.~\ref{thm:LF_necessity} reduces to the $V(x)$ as defined in the lemma statement. This then implies as per Thm.~\ref{thm:LF_necessity} that the Lyapunov conditions of the lemma are satisfied.
Specifically, recall from Thm.~\ref{thm:LF_necessity} that if we define $\tilde{V}(x)=\inf_{y\in\mcl{W}(x)} y$ for $x\in G$, where
	\begin{equation*}
		\mcl{W}(x):=\bl\{y\in\R_{+}\mid \alpha_{1}(\phi_{f}(x,t))\leq\beta(y,t),~\forall t\in\R_{+}\br\},
	\end{equation*}
	then $\tilde{V}$ satisfies
	\begin{align*}
		\alpha_{1}(x)\leq \tilde{V}(x)&\leq \alpha_{2}(x),	&	&\forall x\in G,	\\
		\tilde{V}(\phi_{f}(x,t))&\leq \beta(\tilde{V}(x),t),		&	&\forall t\geq 0.
	\end{align*}
	To show that $\tilde{V}(x)=V(x)$ for all $x\in G$, note first that if $T_f(x)$ is finite, then $\phi_{f}(x,T_{f}(x)+t)=0\leq\beta(0,t)$ for all $t\geq 0$, and therefore $\alpha_{1}(\phi_{f}(x,t))\leq \beta(y,t)$ for all $y\in\R_{+}$ and $t\geq T_{f}(x)$ -- implying we may equivalently represent $\mcl{W}(x)$ as
	\begin{equation*}
		\mcl{W}(x)=\bl\{y\in\R_{+}\mid \alpha_{1}(\phi_{f}(x,t))\leq\beta(y,t),~\forall t\in[0,T_{f}(x))\br\}.
	\end{equation*}
	Alternatively, if $T_{f}(x)=\infty$, this identity holds by definition of $\mcl{W}(x)$.
	Now, define $\hat{\beta}_{z}:[0,z)\times\R\to[0,z]$ for $z>0$ as in Lemma~\ref{lem:beta_negative_appx}. By Lemma~\ref{lem:beta_negative_appx}~(ii) and~(iii), we then have $\hat{\beta}_{\alpha_{2}(x)+\delta}(\beta(y,t),-t)=y$ for all $y,t\geq 0$ for which $\beta(y,t)\in(0,\alpha_{2}(x)+\delta)$. Since $\alpha_{1}(\phi_{f}(x,t))\in(0,\alpha_{2}(x)] \subset (0,\alpha_{2}(x)+\delta)$ for all $x\in G\setminus\{0\}$ and $t\in [0,T_{f}(x))$, and $\hat{\beta}_{\alpha_{2}(x)+\delta}(y,t)$ is monotonically nondecreasing in $y$, it follows that for all $y\in \mcl{W}(x)\cap(0,\alpha_{2}(x)]$ and $t\in [0,T_{f}(x))$,
	\begin{equation*}
		\hat{\beta}_{\alpha_{2}(x)+\delta}\bl(\alpha_{1}(\phi_{f}(x,t)),-t\br)
		\leq \hat{\beta}_{\alpha_{2}(x)+\delta}\bl(\beta(y,t),-t\br)
		=y.
	\end{equation*}
	Taking the infimum over $y\in \mcl{W}(x)\cap(0,\alpha_{2}(x)]$, recalling that $\inf_{y\in\mcl{W}(x)}y=\tilde{V}(x)\leq\alpha_{2}(x)$, we find that for all $x\in G\setminus\{0\}$ and $t\in[0,T_{f}(x))$,
	\begin{equation*}
		\hat{\beta}_{\alpha_{2}(x)+\delta}\bl(\alpha_{1}(\phi_{f}(x,t)),-t\br)
		\leq \inf_{y\in\mcl{W}(x)\cap(0,\alpha_{2}(x)]}y
		=\tilde{V}(x),
	\end{equation*}
	and therefore
	\begin{equation*}
		V(x)=\sup_{t\in[0,T_{f}(x))}\hat{\beta}_{\alpha_{2}(x)+\delta}\bl(\alpha_{1}(\phi_{f}(x,t)),-t\br)
		\leq \tilde{V}(x).
	\end{equation*}
	On the other hand, for all $t\in[0,T_{f}(x))$, we also have
	\begin{equation*}
		\alpha_{1}(\phi_{f}(x,t))
		=\beta\bl(\hat{\beta}_{\alpha_{2}(x)+\delta}\bl(\alpha_{1}(\phi_{f}(x,t)),-t\br),t\br)
		\leq \beta\bl(V(x),t\br),
	\end{equation*}
	and thus $V(x)\in\mcl{W}(x)$. It follows that, for all $x\in G\setminus\{0\}$,
	\begin{equation*}
		V(x)
		\geq \inf_{y\in\mcl{W}(x)}y
		=\tilde{V}(x).
	\end{equation*}
	Since also $V(0)=0=\tilde{V}(0)$, we conclude that $V=\tilde{V}$, and hence as per Thm.~\ref{thm:LF_necessity}, the lemma holds.
\end{proof}

Lemma.~\ref{lem:V_sup_appx} shows that, if $\beta(y,t)$ is monotonically decreasing in $t$, $\beta$-stability of $f$ with respect to two measures can be certified with a Lyapunov function of the form in Eqn.~\eqref{eq:V_sup}.
In the following subsections, we will use this alternative converse Lyapunov construction to prove Thm.~\ref{thm:LF_necessity_diff_FT} and Thm.~\ref{thm:LF_necessity_diff}.

\subsection{Proof of Theorem.~\ref{thm:LF_necessity_diff_FT}}\label{appx:proofs:necessity_diff_FT}

Having established that, for $\beta$ monotonically decreasing, $\beta$-stability can be certified by a converse Lyapunov function of the form in Eqn.~\eqref{eq:V_sup}, the following theorem shows that if the settling time function is continuous, this converse Lyapunov function is continuous as well. By the converse comparison principle, this allows the Lyapunov condition $V(\phi_{f}(x,t))\leq \beta(V(x),t)$ in Thm.~\ref{thm:LF_necessity} to be expressed in terms of the upper Dini derivative of $V$, as $D_{t}^{+}V(\phi_{f}(x,t))|_{t=0}\leq \rho(V(x))$.

\textit{Theorem~\ref{thm:LF_necessity_diff_FT}:}
For $\,\Omega \subseteq \R^n$ and $f:\Omega\to\R^{n}$ with $f(0)=0$, let $G\subseteq\Omega$ be forward invariant, and $\alpha_{1},\alpha_{2}:G\to\R_{+}$ be continuous. For $\rho:\R_{+}\to\R$ continuous on $(0,\infty)$, let $\beta$ be the unique continuous function satisfying $\beta(y,0)=y$ and $\partial_{t}\beta(y,t)=\rho(\beta(y,t))$.
Suppose that $\alpha_{1}$ is positive definite, and $\beta(y,t)$ is monotonically decreasing in $t$ whenever $\beta(y,t)>0$. Finally, suppose that $T_{f}(x):=\sup\,\{t\geq 0\mid \alpha_{1}(\phi_{f}(x,t))>0\}$ (possibly infinite) is continuous for $x\in G\setminus\{0\}$.

If $f$ is $\beta$-stable on $G$ with respect to $\alpha_{1},\alpha_{2}$,
then there exists a continuous function $V:G\to\R_{+}$ which satisfies
\begin{align*}
	\alpha_{1}(x)\leq V(x)&\leq \alpha_{2}(x),	&	&	\\
	\dot{V}(x)&\leq \rho(V(x)),	&	&\forall x\in G, 
\end{align*}
where $\dot{V}(x):=D_{t}^{+}V(\phi_{f}(x,t))|_{t=0}$.

\begin{proof}
	Suppose that $f$ is $\beta$-stable on $G$. Then, for any $\delta>0$, define $V:G\to\R_{+}$ as $V(0):=0$ and
	\begin{equation*}
		V(x):=\sup_{t\in[0,T_{f}(x))}\hat{\beta}_{\alpha_{2}(x)+\delta}\bl(\alpha_{1}(\phi_{f}(x,t)),-t\br),
	\end{equation*}
	for $x\in G\setminus\{0\}$, where $\hat{\beta}_{z}$ is as in Lemma~\ref{lem:beta_negative_appx}. We have by Lemma~\ref{lem:V_sup_appx} that
	\begin{align*}
		\alpha_{1}(x)\leq V(x)&\leq\alpha_{2}(x),	&	&\forall x\in G,	\\
		V\bl(\phi_{f}(x,t)\br)&\leq \beta\bl(V(x),t\br),		&	&\forall t\geq 0.
	\end{align*}
	To show that $V$ is continuous, fix an arbitrary compact subset $K\subseteq G\setminus\{0\}$, and define $a:=\max_{x\in K}\alpha_{2}(x)+\delta$. Then, we have by Lemma~\ref{lem:beta_negative_appx}~(iv) that
	\begin{equation*}
		\hat \beta_{\alpha_{2}(x)+\delta}\bl(\alpha_{1}(\phi_{f}(x,t)),-t\br)
		=\hat \beta_{a}\bl(\alpha_{1}(\phi_{f}(x,t)),-t\br),
	\end{equation*}
	for all $x\in K$  and $t\in[0,T_{f}(x))$, and therefore
	\begin{equation*}
		V(x)=\sup_{t\in[0,T_{f}(x))}\!\!\hat{\beta}_{a}\bl(\alpha_{1}(\phi_{f}(x,t)),-t\br),\qquad \forall x\in K.
	\end{equation*}	
	Here, by joint continuity of $\phi_{f}$ and $\beta$, and continuity of $\alpha_{1}$, for any $x\in K$ and $\epsilon>0$ there exists $\delta_{1}$ such that $\max\{\norm{x-y}_{2},|s-t|\}<\delta_{1}$ implies
	\begin{equation*}
		\bbl|\hat{\beta}_{a}\bl(\alpha_{1}(\phi_{f}(x,s)),-s\br)
		-\hat{\beta}_{a}\bl(\alpha_{1}(\phi_{f}(y,t)),-t\br)\bbr|	
		<\epsilon,
	\end{equation*}
	for $y\in K$, $s<T_{f}(x)$ and $t<T_{f}(y)$. Now suppose without loss of generality that $T_f(x)\leq T_f(y)$. If $\norm{x-y}_{2}<\delta_{1}$, then
	\begin{align*}
		&\hat{\beta}_{a}\bl(\alpha_{1}(\phi_{f}(y,t)),-t\br)
		-V(x)	\\
		&\quad\leq \hat{\beta}_{a}\bl(\alpha_{1}(\phi_{f}(y,t)),-t\br)
		-\hat{\beta}_{a}\bl(\alpha_{1}(\phi_{f}(x,t)),-t\br)	
		<\epsilon,
	\end{align*}
	for all $t< T_{f}(x)$. If $T_f(x)$ is infinite, this implies $V(y)-V(x)<\epsilon$.
	If $T_{f}(x)$ is finite, then we also require the inequality to hold for $t \in [T_{f}(x),T_{f}(y))$. In this case, by continuity of $T_f$, there exists $\delta_{2}>0$ such that $\norm{x-y}_{2}<\delta_{2}$ implies $|T_{f}(y)-T_{f}(x)|<\delta_{1}$ for all $y\in K$. Hence $\norm{x-y}_{2}<\min\{\delta_{1},\delta_{2}\}$ implies
	\begin{align*}
		&\hat{\beta}_{a}\bl(\alpha_{1}(\phi_{f}(y,t)),-t\br)-V(x)	\\
		&\leq \hat{\beta}_{a}\bl(\alpha_{1}(\phi_{f}(y,t)),-t\br) 	\\
		&\enspace -\hat{\beta}_{a}\bl(\alpha_{1}(\phi_{f}(x,t-\min\{\delta_{1},t\})),-[t-\min\{\delta_{1},t\}]\br) <\epsilon,
	\end{align*}
	for all $t\in [T_{f}(x),T_{f}(y))$ --- likewise implying $V(y)-V(x)<\epsilon$.  By similar reasoning, we can show that $\norm{x-y}_{2}<\min\{\delta_{1},\delta_{2}\}$ also implies $V(x)-V(y)<\epsilon$.
	We conclude that $V$ is continuous on $K$. Since this holds for any compact subset $K\subseteq G\setminus\{0\}$, it follows that $V$ is continuous on $G\setminus\{0\}$. Continuity of $V$ at the origin follows from the fact that $V(x)\leq\alpha_{2}(x)$ and continuity of $\alpha_{2}(x)$, noting that $V(0):=0=\alpha_{2}(0)$. Finally, since $V\bl(\phi_{f}(x,t)\br)\leq \beta\bl(V(x),t\br)$, it follows by the converse comparison principle (Lemma~\ref{lem:comparison_N}) that $\dot{V}(x)\leq \rho(V(x))$ for all $x\in G$.
\end{proof}

\subsection{Proof of Theorem.~\ref{thm:LF_necessity_diff}}\label{appx:proofs:necessity_diff}

Although Thm.~\ref{thm:LF_necessity_diff_FT} shows that, for $\beta$ monotonically decreasing, $\beta$-stability implies existence of a continuous Lyapunov function satisfying $\dot{V}(x):=D_{t}^{+}V(\phi_{f}(x,t))|_{t=0}\leq\rho(V(x))$, testing this condition still requires knowledge of the solution map, $\phi_{f}(x,t)$. Therefore, the following lemma provides a necessary Lyapunov condition which can be tested without knowledge of the solution map, showing that if $f$ and $\beta$ are also locally Lipschitz continuous, then for any $\epsilon\in(0,1)$, there exists a locally Lipschitz Lyapunov function that satisfies the weakened condition $\nabla V(x)^T f(x)\leq (1-\epsilon)\rho(V(x))$.

\textit{Theorem~\ref{thm:LF_necessity_diff}:}
For $\,\Omega \subseteq \R^n$ and $f:\Omega\to\R^{n}$ with $f(0)=0$, let $G\subseteq\Omega$ be forward invariant, and $\alpha_{1},\alpha_{2}:G\to\R_{+}$ be continuous. For $\rho:\R_{+}\to\R$ continuous on $(0,\infty)$, let $\beta$ be the unique continuous function satisfying $\beta(y,0)=y$ and $\partial_{t}\beta(y,t)=\rho(\beta(y,t))$.
Suppose further that $\alpha_{1}$ is positive definite and coercive, $\beta(y,t)$ is monotonically decreasing in $t$ for $y>0$, and that $f$, $\alpha_{1}$, and $\beta$ are locally Lipschitz continuous.

If $f$ is $\beta$-stable on $G$ with respect to $\alpha_{1},\alpha_{2}$,
then for every $\epsilon\in(0,1)$ there exists a continuous function $V_{\epsilon}:G\to\R_{+}$ which is locally Lipschitz continuous and differentiable almost everywhere on $G\setminus\{0\}$ and satisfies
\begin{align*}
	\alpha_{1}(x)\leq V_{\epsilon}(x)&\leq \alpha_{2}(x),	&	&\forall x\in G,	\\
	\nabla V_{\epsilon}(x)^T f(x)&\leq (1-\epsilon)\rho(V_{\epsilon}(x)),	&	&\text{for a.e. } x\in G.
\end{align*}%
\begin{proof}	
	Suppose $f$ is $\beta$-stable on $G$,
	and fix arbitrary $\epsilon\in(0,1)$. Since $\beta(y,t)$ is decreasing in $t$, it follows that $f$ is also $\tilde{\beta}$-stable on $G$ with $\tilde{\beta}(y,t)=\beta(y,[1-\epsilon]t)$. Defining $V_{\epsilon}:G\to\R_{+}$ as $V_{\epsilon}(0):=0$ and (for any $\delta>0$)
	\begin{equation*}
		V_{\epsilon}(x):=\sup_{t\in[0,T_{f}(x))}\hat{\beta}_{\alpha_{2}(x)+\delta}\bl(\alpha_{1}(\phi_{f}(x,t)),-[1-\epsilon]t\br),
	\end{equation*}
	for $x\in G\setminus\{0\}$,
	where $\hat{\beta}_{z}$ is as in Lemma~\ref{lem:beta_negative_appx} and $T_{f}(x):=\sup\{t\geq 0\mid \alpha_{1}(\phi_{f}(x,t))>0\}$, 
	we then have by Thm.~\ref{thm:LF_necessity_diff_FT} that for all $x\in G$,
	\begin{align*}
		\alpha_{1}(x)\leq V_{\epsilon}(x)&\leq\alpha_{2}(x),	&	&\text{and}	&
		\dot{V}_{\epsilon}(x)&\leq (1-\epsilon)\rho\bl(V_{\epsilon}(x)\br),		
	\end{align*}
	where $\dot{V}_{\epsilon}(x):=D_{t}^{+}V(\phi_{f}(x,t))|_{t=0}$. If $V_{\epsilon}$ is locally Lipschitz continuous on $G\setminus\{0\}$, it follows by Rademacher's theorem that $V_{\epsilon}(x)$ is differentiable for almost every $x\in G$, and hence satisfies $\nabla V_{\epsilon}(x)^T f(x)=\dot{V}_{\epsilon}(x)\leq (1-\epsilon)\rho(V_{\epsilon}(x))$ for all those $x$. It remains only to prove, therefore, that $V_{\epsilon}$ is locally Lipschitz continuous on $G\setminus\{0\}$.

	First, since $f$ is locally Lipschitz continuous at $0$, and $\alpha_{1}$ is positive definite, $T_{f}(x)=\infty$ for $x \neq 0$ (vector fields which are Lipschitz at $0$ cannot be finite-time stable), whence 
	\begin{equation*}
		V_{\epsilon}(x):=\sup_{t\in[0,\infty)} \hat{\beta}_{\alpha_{2}(x)+\delta}\bl(\alpha_{1}(\phi_{f}(x,t)),-[1-\epsilon]t\br).
	\end{equation*}
	Now, to show that $V_{\epsilon}$ is locally Lipschitz continuous on $G\setminus\{0\}$, we first prove that the supremum defining $V_{\epsilon}$ is achieved in finite time. For this, note that by $\beta$-stability of $f$, we have $\alpha_{1}(\phi_{f}(x,t))\leq\beta(\alpha_{2}(x),t)$, and therefore (using  Lemma~\ref{lem:beta_negative_appx})
	\begin{align*}
		&\hat{\beta}_{\alpha_{2}(x)+\delta}\bl(\alpha_{1}(\phi_{f}(x,t)),-[1-\epsilon]t\br)\\
		&\qquad\leq \hat{\beta}_{\alpha_{2}(x)+\delta}\bl(\beta(\alpha_{2}(x),t),-[1-\epsilon]t\br)	
		=\beta\bl(\alpha_{2}(x),\epsilon t\br),
	\end{align*}
	for all $x\in G\setminus\{0\}$ and $t\geq 0$. As in the proof of Lemma~\ref{lem:beta_negative_appx}, since $\beta(y,t)$ is monotonically decreasing in $t$, for any $z>0$, $\beta(z,\cdot)$ has a monotonically decreasing inverse, $\mu(z,\cdot)$, whence for any $x\in G\setminus\{0\}$ there exists a time $\tau(x)=\frac{1}{\epsilon}\mu(\alpha_{2}(x),\alpha_{1}(x))>0$ such that $\beta\bl(\alpha_{2}(x),\epsilon \tau(x)\br)=\alpha_{1}(x)$. Continuity of $\tau(x)$ follows from continuity of $\beta$, $\alpha_{1}$, and $\alpha_{2}$. Furthermore, for $t\geq\tau(x)$, we then have
	\begin{align*}
		&\hat{\beta}_{\alpha_{2}(x)+\delta}\bl(\alpha_{1}(\phi_{f}(x,t)),-[1-\epsilon]t\br)
		\leq\beta\bl(\alpha_{2}(x),\epsilon t\br)	\\
		&\hspace*{0.75cm}\leq \alpha_{1}(x)
		\leq \sup_{t\in[0,\tau(x)]}\hat{\beta}_{\alpha_{2}(x)+\delta}\bl(\alpha_{1}(\phi_{f}(x,t)),-[1-\epsilon]t\br),
	\end{align*}
which implies that
	\begin{equation*}
		V_{\epsilon}(x)
		=\sup_{t\in[0,\tau(x)]}\hat{\beta}_{\alpha_{2}(x)+\delta}\bl(\alpha_{1}(\phi_{f}(x,t)),-[1-\epsilon]t\br).
	\end{equation*}
 We now further simplify $V_\epsilon$ by restriction to a compact set. Specifically, for any compact set $K\subseteq G\setminus\{0\}$, define $\tau_{K}:=\max_{x\in K}\tau(x)<\infty$. Then, for any $x\in K$, we have
	\begin{equation*}
V_{\epsilon}(x)=\sup_{t\in[0,\tau_{K}]}\hat{\beta}_{\alpha_{2}(x)+\delta}\bl(\alpha_{1}(\phi_{f}(x,t)),-[1-\epsilon]t\br).
	\end{equation*}
	Furthermore, defining $a:=\max_{x\in K}\alpha_{2}(x)+\delta$, we have by Lemma~\ref{lem:beta_negative_appx}~(iv) that $\beta_{\alpha_{2}(x)+\delta}(y,t)=\beta_{a}(y,t)$ for all $y\in(0,\alpha_{2}(x)]$ and $x\in K$.
	Since $\alpha_{1}(\phi_{f}(x,t))\in(0,\alpha_{2}(x)]$ for all $t\in[0,\tau_{K}]$, we can further simplify $V_{\epsilon}$ as
	\begin{equation*}
		V_{\epsilon}(x)=\sup_{t\in[0,\tau_{K}]}\hat{\beta}_{a}\bl(\alpha_{1}(\phi_{f}(x,t)),-[1-\epsilon]t\br).
	\end{equation*}
Having obtained this simplified version of $V_{\epsilon}$ on a compact set, $K$, we now establish local Lipschitz continuity of $V_{\epsilon}$ using Lipschitz continuity of $\phi_f, \alpha$ and $\beta$.

First, define $c:=\min_{x\in K,t\in[0,\tau_{K}]}\alpha_{1}(\phi_{f}(x,t))$, where the minimum exists by compactness of $K$ and continuity of $\phi_{f}$ and $\alpha_1$. Since $0\notin K$ and $\alpha_{1}$ is positive definite, we have $\alpha_{1}(\phi_{f}(x,t))>0$ for all $x\in K$ and $t\in[0,\tau_K]$, whence $c>0$. Since $\alpha_{1}$ is continuous and $\alpha_{1}(0)=0$, there then exists $\delta_{1}>0$ such that $\norm{x}_{2}\leq\delta_{1}$ implies $\alpha_{1}(x)< c$.
Similarly, since $\alpha_{1}$ is coercive and continuous, there exists $\delta_{2}>0$ such that $\norm{x}_{2}\geq \delta_{2}$ implies $\alpha(x)> a$. Since $0<\alpha_{1}(\phi_{f}(x,t))<a$ for all $x\in K$ and $t\in[0,\tau_{K}]$, it follows that $\delta_{1}<\norm{\phi_{f}(x,t)}_{2}<\delta_{2}$ for all $x\in K$ and $t\in[0,\tau_{K}]$.

Now, since $f$ is locally Lipschitz continuous, $f$ is uniformly Lipschitz continuous on $U:=\{x\in G\mid \delta_{1}<\norm{x}_{2}<\delta_{2}\}$, and it follows (by e.g. Thm.~3.4 in~\cite{khalil2002nonlinear}) that there exists $L_{1}\geq 0$ such that
	\begin{equation*}
		\norm{\phi_{f}(x_{1},t)-\phi_{f}(x_{2},t)}_{2}\leq L_{1}\norm{x_{1}-x_{2}}_{2}, 	
	\end{equation*}
	for all $x_{1},x_{2}\in K$ and $t\in[0,\tau_{K}]$.

Having established a Lipschitz bound on $\phi_{f}$, Lipschitz continuity of $\alpha_{1}$ implies a similar bound $L_{2}\geq 0$ where
	\begin{equation}\label{eq:proof:Lipschitz2}
		\bl|\alpha_{1}\bl(\phi_{f}(x_{1},t)\br)-\alpha_{1}\bl(\phi_{f}(x_{2},t)\br)\br|\leq L_{2}\norm{x_{1}-x_{2}}_{2},		\tag{$\star$}
	\end{equation}
	for all $x_{1},x_{2}\in K$ and $t\in[0,\tau_{K}]$. As before, $0 \not \in K$ with $K$ compact imply existence of $0<\gamma_1,\gamma_2<a$ such that $\gamma_1 \le \alpha_{1}(\phi_{f}(x,t))\le \gamma_2$ for all $x\in K$ and $t\in[0,\tau_{K}]$


Finally, since $\beta$ is locally Lipschitz continuous, $\beta(y,t)$ is uniformly Lipschitz continuous in $y\in[\gamma_1,\gamma_2]$ for any $t\in[0,\tau_{K}]$, and therefore $\hat{\beta}_{a}(y,-[1-\epsilon]t)$ is uniformly Lipschitz continuous in $y\in[\gamma_1,\gamma_2]$ for $t\in[0,\tau_{K}]$. Combining with~\eqref{eq:proof:Lipschitz2}, we conclude the existence of $L\geq 0$ such that
	\begin{align*}
		&\bl|\hat{\beta}_{a}\bl(\alpha_{1}\bl(\phi_{f}(x_{1},t)\br),-[1-\epsilon]t\br)-\hat{\beta}_{a}\bl(\alpha_{2}\bl(\phi_{f}(x_{2},t)\br),-[1-\epsilon]t\br)\br|\\
		&\qquad\leq L\norm{x_{1}-x_{2}}_{2},
	\end{align*}
	for all $x_{1},x_{2}\in K$ and $t\in[0,\tau_{K}]$.
	It follows that
	\begin{align*}
		&\hat{\beta}_{a}\bl(\alpha_{1}(\phi_{f}(x_{2},t)),-[1-\epsilon]t\br) \\
		&\qquad \leq \hat{\beta}_{a}\bl(\alpha_{1}(\phi_{f}(x_{1},t)),-[1-\epsilon]t\br) +L\norm{x_{1}-x_{2}}_{2}	\\
		&\qquad \leq V_{\epsilon}(x_{1})+L\norm{x_{1}-x_{2}}_{2},
	\end{align*}
for all $x_{1},x_{2}\in K$ and $t\in[0,\tau_{K}]$, and therefore
	\begin{align*}
		V_{\epsilon}(x_{2})&=\sup_{t\in[0,\tau_{K}]}\hat{\beta}_{a}\bl(\alpha_{1}(\phi_{f}(x_{2},t)),-[1-\epsilon]t\br)	\\
		&\leq  V_{\epsilon}(x_{1})+L\norm{x_{1}-x_{2}}_{2}.
	\end{align*}
Since $x_1,x_2$ are arbitrary, we likewise have $V_{\epsilon}(x_{1})\leq  V_{\epsilon}(x_{2})+L\norm{x_{1}-x_{2}}_{2}$. Hence
\[
|V_{\epsilon}(x_{1})-V_{\epsilon}(x_{2})|\le L\norm{x_{1}-x_{2}}_{2},\qquad \forall x_{1},x_{2}\in K.
\]
Since $K\subseteq G\setminus\{0\}$ is an arbitrary compact set, we conclude that $V_{\epsilon}$ is locally Lipschitz continuous on $G\setminus\{0\}$.
\end{proof}

\end{appendices}

\end{document}